\documentclass[12pt,reqno]{amsart} 
\usepackage{amssymb}
\usepackage{mathrsfs}
\usepackage{enumerate}
\usepackage[all]{xy}
\usepackage[usenames,dvipsnames]{color}
\usepackage[colorlinks=true, citecolor=OliveGreen, 
linkcolor=OliveGreen]{hyperref}

\makeatletter
\@namedef{subjclassname@2020}{\textup{2020} Mathematics Subject 
Classification} 
\makeatother

\let\Gamma=\varGamma
\let\Omega=\varOmega
\let\Sigma=\varSigma

\renewenvironment{enumerate}[1][]
{\begin{enumerat}[#1]\setlength{\itemsep}{6pt}}{\end{enumerat}}

\newenvironment{enuma}{\begin{enumerate}[{\rm(a) }]}{\end{enumerate}}
\newenvironment{enumi}{\begin{enumerate}[{\rm(i) }]}{\end{enumerate}}

\renewenvironment{itemize}
{\begin{list}{$\bullet$}{\setlength{\itemsep}{6pt} \itemindent=-3mm 
\leftmargin=8mm }}
{\end{list}}

\definecolor{darkgreen}{rgb}{0,0.5,0}
\definecolor{bluegreen}{rgb}{0,0.2,0.8}
\definecolor{darkred}{rgb}{0.8,0,0}
\definecolor{newercolor}{rgb}{0.2,0,1}
\definecolor{darkyellow}{rgb}{0.7,0.7,0}
\definecolor{darkorange}{rgb}{0.8,0.4,0}

\numberwithin{table}{section}

\newcommand{\boldd}[1]{{\mathversion{bold}\textbf{#1}}}

\newcommand{\bmid}{\mathrel{\big|}}

\newlength{\short}
\newcommand{\4}[1]{\widebar{#1}}
\newcommand{\5}[1]{\widehat{#1}}

\newcommand{\9}[1]{{}^{#1}\!}

\def\pair[#1,#2]{[\hskip-1.5pt[#1,#2]\hskip-1.5pt]}

\SelectTips{cm}{10} \UseTips   

\let\oldcirc=\circ
\renewcommand{\circ}{\mathchoice
    {\mathbin{\scriptstyle\oldcirc}}{\mathbin{\scriptstyle\oldcirc}}
    {\mathbin{\scriptscriptstyle\oldcirc}}
    {\mathbin{\scriptscriptstyle\oldcirc}}}

\def\beq#1\eeq{\begin{equation*}#1\end{equation*}}
\def\beqq#1\eeqq{\begin{equation}#1\end{equation}}

\numberwithin{equation}{section}

\newtheorem{Thm}{Theorem}[section]
\newtheorem{Prop}[Thm]{Proposition}
\newtheorem{Cor}[Thm]{Corollary}
\newtheorem{Lem}[Thm]{Lemma}

\newtheorem{Not}[Thm]{Notation}
\newtheorem{Rmk}[Thm]{Remark}

\newtheorem{Quest}[Thm]{Question}

\newtheorem{Th}{Theorem}

\theoremstyle{definition}
\newtheorem{Defi}[Thm]{Definition}
\newtheorem{Ex}[Thm]{Example}

\newcommand{\widebar}[1]
      {\overset{{\mskip1mu\leaders\hrule height0.4pt\hfill\mskip1mu}}{#1}
      \vphantom{#1}}

\newcounter{let} 
\loop\stepcounter{let}
\expandafter\edef\csname cal\alph{let}\endcsname%
{\noexpand\mathcal{\Alph{let}}}
\ifnum\thelet<26\repeat

\loop\stepcounter{let}
\expandafter\edef\csname scr\alph{let}\endcsname%
{\noexpand\mathscr{\Alph{let}}}
\ifnum\thelet<26\repeat

\newcommand{\tdef}[2][]{\expandafter\newcommand\csname#2\endcsname%
{#1\textup{#2}}}
\tdef{Iso}   \tdef{Aut}    \tdef{Out}    \tdef{Inn}    \tdef{Hom}
\tdef{End}   \tdef{Inj}    \tdef{map}    \tdef{Ker}    \tdef{Ob}
\tdef{Mor}   \tdef{Res}    \tdef{Id}     \tdef{Fr}     \tdef{Spin} 
\tdef{rk}    \tdef{conj}   \tdef{incl}   \tdef{proj}   \tdef{diag} 
\tdef{trf}   \tdef{Sol}    \tdef{Sz}     \tdef{cj}
\tdef{Rep}   \tdef{pr}    \tdef{Inndiag} \tdef{Outdiag}  
\tdef{supp}  \tdef{Isom}   \tdef{ord}    \tdef{Coker}   \tdef{Tr}
\tdef[_]{typ} \tdef[^]{op} \tdef[^]{ab}   \tdef{lcm}  
\tdef{restr}  \tdef{Comp}  \tdef{ev}    \tdef{Stab}
\tdef{srk}    \tdef[_]{red}

\newcommand{\ON}{\textit{O'N}}
\newcommand{\chr}{\textup{char}}

\newcommand{\fdef}[1]{\expandafter\newcommand\csname#1\endcsname%
{\mathfrak{#1}}}
\fdef{X}  \fdef{foc}  \fdef{hyp}  \fdef{Lie} \fdef{Y}

\newcommand{\bbdef}[1]{\expandafter\newcommand%
\csname#1\endcsname{\mathbb{#1}}}
\bbdef{C} \bbdef{F} \bbdef{R} \bbdef{Z} \bbdef{N} \bbdef{Q} \bbdef{K} 
\bbdef{W}

\newcommand{\itdef}[1]{\expandafter\newcommand\csname#1\endcsname%
{\textit{#1}}}
\itdef{PSL}  \itdef{PSU}  \itdef{SL}  \itdef{SU}  \itdef{GL}  \itdef{GU}
\itdef{Sp}   \itdef{PSp} \itdef{PSO}  \itdef{SO}  \itdef{SD}  \itdef{PGU} 
\itdef{PGL}  \itdef{Co}  \itdef{Fi}   \itdef{GO}  \itdef{BDI} \itdef{UT}
\itdef{HS}   \itdef{McL} \itdef{Suz}  \itdef{He}  \itdef{Ly}  \itdef{Ru}

\newcommand{\POmega}{\textit{P}\varOmega}

\newcommand{\sminus}{\smallsetminus}
\newcommand{\lie}[3]{\def\test{#2}\def\tst{G}\ifx\test\tst{{}^{#1}#2_{#3}}
\else{{}^{#1}\!#2_{#3}}\fi}
\renewcommand{\*}{\,\lower6pt\hbox{\Large{\textup{*}}}\,}
\newcommand{\syl}[2]{\textup{Syl}_{#1}(#2)}
\newcommand{\sylp}[1]{\syl{p}{#1}}

\renewcommand{\Im}{\textup{Im}}
\newcommand{\autf}{\Aut_{\calf}}

\newcommand{\outf}{\Out_{\calf}}

\newcommand{\homf}{\Hom_{\calf}}

\newcommand{\defeq}{\overset{\textup{def}}{=}}

\newcommand{\mxfoura}[8]{\left(\begin{smallmatrix}#1&#2&#3&#4\\#5&#6&#7&#8}
\newcommand{\mxfourb}[8]{\\#1&#2&#3&#4\\#5&#6&#7&#8\end{smallmatrix}\right)}

\let\emptyset=\varnothing
\renewcommand{\:}{\colon}
\newcommand{\pcom}{{}^\wedge_p}

\newcommand{\nsg}{\trianglelefteq}
\newcommand{\snsg}{\nsg\,\nsg}
\let\nnsg=\ntrianglelefteq

\let\too=\longrightarrow
\let\xto=\xrightarrow

\let\fromm=\longleftarrow
\renewcommand{\gg}{\mathbb{G}}
\newcommand{\hh}{\mathbb{H}}
\newcommand{\gen}[1]{{\langle}#1{\rangle}}
\newcommand{\Gen}[1]{{\bigl\langle}#1{\bigr\rangle}}

\newcommand{\longleft}[1]{\;{\leftarrow%
\count255=0 \loop \mathrel{\mkern-6mu}%
    \relbar\advance\count255 by1\ifnum\count255<#1\repeat}\;}
\newcommand{\longright}[1]{\;{\count255=0 \loop \relbar\mathrel{\mkern-6mu}%
    \advance\count255 by1\ifnum\count255<#1\repeat\rightarrow}\;}
\newcommand{\Right}[2]{\overset{#2}{\longright#1}}
\newcommand{\RIGHT}[3]{\mathrel{\mathop{\kern0pt\longright#1}
        \limits^{#2}_{#3}}}

\newcommand{\LEFT}[3]{\mathrel{\mathop{\kern0pt\longleft#1}\limits^{#2}_{#3}}
}
\newcommand{\dRIGHT}[3]{\mathrel{%
   \mathop{\vcenter{\baselineskip=0pt\hbox{$\kern0pt\longright#1$}%
   \hbox{$\kern0pt\longright#1$}}}\limits^{#2}_{#3}}}
\newcommand{\LRIGHT}[3]{\mathrel{%
   \mathop{\vcenter{\baselineskip=0pt\hbox{$\kern0pt\longleft#1$}%
   \hbox{$\kern0pt\longright#1$}}}\limits^{#2}_{#3}}}
\newcommand{\RLEFT}[3]{\mathrel{%
   \mathop{\vcenter{\baselineskip=0pt\hbox{$\kern0pt\longright#1$}%
   \hbox{$\kern0pt\longleft#1$}}}\limits^{#2}_{#3}}}
\newcommand{\onto}[1]{\;{\count255=0 \loop \relbar\mathrel{\mkern-6mu}%
    \advance\count255 by1
    \ifnum\count255<#1 \repeat \twoheadrightarrow}\;}

\newcommand{\EE}[1]{\textbf{\textup{E}}_{#1}}
\newcommand{\minsc}[1]{\mathfrak{minsc}(#1)}
\newcommand{\ST}[1]{\text{\boldd{$\textit{ST}_{#1}$}}}
\newcommand{\longline}{\bigskip\hfill\hbox to 8cm{\hrulefill}%
\hfill\bigskip}
\newcommand{\scl}{^{\mathfrak{sc}}}
\newcommand{\hocolim}{\mathop{\textup{hocolim}}}
\newcommand{\holim}{\mathop{\textup{holim}}}
\def\LFS(#1){\textup{LFS($#1$)}} 
\def\LF(#1){\textup{LF($#1$)}} 

\fdef{Fin}

\begin{document}

\title{Realizability of fusion systems by discrete groups}

\author{Carles Broto}
\address{Departament de Matem\`atiques, Edifici Cc, Universitat Aut\`onoma de 
Barcelona, 08193 Cerdanyola del Vall\`es (Barcelona), Spain.}
\email{carles.broto@uab.cat}
\thanks{C. Broto is partially supported by MICINN
grant PID2020-116481GB-100 and AGAUR grant 2021-SGR-01015.}

\author{Ran Levi}
\address{Institute of Mathematics, University of Aberdeen,
Fraser Noble 138, Aberdeen AB24 3UE, U.K.}
\email{r.levi@abdn.ac.uk}
\thanks{}

\author{Bob Oliver}
\address{Universit\'e Sorbonne Paris Nord, LAGA, UMR 7539 du CNRS, 
99, Av. J.-B. Cl\'ement, 93430 Villetaneuse, France.}
\email{bobol@math.univ-paris13.fr}
\thanks{B. Oliver is partially supported by UMR 7539 of the CNRS}

\thanks{All three authors would like to thank the Isaac Newton Institute 
for Mathematical Sciences and the Gaelic College on Isle of Skye for their 
support and hospitality during the programme ``Topology, representation 
theory, and higher structures'', supported by EPSRC grant no. EP/R014604/1.}




\begin{abstract}
For a prime $p$, fusion systems over discrete $p$-toral groups are categories 
that model and generalize the $p$-local structure of Lie groups and certain 
other infinite groups in the same way that fusion systems over finite 
$p$-groups model and generalize the $p$-local structure of finite groups. 
In the finite case, it is natural to say that a fusion system $\calf$ is 
realizable if it is isomorphic to the fusion system of a finite group, but 
it is less clear what realizability should mean in the discrete $p$-toral 
case. 

In this paper, we look at some of the different types of realizability for 
fusion systems over discrete $p$-toral groups, including realizability by 
linear torsion groups and sequential realizability, of which the latter is 
the most general. After showing that fusion systems of compact Lie groups 
are always realized by linear torsion groups (hence sequentially 
realizable), we give some new tools for showing that certain fusion systems 
are not sequentially realizable, and illustrate it with two large families 
of examples. 
\end{abstract}

\maketitle

\bigskip

\section*{Introduction}

For a fixed prime $p$, let $\Z/p^\infty$ denote the union of an 
increasing sequence of finite cyclic $p$-groups $\Z/p^n$. A discrete 
$p$-torus is a group isomorphic to a finite product of copies of 
$\Z/p^\infty$, and a \emph{discrete $p$-toral group} is a group that 
contains a discrete $p$-torus with (finite) $p$-power index. For example, 
if $T$ is a torus in the usual sense (a product of copies of $S^1$), then 
the group of all elements of $p$-power order in $T$ is a discrete $p$-torus 
(hence the name). 

A \emph{fusion system} over a discrete $p$-toral group $S$ is a category 
$\calf$ whose objects are the subgroups of $S$, whose morphisms are 
injective homomorphisms between the subgroups including all morphisms 
induced by conjugation in $S$, and such that $\varphi\:P\to Q$ in $\calf$ 
implies $\varphi^{-1}\:\varphi(P)\to P$ is also in $\calf$. A fusion system 
is \emph{saturated} if it satisfies certain additional conditions listed in 
Definition \ref{d:sfs}. Saturated fusion systems over discrete 
$p$-toral groups arise in different contexts relevant to algebra and 
topology (see, e.g., \cite{BLO6} and \cite{KMS1,KMS2}). In 
\cite{BLO3}, we proved their basic properties and showed how they appear 
naturally in various situations.

When $G$ is a group and $S\le G$ is a discrete $p$-toral subgroup, we let 
$\calf_S(G)$ (the ``fusion system of $G$ with respect to $S$'') be the 
category whose objects are the subgroups of $S$, and whose morphisms are 
those homomorphisms between subgroups induced by conjugation in $G$. This 
is always a fusion system, and it is saturated whenever $G$ is finite and 
$S$ is a Sylow $p$-subgroup of $G$. These examples provided Puig with 
part of his original motivation for defining saturated fusion 
systems. 

When $\calf$ is a fusion system over a finite $p$-group $S$, it is natural 
to say that $\calf$ is realizable if it is isomorphic to $\calf_S(G)$ for 
some finite group $G$ with $S\in\sylp{G}$. In this situation, the 
$p$-completed classifying space $BG\pcom$ ($p$-completed in the sense of 
Bousfield and Kan \cite{BK}) is always homotopy equivalent to the 
classifying space of the fusion system (see \cite[Proposition 1.1]{BLO1} 
and \cite[Definition 1.8]{BLO2}). Many examples have been constructed 
of fusion systems over finite $p$-groups that do not arise in this way, 
such as the systems $\calf_{\textup{Sol}}(q)$ constructed in \cite{LO,LO-corr} 
(essentially the only known examples when $p=2$), and those constructed in 
\cite{RV} and \cite{indp2} for odd primes $p$.

In this paper, we look at the question of what ``realizable'' should mean 
for fusion systems over discrete $p$-toral groups. When we first looked at 
fusion systems over discrete $p$-toral groups, we were motivated in part by 
the example of fusion systems of compact Lie groups, but we also showed 
that such fusion systems arise from classifying spaces of $p$-compact 
groups, and of torsion subgroups of $\GL_n(K)$ when $K$ is a field with 
$\chr(K)\ne p$ --- which we call linear torsion groups
(see Theorems 9.10, 10.7, and 8.10 in \cite{BLO3}). In view of this, 
it became clear that it would be much too restrictive to say that $\calf$ 
is realizable only if it is isomorphic to the fusion system of a compact 
Lie group. The question of what should be the correct concept of 
realizability in this setting (if there is one) remained open. 
This paper aims to address this question. As we shall see, rather than a 
simple answer, there are several forms of realizability, of which the most 
general that we have found is what we call ``sequential realizability''.

A fusion system $\calf$ over $S$ is \emph{sequentially realizable} if it 
is the union of an increasing sequence $\calf_1\le\calf_2\le\cdots$ of 
fusion subsystems over finite subgroups of $S$, each of which is realized 
by a finite group. (Note that we do not assume any relations between 
the groups realizing the $\calf_i$.) We show that fusion systems of compact Lie 
groups and $p$-fusion systems of linear torsion groups in characteristic 
different from $p$ are all sequentially realizable. We will show in a 
later paper that sequentially realizable fusion systems are always 
saturated (we avoid that question in this paper by assuming saturation 
when necessary). 

Quite surprisingly, it turns out that a fusion system can be sequentially 
realizable, and at the same time the union of an increasing sequence of fusion 
subsystems over finite $p$-groups that are not realizable. In Example 
\ref{ex:seq.real-exo}, we construct a fusion system $\calf$ over a discrete 
$p$-toral group $S$, together with an increasing sequence of finite 
saturated fusion subsystems $\calf_1\le\calf_2\le\cdots$ whose union is 
$\calf$, where the $\calf_i$ alternate between being realizable (by finite 
groups) and exotic (not realizable). For this reason, while we often use 
the word ``exotic'' to mean ``not realizable'', the phrase ``sequentially 
exotic'' does not seem appropriate when we mean ``not sequentially 
realizable''. 

As one example, we show in Proposition \ref{p:f.s.union} that a 
saturated fusion system $\calf$ is sequentially realizable whenever 
$\calf\cong\calf_S(G)$ for some locally finite group $G$ and some discrete 
$p$-toral subgroup $S\le G$ that is a Sylow $p$-subgroup of $G$. For 
example, if $G$ is a linear torsion group in characteristic different 
from $p$, then $G$ is locally finite and has Sylow $p$-subgroups that are 
discrete $p$-toral, and the fusion system $\calf_S(G)$ (for $S\in\sylp{G}$) 
is sequentially realizable.

Our first result shows that fusion systems over finite $p$-groups 
that are exotic in the earlier sense are still exotic with respect to these 
new criteria.

\begin{Th}[Theorem \ref{t:fin.real.}] \label{ThA}
Let $\calf$ be a saturated fusion system over a finite $p$-group $S$. If 
$\calf$ is sequentially realizable, or if it is realized by some locally 
finite group $G$ containing $S$ as a maximal $p$-subgroup, then $\calf$ is 
realized by a finite group containing $S$ as a Sylow $p$-subgroup.
\end{Th}

When $G$ is a compact Lie group and $p$ is a prime, we let $\sylp{G}$ be 
the set of all maximal discrete $p$-toral subgroups of $G$. For each 
$S\in\sylp{G}$, there is a maximal torus $T\le G$ such that $S\cap T$ is 
the subgroup of elements of $p$-power order in $T$ and 
$ST/T\in\sylp{N_G(T)/T}$ (see, e.g., \cite[Proposition 9.3]{BLO3}). 
We showed in \cite[Lemma 9.5]{BLO3} that $\calf_S(G)$ is a 
saturated fusion system for each $S\in\sylp{G}$. 

\begin{Th}[Theorem \ref{t:cpt.Lie}] \label{ThB}
Let $G$ be a compact Lie group, and fix a prime $p$ and $S\in\sylp{G}$. 
Then there is a linear torsion group $\Gamma$ in characteristic different 
from $p$ such that $\calf_{S_\Gamma}(\Gamma)\cong\calf_S(G)$ for 
$S_\Gamma\in\sylp\Gamma$. In particular, $\calf_S(G)$ is sequentially 
realizable. 
\end{Th}

In fact, in the situation of Theorem \ref{ThB}, for each prime $q\ne p$, 
there is a group $\Gamma$ that is linear over $\4\F_q$ (hence torsion), 
such that $B\Gamma\pcom\simeq BG\pcom$, and 
$\calf_S(G)\cong\calf_{S_\Gamma}(\Gamma)$ for $S_\Gamma\in\sylp\Gamma$.

The next theorem shows that sequential realizability imposes strong 
restrictions on the structure of fusion systems. It is our main tool for 
proving the nonrealizability of certain fusion systems. Its proof requires 
the classification of finite simple groups.

\begin{Th}[Theorem \ref{t:seq-exotic}] \label{ThC}
Let $\calf$ be a saturated fusion system over an infinite discrete 
$p$-toral group $S$, let $T\nsg S$ be the identity component, and set 
$W=\autf(T)$. Assume
\begin{enumi} 

\item $S>T$ and $C_S(T)=T$; 

\item no subgroup $T\le P<S$ is strongly closed in $\calf$; and 

\item no infinite proper subgroup of $T$ is invariant under the action of 
$O^{p'}(W)$. 

\end{enumi}
Assume also that $\calf$ is sequentially realizable. Then $W$ contains a 
normal subgroup of index prime to $p$ that is isomorphic to of one of the 
groups listed in cases (a)--(e) of Proposition \ref{p:Weyl}.
\end{Th}

Theorem \ref{ThC} is proven as Theorem \ref{t:seq-exotic}. That 
theorem is stated in slightly greater generality, but Theorem 
\ref{ThC} suffices for our applications here.

As a first application of Theorem \ref{ThC}, we determine in Section 
\ref{s:p-cpt} exactly which fusion systems of simple, connected 
$p$-compact groups are sequentially realizable. Then, in Section 
\ref{s:other.ex.}, we consider simple fusion systems over 
nonabelian infinite discrete $p$-toral groups containing a discrete 
$p$-torus with index $p$ (classified in \cite[\S\,5]{indp3}), and determine 
exactly which of them are sequentially realizable. In all of these cases, 
either the fusion system in question is realized by an explicitly given 
linear torsion group, or it fails to be sequentially realizable by Theorem 
\ref{ThC}. 

Another result, whose proof is closely related to that of Theorem \ref{ThC} 
and also depends on the classification of finite simple groups, is the 
following. It is stated more generally in Theorem \ref{t:cn.p-cpt} as a 
result about connected $p$-compact groups.

\begin{Th}[Theorem \ref{t:cn.p-cpt}(b)] \label{ThD}
Fix a compact connected Lie group $G$, a prime $p$, and $S\in\sylp{G}$, and 
assume that $p$ divides the order of the Weyl group of $G$. Then there is 
no linear torsion group in characteristic $0$ whose fusion system with 
respect to a Sylow $p$-subgroup is isomorphic to $\calf_S(G)$. 
\end{Th}

In particular, Theorem \ref{ThD} implies that there is no torsion subgroup 
$\Gamma\le G$ with the same fusion system as $G$. Note that by Theorem 
\ref{t:cpt.cn.Lie}, under the hypotheses of Theorem \ref{ThD}, $\calf_S(G)$ 
is realized by fusion systems of linear torsion groups in every 
characteristic except $p$ and $0$.

After a brief overview in Section \ref{s:background} of some general 
definitions and results about fusion systems, we define sequential 
realizability and give some of its basic properties in Section 
\ref{s:seq.real.}, and then look at the special case of linear torsion 
groups in Section \ref{s:LT-gps}. We then prove in Section \ref{s:cpt.Lie} 
that the fusion system of a compact Lie group is always realized by a 
linear torsion group (Theorem \ref{ThB}). We then show some general results 
in Sections \ref{s:incr.seq.} and \ref{s:fin.p-gp} which are applied in 
Section \ref{s:realiz} to prove Theorem \ref{ThC}. We finish by looking at 
realizability of fusion systems of $p$-compact groups in Section 
\ref{s:p-cpt}, and in Section \ref{s:other.ex.} that of fusion systems over 
discrete $p$-toral groups with a discrete $p$-torus of index $p$. 

\medskip

\textbf{Notation:} Our notation is fairly standard, with a few exceptions. 
By a prime, we always mean a nonzero prime number. 
Composition is always taken from right to left. If $H\le G$ are groups and 
$x,g\in G$, then $\9xH=xHx^{-1}$, $\9xg=xgx^{-1}$, and $c_x^H$ (or $c_x$) 
denotes the conjugation homomorphism $(g\mapsto \9xg)$ from $H$ to $\9xH$. 
If $H_1,H_2,K\le G$, then $\Hom_K(H_1,H_2)$ is the set of all 
$c_x^{H_1}\in\Hom(H_1,H_2)$ for $x\in K$ such that $\9xH_1\le H_2$. Also,
\begin{itemize} 

\item $\Phi(P)$ denotes the Frattini subgroup of a finite $p$-group $P$; 

\item $\Omega_n(P)=\gen{x\in P\,|\,x^{p^n}=1}$ when $P$ is a $p$-group and 
$n\ge1$; 

\item $H\circ K$ denotes a central product of $H$ and $K$; 

\item $H.K$ or $H{:}K$ denotes an arbitrary extension or a split extension 
of $H$ by $K$ (i.e., a group with normal subgroup isomorphic to $H$ and 
quotient isomorphic to $K$); 

\item $E_{p^n}$ (or $p^n$ when part of an extension) denotes an elementary 
abelian $p$-group of rank $n$; 

\item $p^{1+2n}_\pm$ denotes an extraspecial $p$-group of order $p^{1+2n}$; 

\item $\Fin(G)$ denotes the set of finite subgroups of a discrete group 
$G$; 

\item $\Z_p$ and $\Q_p$ denote the rings of $p$-adic integers and $p$-adic 
rationals, respectively; 

\item $\ord_p(n)$ denote the order of $n$ in $(\Z/p)^\times$ if $n$ is 
prime to $p$; and 

\item $v_p(n)$ denote the $p$-adic valuation of an integer $n$.

\end{itemize}

Finally, when $X$ is a space, we let $\Aut(X)$ denote the monoid of self 
homotopy equivalences of $X$, and let $\Out(X)$ be the group of homotopy 
classes of elements in $\Aut(X)$.

\section{Saturated fusion systems over discrete 
\texorpdfstring{$p$-toral}{p-toral} groups} 
\label{s:background}

We begin by recalling some definitions and notation from \cite[Sections 
1--2]{BLO3}, starting with the definition of a discrete $p$-toral group. 

\begin{Defi} \label{d:d.p-toral}
Let $p$ be a prime.
\begin{enuma} 

\item A \emph{discrete $p$-torus} is a group that is isomorphic to 
$(\Z/p^\infty)^r$ for some $r\ge0$, where $\Z/p^\infty\cong\Z[\frac1p]/\Z$ 
is the union of the cyclic $p$-groups $\Z/p^n$. The \emph{rank} of a 
discrete $p$-torus $U\cong(\Z/p^\infty)^r$ is $r=\rk(U)$. 

\item A \emph{discrete $p$-toral group} is a group with a normal subgroup 
of $p$-power index that is a discrete $p$-torus. 

\item The \emph{identity component} of a discrete $p$-toral group $P$ is 
the unique discrete $p$-torus of finite index in $P$; 
equivalently, the intersection of all subgroups of finite index in $P$. 

\item The \emph{order} of a discrete $p$-toral group $P$ with identity 
component $U$ is the pair $|P|=(\rk(U),|P/U|)\in\N^2$, where $\N^2$ is 
ordered lexicographically. 

\end{enuma}
\end{Defi}

Thus if $Q\le P$ is a pair of discrete $p$-toral groups, then 
$|Q|\le|P|$, and $|Q|=|P|$ if and only if $Q=P$. 

We next recall some more terminology. 

\begin{Defi} \label{d:loc.fin.}
\begin{enuma} 

\item A \emph{$p$-group} (for a prime $p$) is a discrete group all of 
whose elements have (finite) $p$-power order.

\item A discrete group $G$ is \emph{locally finite} if every finitely 
generated subgroup of $G$ is finite.

\item When $p$ is a prime and $G$ is any group, a \emph{Sylow 
$p$-subgroup} of $G$ is a $p$-subgroup of $G$ that contains all other 
$p$-subgroups up to conjugacy. We let $\sylp{G}$ be the (possibly empty) 
set of Sylow $p$-subgroups of $G$. 

\item A discrete group $G$ is \emph{artinian} if each descending sequence 
of subgroups of $G$ becomes constant.

\end{enuma}
\end{Defi}

Definitions in the literature of ``Sylow $p$-subgroups'' of infinite 
discrete groups vary slightly (see, e.g., \cite[p. 85]{KW}), but the 
strict criterion given above is the most appropriate for our purposes.

Discrete $p$-toral groups play an important role when working with compact 
Lie groups and $p$-compact groups (see, e.g., \cite[\S\,6]{DW}), and that 
in turn made it natural for us to consider them when constructing fusion 
systems for these objects. Since local finiteness and the descending chain 
condition are used in many of our arguments, the following characterization 
of discrete $p$-toral groups helps to explain their importance.

\begin{Prop}[{\cite[Proposition 1.2]{BLO3}}] \label{p:artin-loc.f.}
A group is discrete $p$-toral if and only if it is a $p$-group, artinian, 
and locally finite.
\end{Prop}

Whenever $P$ is a discrete $p$-toral group with identity component $U\nsg 
P$, there is a finite subgroup $R\in\Fin(P)$ such that $P=RU$. This holds 
since $P$ is locally finite and $P/U$ is finite. 

We will need the following variant on the standard result that an inverse 
limit of a system of finite nonempty sets is nonempty. 


\begin{Lem} \label{l:invlim}
Fix a group $\Gamma$, and let 
$(\dots\xto{~r_4~}\Phi_3\xto{~r_3~}\Phi_2\xto{~r_2~}\Phi_1)$ be an 
inverse system of nonempty $\Gamma$-sets and $\Gamma$-maps such that 
$\Phi_i/\Gamma$ is finite for each $i\ge1$. Then 
$\lim_i(\Phi_i,r_i)\ne\emptyset$. 
\end{Lem}

\begin{proof} Since $\Phi_i/\Gamma$ is finite and nonempty for each $i$, 
the inverse limit $\lim_i(\Phi_i/\Gamma,r_i/\Gamma)$ is nonempty. Choose an 
element $(\Psi_i)_{i\ge1}$ in that inverse limit. Thus for each $i$, 
$\Psi_i\subseteq\Phi_i$ is a $\Gamma$-orbit and $r_i(\Psi_i)=\Psi_{i-1}$. 
Thus $(\dots\xto{r'_4}\Psi_3\xto{r'_3}\Psi_2\xto{r'_2}\Psi_1)$ is an inverse 
system of $\Gamma$-sets where each map $r'_i=r_i|_{\Psi_i}$ is surjective, 
and hence $\lim_i(\Phi_i,r_i)\supseteq\lim_i(\Psi_i,r'_i)\ne\emptyset$. 
\end{proof}

As a first application of Lemma \ref{l:invlim}, we note the following 
striking property of discrete $p$-toral groups, one which allows us to give 
a slightly weaker condition for a $p$-subgroup to be Sylow in Proposition 
\ref{p:Pi<Q}.

\begin{Lem} \label{l:P<Q}
Let $P$ and $Q$ be two discrete $p$-toral groups. If every finite subgroup 
of $P$ is isomorphic to a subgroup of $Q$, then $P$ is isomorphic to a 
subgroup of $Q$. 
\end{Lem}

\begin{proof} Choose an increasing sequence $P_1\le P_2\le P_3\le\cdots$ of 
finite subgroups of $P$ such that $\bigcup_{i=1}^\infty P_i=P$. For each 
$i\ge1$, let $\Inj(P_i,Q)$ be the set of injective homomorphisms from $P_i$ 
to $Q$. These sets form an inverse system $(\Inj(P_i,Q),r_i)$ of sets with 
$\Inn(Q)$-action, where each map $r_i\:\Inj(P_i,Q)\too\Inj(P_{i-1},Q)$ is 
defined by restriction. Each orbit set $\Inj(P_i,Q)/\Inn(Q)$ is finite (see 
\cite[Lemma 1.4(a)]{BLO3}), and is nonempty by assumption. So the inverse 
limit of this system is nonempty by Lemma \ref{l:invlim}. Choose 
$(\varphi_i)_{i\ge1}\in\lim_i(\Inj(P_i,Q),r_i)$; then 
$\bigcup_{i=1}^\infty\varphi_i$ is an injective homomorphism from $P$ to 
$Q$. 
\end{proof}

\begin{Prop} \label{p:Pi<Q}
Fix a prime $p$ and a discrete group $G$. Assume $G$ has Sylow 
$p$-subgroups that are discrete $p$-toral. Then a $p$-subgroup $P\le G$ is 
a Sylow $p$-subgroup if every finite $p$-subgroup of $G$ is conjugate to 
a subgroup of $P$.
\end{Prop}

\begin{proof} Assume $S\in\sylp{G}$ and is discrete $p$-toral. Let $P$ be a 
$p$-subgroup of $G$ that contains every finite $p$-subgroup of $G$ up to 
conjugacy. In particular, every finite subgroup of $S$ is 
isomorphic to a subgroup of $P$, and hence $S$ is isomorphic to a subgroup 
of $P$ by Lemma \ref{l:P<Q}. Then $P$ and $S$ are conjugate in $G$ since 
$S\in\sylp{G}$, and so $P$ is also a Sylow $p$-subgroup.
\end{proof}

The following example helps to explain why we needed to assume that $G$ has 
Sylow $p$-subgroups in Proposition \ref{p:Pi<Q}. It is based on 
\cite[Example 3.3]{KW}. 

\begin{Ex} \label{ex:FqS.S} 
Fix two distinct primes $p$ and $q$ and an infinite discrete $p$-toral 
group $S$. Set $H=\F_qS$, regarded as an abelian $q$-group with action of 
$S$, and set $G=H\rtimes S$. Then $G$ is locally finite, every $p$-subgroup 
of $G$ is isomorphic to a subgroup of $S$, and every finite $p$-subgroup of 
$G$ is conjugate to a subgroup of $S$. In contrast, for each proper 
infinite subgroup $T<S$, there is a maximal $p$-subgroup $P\le G$ such that 
$PH=TH$, and $P$ is not conjugate to a subgroup of $S$. In particular, $G$ 
has no Sylow $p$-subgroups unless $S\cong\Z/p^\infty$. 
\end{Ex}

\begin{proof} The group $G$ is locally finite since it is an 
extension of one locally finite group by another. If $P\le G$ is a 
$p$-subgroup, then $P\cap H=1$, so $P\cong PH/H\le G/H\cong S$. 

It remains to prove the claims involving conjugacy between $p$-subgroups of 
$G$. When doing this, it is convenient to consider the groups 
$\5H=\map(S,\F_q)$ and $\5G=\5H\rtimes S$, where $S$ acts on $\5H$ by 
setting
	\[ g(\xi)(h) = \xi(g^{-1}h) \quad \textup{for all $g,h\in S$ and 
	$\xi\:S\too\F_q$.} \]
We consider $H=\F_qS$ as a 
subgroup of $\5H$: the subgroup of those $\xi\:S\too\F_p$ in $\5H$ with 
finite support. In this way, we have $G\le\5G$. 

Now, $\5H\cong\Hom_{\Z}(\Z S,\F_q)=\textup{Coind}_1^S(\F_q)$ in the notation 
of \cite[\S\,III.5]{Brown}. So $H^1(S;\5H)\cong H^1(1;\F_q)=0$ by 
Shapiro's lemma (see \cite[Proposition III.6.2]{Brown}), and a similar argument 
shows that $H^1(T;\5H)=0$ for all $T\le S$. So every $p$-subgroup $P\le\5G$ 
such that $P\5H=T\5H$ is conjugate to $T$ by an element of $\5H$ (see 
\cite[Proposition IV.2.3]{Brown}).

Assume $P\le G$ is a finite $p$-subgroup. Then $PH=UH$ for some finite 
subgroup $U\le S$, and so $P=\9\xi U$ for some $\xi\in\5H$. Hence for each $g\in 
U$, we have $[g,\xi]\in H$, and so $g(\xi)-\xi\in\map(S,\F_q)$ has 
finite support. Since $U$ is finite, this means that $\xi$ is constant on 
all but finitely many cosets of $U$, and hence that $\xi\in H+C_{\5H}(U)$. 
So $P=\9\xi U$ is conjugate to $U$ by an element of $H$. Since every finite 
$p$-subgroup of $G$ has this form (for some finite $U\le S$), this proves 
that every finite $p$-subgroup is conjugate to a subgroup of $S$.

Now let $T<S$ be a proper infinite subgroup. We will construct 
$\xi\in\5H$ such that $\9\xi T$ is a maximal $p$-subgroup in $G$ and hence 
not conjugate in $G$ to a subgroup of $S$. To do this, let 
$1=T_0<T_1<T_2<T_3<\cdots$ be a strictly increasing sequence of finite 
subgroups of $T$ such that $T=\bigcup_{i=1}^\infty T_i$. For each $i\ge1$, 
choose $g_i\in T_i\sminus T_{i-1}$ and set $X_i=T_{i-1}g_i$. Define 
$\xi\in\5H=\map(S,\F_q)$ by setting $\xi(g)=1$ if $g\in\bigcup_{i=1}^\infty 
X_i$ and $\xi(g)=0$ otherwise. Set $P=\9\xi T$. 

We first check that $P\le G$ by showing that $[T,\xi]\le H$. To see this, 
fix $g\in T$, and let $i$ be such that $g\in T_i$. Then $\xi$ is constant 
on left $g$-orbits in $T\sminus T_i$, so $g(\xi)-\xi$ has support contained 
in the finite subgroup $T_i$, and hence lies in $H$. Thus $[g,\xi]\in H$, 
and since $g\in T$ was arbitrary, we get $[T,\xi]\le H$. 

It remains to prove that $P$ is a maximal $p$-subgroup of $G$. Assume 
otherwise: then there are $U\le S$ and $\eta\in\5H$ such that $U>T$ and 
$\9\xi T<\9\eta U\le G$. Thus $[\eta,U]\le H$ and $[\eta-\xi,T]=1$, and in 
particular, $\eta-\xi$ is constant on cosets of $T$. Fix an element $g\in 
U\sminus T$; then by construction, $g(\xi)-\xi$ is nonzero on infinitely 
many elements of $T$ and zero on infinitely many elements of $T$. Since 
$g(\eta-\xi)-(\eta-\xi)$ is constant on $T$, the element $g(\eta)-\eta$ has 
infinite support, contradicting the assumption that $[\eta,U]\le H$. We 
conclude that $P$ is maximal. 

Thus $G$ can have Sylow $p$-subgroups only if $S$ has no proper 
infinite subgroups; i.e., only if $S\cong\Z/p^\infty$.
\end{proof}

The following well known property of finite $p$-groups will also be needed 
for discrete $p$-toral groups. 

\begin{Lem} \label{l:P^Z(S)}
Let $S$ be a discrete $p$-toral group, and let $1\ne P\nsg S$ be a nontrivial 
normal subgroup. Then $P\cap\Omega_1(Z(S))\ne1$. 
\end{Lem}

\begin{proof} If $P$ is finite, then $P\cap Z(S)$ is the fixed subgroup of 
the action of the finite $p$-group $S/C_S(P)$ on $P$, and hence is 
nontrivial. So $P\cap\Omega_1(Z(S))=\Omega_1(P\cap Z(S))\ne1$.

If $P$ is infinite, let $U\nsg P$ be its identity component. Then 
$\Omega_1(U)\nsg S$, and so $P\cap\Omega_1(Z(S))\ge 
\Omega_1(U)\cap\Omega_1(Z(S))\ne1$ since $\Omega_1(U)$ is finite.
\end{proof}

We next recall some more definitions. 

\begin{Defi} 
A \emph{fusion system} $\calf$ over a discrete $p$-toral group $S$ is a 
category whose objects are the subgroups of $S$, where 
	\[ \Hom_S(P,Q) \subseteq \homf(P,Q) \subseteq \Inj(P,Q) \]
for each $P,Q\le S$, and such that $\varphi\in\homf(P,Q)$ implies 
$\varphi^{-1}\in\homf(\varphi(P),P)$. Here, $\Inj(P,Q)$ is the set of 
injective homomorphisms from $P$ to $Q$. 
\end{Defi}

When $\calf$ is a fusion system over $S$, then for $P\le S$ and $x\in S$ we 
set
	\[ P^\calf=\{\varphi(P)\,|\,\varphi\in\homf(P,S)\} 
	\qquad\textup{and}\qquad 
	x^\calf=\{\varphi(x)\,|\,\varphi\in\homf(\gen{x},S)\}: \]
the \emph{$\calf$-conjugacy classes} of $P$ and $x$. We also, for 
each $P\le S$, write 
	\[ \autf(P)=\homf(P,P) \qquad\textup{and}\qquad 
	\outf(P)=\autf(P)/\Inn(P), \]
and refer to $\autf(P)$ as the \emph{automizer} of $P$ in $\calf$.

\begin{Defi} \label{d:sfs}
Let $\calf$ be a fusion system over a discrete $p$-toral group $S$.
\begin{enuma}

\item A subgroup $P\le{}S$ is \emph{fully centralized in $\calf$} if
$|C_S(P)|\ge|C_S(P^*)|$ for all $P^*\in P^\calf$. 

\item A subgroup $P\le{}S$ is \emph{fully normalized in $\calf$} if
$|N_S(P)|\ge|N_S(P^*)|$ for all $P^*\in P^\calf$. 

\item A subgroup $P\le S$ is \emph{fully automized in $\calf$} if 
$\outf(P)=\autf(P)/\Inn(P)$ is finite and $\Out_S(P)\in\sylp{\outf(P)}$.

\item A subgroup $P\le S$ is \emph{receptive in $\calf$} if for each $Q\in 
P^\calf$ and each $\varphi\in\homf(Q,P)$, if we set
	\[ N_\varphi = \{ g\in{}N_S(P) \,|\, \varphi c_g\varphi^{-1} \in 
	\Aut_S(\varphi(P)) \}, \]
then there is $\widebar{\varphi}\in\homf(N_\varphi,S)$ such that
$\widebar{\varphi}|_P=\varphi$.

\item The fusion system $\calf$ is \emph{saturated} if the following
three conditions hold:
\begin{itemize} \smallskip

\item (Sylow axiom) Each subgroup $P\le{}S$ fully normalized in 
$\calf$ is also fully automized and fully centralized in $\calf$.

\item (Extension axiom) Each subgroup $P\le{}S$ fully centralized in 
$\calf$ is also receptive in $\calf$.

\item (Continuity axiom) If $P\le S$, and $\varphi\in\Hom(P,S)$ is an 
injective homomorphism such that $\varphi|_R\in\homf(R,S)$ for each 
$R\in\Fin(P)$, then $\varphi\in\homf(P,S)$.

\end{itemize}
\end{enuma}
\end{Defi}

When $\calf$ is a fusion system over a finite $p$-group $S$, it follows 
directly from the definition that every subgroup of $S$ is 
$\calf$-conjugate to a fully normalized and a fully centralized subgroup. 
When $\calf$ is a fusion system over an infinite discrete $p$-toral group 
$S$, then this is still true, and is a consequence of \cite[Lemma 
1.6]{BLO3}.

\begin{Defi} \label{d:FS(G)}
When $G$ is a discrete group and $S\le G$ is a discrete $p$-toral subgroup, 
let $\calf_S(G)$ be the fusion system over $S$ where for each $P,Q\le S$, 
	\[ \Hom_{\calf_S(G)}(P,Q) = \{c_g \,|\, g\in G,~ \9gP\le Q \}. \]
\end{Defi}

The following are some of the basic definitions for subgroups in a fusion 
system. Recall that for a finite group $G$, a proper subgroup $H<G$ 
is \emph{strongly $p$-embedded} if $p\mid|H|$, and for each $x\in G\sminus 
H$ we have $p\nmid|H\cap\9xH|$. 

\begin{Defi} \label{d:subgroups}
Let $\calf$ be a fusion system over a discrete $p$-toral group $S$. For a 
subgroup $P\le S$, 
\begin{itemize} 

\item $P$ is \emph{$\calf$-centric} if $C_S(Q)\le Q$ for each $Q\in 
P^\calf$;

\item $P$ is \emph{$\calf$-radical} if $O_p(\outf(P))=1$; 

\item $P$ is \emph{$\calf$-essential} if it is $\calf$-centric and fully 
normalized in $\calf$, and $\outf(P)$ contains a strongly $p$-embedded 
subgroup; 

\item $P$ is \emph{weakly closed in $\calf$} if $P^\calf=\{P\}$; and 

\item $P$ is \emph{strongly closed in $\calf$} if $x^\calf\subseteq P$ for 
each $x\in P$. 

\end{itemize}
We let $\calf^{rc}\subseteq\calf^c$ denote the sets of subgroups of $S$ that 
are $\calf$-centric and $\calf$-radical, or $\calf$-centric, respectively. 
\end{Defi}

The next proposition gives some of the basic finiteness properties of 
fusion systems in this context.

\begin{Prop}[{\cite[Lemma 2.5 and Corollary 3.5]{BLO3}}] 
\label{p:Frc-finite}
Let $\calf$ be a saturated fusion system over a discrete $p$-toral group 
$S$. Then 
\begin{enuma} 

\item $\homf(P,Q)/\Inn(Q)$ is finite for each $P,Q\le S$; and 

\item $\calf^{rc}$ is the union of finitely many $S$-conjugacy classes of 
subgroups. 

\end{enuma}
\end{Prop}

\begin{Lem} \label{l:NSP->NSQ}
Let $\calf$ be a saturated fusion system over a discrete $p$-toral group 
$S$, and assume $P\le S$ is fully normalized in $\calf$. Then for each 
$Q\in P^\calf$, there is $\varphi\in\homf(N_S(Q),N_S(P))$ such that 
$\varphi(Q)=P$.
\end{Lem}

\begin{proof} By definition, every fully normalized subgroup 
is also fully automized and receptive. So the lemma follows from 
\cite[Lemma 1.7(c)]{BLO6}. 
\end{proof}

We will be using the following version of Alperin's fusion theorem for 
fusion systems over discrete $p$-toral groups.

\begin{Thm}[{\cite[Theorem 3.6]{BLO3}}] \label{t:AFT}
Let $\calf$ be a saturated fusion system over a discrete $p$-toral group 
$S$. Then each morphism in $\calf$ is a composite of restrictions of 
elements in $\autf(Q)$ for fully normalized subgroups $Q\in\calf^{rc}$. 
\end{Thm}

We now turn our attention to quotient fusion systems. 

\begin{Defi} \label{d:F/Q}
Let $\calf$ be a fusion system over a discrete $p$-toral group 
$S$, and assume $Q\le S$ is weakly closed in $\calf$. In particular, $Q\nsg 
S$. Let $\calf/Q$ be the fusion system over $S/Q$ defined by setting, for 
each $P,R\le S$ containing $Q$, 
	\[ \Hom_{\calf/Q}(P/Q,R/Q) = \{\varphi/Q \,|\, \varphi\in\homf(P,R) 
	\}; \]
where $\varphi/Q\in\Hom(P/Q,R/Q)$ sends $xQ$ to $\varphi(x)Q$ for all 
$x\in P$.
\end{Defi}

\begin{Lem} \label{l:F(G)/Q}
Let $\calf$ be a fusion system over a discrete $p$-toral group $S$, and 
assume $Q\le S$ is weakly closed in $\calf$. If $\calf=\calf_S(G)$ for some 
discrete group $G$ with $S\le G$, then $\calf/Q=\calf_{S/Q}(N_G(Q)/Q)$.
\end{Lem}

\begin{proof} The inclusion $\calf_{S/Q}(N_G(Q)/Q)\le \calf_S(G)/Q=\calf/Q$ 
is clear. Conversely, for each $P,R\le S$ containing $Q$ and each $g\in G$ 
such that $\9gP\le R$, since $c_g\in\homf(P,R)$ and $Q$ is weakly closed in 
$\calf$, we have $g\in N_G(Q)$ and hence 
$c_{gQ}\in\Hom_{\calf/Q}(P/Q,R/Q)$. So $\calf/Q=\calf_{S/Q}(N_G(Q)/Q)$.
\end{proof}

The proof that quotient systems of saturated fusion systems are again 
saturated is also elementary. 


\begin{Lem} \label{l:F/Q}
Let $\calf$ be a saturated fusion system over a discrete $p$-toral group 
$S$, and assume $Q\le S$ is weakly closed in $\calf$. Then $\calf/Q$ is a 
saturated fusion system over $S/Q$. 
\end{Lem}

\begin{proof} By \cite[Corollary 1.8]{BLO6}, it suffices to show that 
\begin{enumi} 

\item every subgroup of $S/Q$ is $\calf/Q$-conjugate to one that is fully 
automized and receptive; and 

\item the continuity axiom holds for $\calf/Q$. 

\end{enumi}

\noindent\textbf{(i) } (The following argument is based on the proof of 
\cite[Lemma II.5.4]{AKO}.) Fix a subgroup $P/Q\le S/Q$, and choose $R\in 
P^\calf$ such that $R$ is fully normalized in $\calf$. Then $R$ is fully 
automized and receptive in $\calf$ since $\calf$ is saturated. Also, $R\ge 
Q$ and $R/Q\in(P/Q)^{\calf/Q}$, so it will suffice to prove that $R/Q$ is 
fully automized and receptive in $\calf/Q$.

For each $U,V\le S$ containing $Q$, let 
	\[ \Psi_{U,V}\: \homf(U,V) \Right4{} \Hom_{\calf/Q}(U/Q,V/Q) \]
be the natural map that sends $\varphi$ to $\varphi/Q$. By definition of 
$\calf/Q$, $\Psi_{U,V}$ is surjective for all $U$ and $V$. Also, 
$\Psi_{U,V}(\Hom_S(U,V))=\Hom_{S/Q}(U/Q,V/Q)$ in all cases. 
Since $\Aut_S(R)\in\sylp{\autf(R)}$ (recall $R$ is assumed fully 
automized), we have $\Aut_{S/Q}(R/Q)\in\sylp{\Aut_{\calf/Q}(R/Q)}$, 
and hence $R/Q$ is fully automized in $\calf/Q$.

To see that $R/Q$ is receptive in $\calf/Q$, fix an isomorphism 
$\5\varphi\in\Iso_{\calf/Q}(U/Q,R/Q)$, and choose 
$\varphi\in\Psi_{U,R}^{-1}(\5\varphi)\subseteq\Iso_\calf(U,R)$. Let 
$N_\varphi\le N_S(U)$ and $N_{\5\varphi}\le N_{S/Q}(U/Q)$ be as in 
Definition \ref{d:sfs}(d). For each $g\in N_\varphi$, we have $\varphi 
c_g\varphi^{-1}\in\Aut_S(R)$, and hence $\5\varphi 
c_{gQ}\5\varphi{}^{-1}\in\Psi_{R,R}(\Aut_S(R))=\Aut_{S/Q}(R/Q)$, proving 
that $gQ\in N_{\5\varphi}$ and hence that $N_\varphi/Q\le N_{\5\varphi}$. 

Let $N\le N_S(U)$ be such that $N/Q=N_{\5\varphi}$. Then 
	\[ \Psi_{R,R}(\varphi\Aut_N(U)\varphi^{-1}) = 
	\5\varphi\Aut_{N_{\5\varphi}}(U/Q)\5\varphi{}^{-1} \le 
	\Aut_{S/Q}(R/Q) = \Psi_{R,R}(\Aut_S(R)). \]
Since $R$ is fully automized, $\Aut_S(R)\in\sylp{\autf(R)}$, and hence 
$\Aut_S(R)\in\sylp{\Aut_S(R)\Ker(\Psi_{R,R})}$. So there is 
$\psi\in\Ker(\Psi_{R,R})$ such that $\9{\psi\varphi}\Aut_N(U)\le\Aut_S(R)$. 
Upon replacing $\varphi$ by $\psi\varphi$, we get 
$N_{\5\varphi}=N_\varphi/Q$, and still have $\5\varphi=\varphi/Q$. 

Since $R$ is receptive in $\calf$, the isomorphism $\varphi$ extends to 
$\4\varphi\in\homf(N_\varphi,S)$. So 
$\4\varphi/Q\in\Hom_{\calf/Q}(N_{\5\varphi},S/Q)$ is an extension of 
$\5\varphi$, and we conclude that $R/Q$ is receptive in $\calf/Q$.

\smallskip

\noindent\textbf{(ii) } Assume $P\ge Q$ and $\5\varphi\in\Hom(P/Q,S/Q)$ are 
such that $\5\varphi|_{R/Q}\in\Hom_{\calf/Q}(R/Q,S/Q)$ for each 
$R/Q\in\Fin(P/Q)$. Choose $Q\le P_1\le P_2\le\dots\le P$ such that 
$|P_i/Q|<\infty$ for each $i$ and $P=\bigcup_{i=1}^\infty P_i$. For each 
$i$, set 
	\[ \Phi_i = \{ \psi\in\homf(P_i,S) \,|\, \psi/Q=\5\varphi|_{P_i/Q} 
	\} . \] 
Thus $(\Phi_i)_{i\ge1}$ is an inverse system of sets via 
restriction of morphisms, and $\Phi_i\ne\emptyset$ for each $i$ since 
$\5\varphi|_{P_i/Q}\in\Mor(\calf/Q)$. We claim that 
$\lim_i(\Phi_i)\ne\emptyset$. 

Let $\Gamma\le S$ be such that $Q\nsg\Gamma$ and 
$\Gamma/Q=C_{S/Q}(\5\varphi(P/Q))$. Since $S/Q$ is artinian and the 
centralizers $C_{S/Q}(\5\varphi(P_i/Q))$ form a descending sequence 
intersecting in $\Gamma/Q$, there is $m\ge1$ such that 
$\Gamma/Q=C_{S/Q}(\5\varphi(P_i/Q))$ for each $i\ge m$. Also, for each 
$i\ge1$, and each $\psi\in\Phi_i$ and $\gamma\in\Gamma$, we have 
$c_\gamma\circ\psi\in\Phi_i$ since $[\gamma,\psi(P_i)]\le Q$ (since $\gamma 
Q\in C_{S/Q}(\5\varphi(P_i/Q))$). Thus $\gamma\in\Gamma$ acts on $\Phi_i$ 
via composition with $c_\gamma$. 

Assume $i\ge m$, $\psi\in\Phi_i$, and $x\in S$ are such that 
$c_x\psi\in\Phi_i$. Then $[x,\psi(P_i)]\le Q$, so 
$[xQ,\5\varphi(P_i/Q)]=1$, and hence $x\in\Gamma$. Thus two elements of 
$\Phi_i$ in the same $\Inn(S)$-orbit are in the same $\Gamma$-orbit. So the 
natural map $\Phi_i/\Gamma\too\homf(P_i,S)/\Inn(S)$ is injective, and hence 
$\Phi_i/\Gamma$ is finite since $\homf(P_i,S)/\Inn(S)$ is finite by 
Proposition \ref{p:Frc-finite}(a). 

We thus have an inverse system $(\Phi_i)_{i\ge m}$ of nonempty 
$\Gamma$-sets with finite orbit sets, and hence 
$\lim_i(\Phi_i)\ne\emptyset$ by Lemma \ref{l:invlim}. Choose 
$(\psi_i)_{i\ge m}\in\lim_i(\Phi_i)$. So  
$\psi_i|_{P_{i-1}}=\psi_{i-1}$ for each $i>m$, and we can define 
$\psi=\bigcup_{i=m}^\infty\psi_i\in\Hom(P,S)$. Then $\psi\in\homf(P,S)$ by 
the continuity axiom for $\calf$, and 
$\5\varphi=\psi/Q\in\Hom_{\calf/Q}(P/Q,S/Q)$ since 
$(\psi/Q)_{P_i/Q}=\psi_i/Q=\5\varphi|_{P_i/Q}$ for each $i\ge m$.
\end{proof}

We also need to work with isomorphisms between fusion systems. Recall, in 
the following definition, that we write composition from right to left.

\begin{Defi} \label{d:F1=F2}
Let $\calf_i$ be a fusion system over the discrete $p$-toral group $S_i$ 
for $i=1,2$. An \emph{isomorphism} $(\rho,\5\rho)\:\calf_1\too\calf_2$ 
consists of an isomorphism of groups $\rho\:S_1\xto{~\cong~}S_2$ and an 
isomorphism of categories $\5\rho\:\calf_1\xto{~\cong~}\calf_2$ such that 
$\5\rho$ sends an object $P$ to $\rho(P)$, and sends a morphism 
$\varphi\in\Hom_{\calf_1}(P,Q)$ to the morphism $\5\rho(\varphi)\in 
\Hom_{\calf_2}(\rho(P),\rho(Q))$ such that 
	\[ \5\rho(\varphi)\circ(\rho|_P)= (\rho|_Q)\circ\varphi \in 
	\Hom(P,\rho(Q)). \]
\end{Defi}

Note that in the above definition, the functor $\5\rho$ is uniquely 
determined by the isomorphism $\rho\:S_1\too S_2$. So in practice, we 
regard an isomorphism of fusion systems as an isomorphism between the Sylow 
groups that satisfies the extra conditions needed for there to exist an 
isomorphism of categories.

When $\calf$ is a saturated fusion system over a discrete $p$-toral group 
$S$, a \emph{centric linking system} associated to $\calf$ is a category 
$\call$ whose objects are the $\calf$-centric subgroups of $S$, together 
with a functor $\pi\:\call\too\calf$ that is the inclusion on objects and 
is surjective on morphism sets and satisfies certain additional conditions. 
The following definition describes one way to construct centric linking 
systems under certain conditions on a discrete group $G$ and a 
discrete $p$-toral subgroup $S\le G$. 

\begin{Defi} \label{d:LSc(G)}
Let $G$ be a discrete group. 
\begin{enuma} 

\item A $p$-subgroup $P\le G$ is \emph{$p$-centric in $G$} if $Z(P)$ is the 
unique Sylow $p$-subgroup of $C_G(P)$ (equivalently, $C_G(P)/Z(P)$ has no 
elements of order $p$). 

\item If $G$ is locally finite and $S\le G$ is a discrete $p$-toral 
subgroup, let $\call_S^c(G)$ be the category whose objects are the 
subgroups of $S$ that are $p$-centric in $G$, and where for each pair of 
objects $P,Q\le S$ we set
	\[ \Mor_{\call_S^c(G)}(P,Q) = \{g\in G \,|\, \9gP\le Q \} \big/ 
	O^p(C_G(P)). \]
Here, $O^p(C_G(P))$ means the subgroup generated by all elements of order 
prime to $p$ in $C_G(P)$.

\end{enuma}
\end{Defi}

Since linking systems play a relatively minor role throughout most of this 
paper (mostly mentioned as ``additional information''), we don't give the 
precise definition here, but instead refer to \cite[Definition 4.1]{BLO3}. 
The exception to this is when we look at fusion systems of compact Lie 
groups (Section \ref{s:cpt.Lie}) and those of $p$-compact groups (Section 
\ref{s:p-cpt}). Linking systems do play an important role in those two 
sections, and we list there the precise properties of linking systems that 
we need to use.

When $\call$ is a centric linking system associated to a fusion system 
$\calf$, we let $|\call|$ denote its geometric realization (see, e.g., 
\cite[\S\,III.2.2]{AKO}). We regard $|\call|\pcom$ as a ``classifying 
space'' for $\calf$, where $(-)\pcom$ denotes $p$-completion in the sense 
of Bousfield and Kan (see \cite[\S\,III.1.4]{AKO} for a brief summary of 
some of its elementary properties). By \cite[Theorem 7.4]{BLO3}, two 
saturated fusion systems are isomorphic if their classifying spaces are 
homotopy equivalent, and this is the basis for showing in Section 
\ref{s:cpt.Lie} that fusion systems of compact Lie groups can all be 
realized by certain discrete linear groups.


\section{Sequential realizability}
\label{s:seq.real.}

When working with fusion systems over finite $p$-groups, it is natural to 
say that a fusion system is realizable if it is isomorphic to the fusion 
system of a finite group, and is exotic otherwise. When we turn to fusion 
systems over discrete $p$-toral groups, it is less clear what is the most 
natural condition for realizability. In this section, we define 
sequential realizability and look at its basic properties. For example, we 
show in Corollary \ref{c:f.s.union} that the fusion system of a countable 
locally finite group with Sylow $p$-subgroups is sequentially realizable if 
it is saturated.

No fusion system over an infinite discrete $p$-toral group $S$ can be the 
union \emph{as categories} of fusion subsystems over finite subgroups of 
$S$, since $S$ itself can't be an object in any such union. So before we 
define sequential realizability, we need to make precise what we mean by 
an infinite union of fusion systems. The following definition can be 
thought of as a simplified version of the definition in \cite[Definition 
3.1]{Gonzalez} of a finite approximation of linking systems. 

\begin{Defi} \label{d:f.s.union}
Assume $\calf_1\le\calf_2\le\calf_3\le\cdots$ is an increasing sequence of
saturated fusion systems over finite subgroups $S_1\le 
S_2\le S_3\le\cdots$ of a discrete $p$-toral group $S=\bigcup_{i=1}^\infty 
S_i$. Define $\bigcup_{i=1}^\infty\calf_i$ to be the fusion system over $S$ 
where for each $P,Q\le S$, 
	\begin{multline*} 
	\Hom_{\bigcup_{i=1}^\infty\calf_i}(P,Q) = 
	\bigl\{ \varphi\in\Hom(P,Q) \bmid \forall\, R\in\Fin(P)~ 
	\exists\,i\ge1~ \\ \textup{such that $R\le S_i$ and 
	$\varphi|_R\in\Hom_{\calf_i}(R,Q\cap S_i)$} \bigr\}. 
	\end{multline*}
\end{Defi}

In other words, we define the union of an increasing sequence of fusion 
systems $\calf_i$ over $S_1\le S_2\le\dots$ to be the smallest fusion 
system over $\bigcup_{i=1}^\infty S_i$ that contains the $\calf_i$ and 
satisfies the continuity axiom (see Definition \ref{d:sfs}(e)).

\begin{Defi} \label{d:seq.real.}
A fusion system $\calf$ over a discrete $p$-toral group $S$ 
is \emph{sequentially realizable} if there is an increasing sequence 
$\calf_1\le\calf_2\le\dots$ of saturated fusion subsystems of $\calf$ over 
finite subgroups $S_1\le S_2\le\dots$ of $S$, such that 
$S=\bigcup_{i=1}^\infty S_i$ and $\calf=\bigcup_{i=1}^\infty\calf_i$, and 
such that each $\calf_i$ is realizable by a finite group.
\end{Defi}

Note that we do not assume in the definition that $\calf$ is saturated. We 
will show in a later paper that sequentially realizable fusion systems over 
discrete $p$-toral groups are always saturated. In this paper, when we work 
with sequentially realizable fusion systems, we always say explicitly 
whether or not we assume that they are saturated.

In all examples where we can prove that a fusion system is sequentially 
realizable, we do so by showing that it is realized by a linear torsion 
group, as defined in the next section. 

We next look at fusion systems over finite $p$-groups.

\begin{Thm} \label{t:fin.real.}
Let $\calf$ be a fusion system over a finite $p$-group $S$. 
Assume that either 
\begin{enumi} 

\item $\calf$ is sequentially realizable; or 

\item $\calf\cong\calf_S(G)$ for some locally finite group $G$ that contains 
$S$ as a maximal $p$-subgroup.

\end{enumi}
Then $\calf\cong\calf_S(G_0)$ for some finite subgroup $G_0\le G$ that 
contains $S$ as a Sylow $p$-subgroup.
\end{Thm}

\begin{proof} \textbf{(i) } If $\calf$ is sequentially realizable, then 
$\calf=\bigcup_{i=1}^\infty\calf_i$ for some increasing sequence 
$\calf_1\le\calf_2\le\cdots$ of finite subsystems, each of which is 
realized by a finite group. But since $\calf$ is finite, we have 
$\calf=\calf_i$ for some $i$, and hence it is realized by a finite group.

\smallskip

\noindent\textbf{(ii) } Assume $\calf=\calf_S(G)$ for some locally finite 
group $G$ which contains $S$ as a maximal $p$-subgroup. Let 
$\rho\:\Mor(\calf)\too G$ be a map of sets that sends a morphism 
$\varphi\in\homf(P,Q)$ to an element $\rho(\varphi)$ such that 
$\varphi=c_{\rho(\varphi)}|_P$. Set $G_0=\gen{S,\Im(\rho)}$. Then $G_0$ is 
finitely generated since $\Mor(\calf)$ is finite, and is finite since $G$ 
is locally finite. Also, $S\in\sylp{G_0}$ since it is a maximal 
$p$-subgroup of $G$. 

By construction, $\calf\le\calf_S(G_0)\le\calf_S(G)=\calf$. So 
$\calf$ is realized by $G_0$. 
\end{proof}

The next proposition shows that fusion systems of certain locally finite 
groups are sequentially realizable. By ``countable'' we always mean ``at 
most countable'' (i.e., possibly finite).

\begin{Prop} \label{p:f.s.union} 
Let $\calf$ be a saturated fusion system over a discrete $p$-toral 
group $S$. Assume $G$ is a locally finite group containing $S$ such 
that $\calf=\calf_S(G)$. Then the following hold.
\begin{enuma} 

\item There is a countable subgroup $G_*\le G$ such that $S\le G_*$ and 
$\calf=\calf_S(G_*)$.

\item If $G$ is countable and $S\in\sylp{G}$, then there is an increasing 
sequence of finite subgroups $\{G_i\}_{i\ge1}$ of $G$ such that $S_i\defeq 
S\cap G_i \in \sylp{G_i}$ for each $i$, and 
$G=\bigcup\nolimits_{i=1}^\infty G_i$. For each such sequence, 
\begin{enumerate}[\rm(b.1) ] \medskip

\item $S=\bigcup_{i=1}^\infty S_i$, and 

\item $\calf=\bigcup_{i=1}^\infty\calf_{S_i}(G_i)$ in the sense of 
Definition \ref{d:f.s.union}.

\end{enumerate}
\end{enuma}
\end{Prop}

\begin{proof} \textbf{(a) } By Proposition \ref{p:Frc-finite} and since 
$\calf$ is saturated, there are finitely many $S$-conjugacy classes of 
subgroups in $\calf^{rc}$, and $\homf(P,Q)/\Inn(Q)$ is finite for each 
$P,Q\le S$. Since $S$ is countable, this implies that the object set 
$\calf^{rc}$ is countable, and that $\homf(P,Q)$ is countable for each 
$P,Q\le S$. Hence the category $\calf^{rc}$ has countably many morphisms. 

Choose a map of sets $\rho\:\Mor(\calf^{rc})\too G$ that sends a morphism 
$\varphi\in\homf(P,Q)$ (for $P,Q\in\calf^{rc}$) to an element 
$\rho(\varphi)$ such that $\varphi=c_{\rho(\varphi)}|_P$. Set 
$G_*=\gen{S\cup\Im(\rho)}$. Thus $G_*$ is countably generated since $S$ and 
$\Mor(\calf^{rc})$ are countable, and hence is countable. Also, 
$\calf\le\calf_S(G_*)\le\calf_S(G)$ where the first inclusion holds by 
Theorem \ref{t:AFT}, and the inclusions are equalities since 
$\calf=\calf_S(G)$.

\smallskip

\noindent\textbf{(b) } Now assume $G$ is countable and $S\in\sylp{G}$. 
Choose a sequence of elements $h_1,h_2,h_3,\dots$ in $G$ such that 
$G=\gen{h_i\,|\,i\ge1}$, and set $H_m=\gen{h_1,h_2,\dots,h_m}$ for each 
$m\ge1$. Then each $H_m$ is finite since $G$ is locally finite (Definition 
\ref{d:LT-real.}), and we can choose subgroups $T_1\le T_2\le T_3\le\dots$ 
such that $T_i\in\sylp{H_i}$ for each $i$. Set $T=\bigcup_{i=1}^\infty 
T_i$. Then $\9x\,T\le S$ for some $x\in G$ (since $S\in\sylp{G}$), and upon 
setting $G_i=\9xH_i$ and $S_i=\9x\,T_i$, we get $S_i\le S\cap G_i$ with 
equality since $S_i\in\sylp{G_i}$. Also, $\bigcup_{i=1}^\infty G_i = 
\bigcup_{i=1}^\infty H_i = G$ by construction (and since $x\in G$). 

\smallskip

For the rest of the proof, $G_1\le G_2\le\dots$ and $S_1\le S_2\le\dots$ 
are arbitrary finite subgroups of $G$ and $S$, respectively, such that 
$G=\bigcup_{i=1}^\infty G_i$ and $S_i\in\sylp{G_i}$. Note that $S_i=S\cap 
G_i$ for each $i$ since $S_i$ is a maximal $p$-subgroup of $G_i$.

\smallskip

\noindent\textbf{(b.1) } By definition, $\bigcup_{i=1}^\infty S_i = 
\bigcup_{i=1}^\infty(G_i\cap S) =\bigl(\bigcup_{i=1}^\infty G_i\bigr)\cap S 
= S$.

\smallskip

\noindent\textbf{(b.2) } Set $\calf_i=\calf_{S_i}(G_i)$ for each $i$; thus 
$\calf_i\le\calf$ by construction. To prove that $\calf$ is the union of 
the $\calf_i$, fix a morphism $\varphi\in\homf(P,Q)$ for some $P,Q\le S$, 
and let $g\in G$ be such that $\9gP\le Q$ and $\varphi=c_g|_P$. Choose 
$n\ge1$ such that $g\in G_n$. For each $i\ge n$, if we set $P_i=P\cap G_i$ 
and $Q_i=Q\cap G_i$, then $c_g^{P_i}\in\Hom_{\calf_i}(P_i,Q_i)$ is the 
restriction of $\varphi$. Also, $P=\bigcup_{i=1}^\infty P_i$, and so 
$\calf$ is contained in the union of the $\calf_i$, with equality since it 
satisfies the continuity axiom. 
\end{proof}

\begin{Cor} \label{c:f.s.union}
Let $G$ be a locally finite group, and let $S\le G$ be a discrete $p$-toral 
subgroup. Assume 
\begin{enumi} 

\item $S\in\sylp{G_*}$ for some countable $G_*\le G$ such that 
$\calf_S(G_*)=\calf_S(G)$; and 


\item $\calf_S(G)$ is saturated.

\end{enumi}
Then $\calf_S(G)$ is sequentially realizable.
\end{Cor}

\begin{proof} Set $\calf=\calf_S(G)$, and let $G_*\le G$ be as in (i). Thus 
$\calf$ is a saturated fusion system by (ii), $G_*$ is countable, $S\in\sylp{G_*}$, 
and $\calf_S(G_*)=\calf$. By Proposition \ref{p:f.s.union}(b) and since 
$S\in\sylp{G_*}$, there is an increasing sequence of finite subgroups 
$G_1\le G_2\le\dots$ of $G_*$ such that $S\cap G_i\in\sylp{G_i}$ for each 
$i$, $S=\bigcup_{i=1}^\infty(S\cap G_i)$, and 
$\calf=\bigcup_{i=1}^\infty\calf_{S\cap G_i}(G_i)$. So $\calf=\calf_S(G)$ 
is sequentially realizable. 
\end{proof}

Let $\calf_1$ and $\calf_2$ be fusion systems over discrete $p$-toral 
groups $S_1$ and $S_2$, respectively, set $S=S_1\times S_2$, and let 
$\pr_i\:S\too S_i$ be the projection for $i=1,2$. The product 
$\calf_1\times\calf_2$ is defined to be the fusion 
system over $S$, where for each $P,Q\le S$ we have 
	\[ \Hom_{\calf_1\times\calf_2}(P,Q) = \bigl\{(\varphi_1,\varphi_2)|_P 
	\,\big|\, \varphi_i\in\Hom_{\calf_i}(\pr_i(P),\pr_i(Q)), ~ 
	(\varphi_1,\varphi_2)(P)\le Q \bigr\}. \]
In other words, it is the smallest fusion system over $S$ that contains all 
morphisms $\varphi_1\times\varphi_2$ for $\varphi_i\in\Mor(\calf_i)$. It is 
also the largest fusion system $\calf$ over $S$ in which $S_1$ and $S_2$ 
are strongly closed, and such that $\calf/S_1\le\calf_2$ and 
$\calf/S_2\le\calf_1$. One can show that a product of fusion systems is 
saturated if each factor is, but we won't need to use 
that here. 

We next check that sequential realizability of fusion systems is preserved by 
products and quotients.

\begin{Prop} \label{p:F1xF2.real}
Let $\calf$ be a fusion system over a discrete $p$-toral group $S$. 
\begin{enuma} 

\item If $\calf$ is sequentially realizable, then so is $\calf/Q$ for each 
subgroup $Q\nsg S$ strongly closed in $\calf$. 

\item If $\calf=\calf_1\times\calf_2$ for some pair of fusion subsystems 
$\calf_1$ and $\calf_2$, then $\calf$ is sequentially realizable if and 
only if each of the factors $\calf_i$ is sequentially realizable.

\end{enuma}
\end{Prop}

\begin{proof} \textbf{(a) } Assume $Q\nsg S$ is strongly closed in $\calf$ 
and $\calf$ is sequentially realizable. Thus there are finite subgroups 
$S_1\le S_2\le S_3\le\dots$ of $S$, fusion subsystems 
$\calf_1\le\calf_2\le\dots$ of $\calf$, and finite groups $G_1,G_2,\dots$ 
such that $S=\bigcup_{i=1}^\infty S_i$ and 
$\calf=\bigcup_{i=1}^\infty\calf_i$; and also $S_i\in\sylp{G_i}$ and 
$\calf_i=\calf_{S_i}(G_i)$ for each $i$. 

For each $i$, set $Q_i=S_i\cap Q$. Then $Q_i$ is strongly closed in 
$\calf_i$ since $Q$ is strongly closed in $\calf$, and 
$\calf_i/Q_i=\calf_{S_i/Q_i}(N_{G_i}(Q_i)/Q_i)$ by Lemma \ref{l:F(G)/Q}. 
Also, $\calf_i/Q_i$ is isomorphic to the image of $\calf_i\le\calf$ in 
$\calf/Q$, and so $\calf/Q$ is the union of fusion subsystems isomorphic to 
the $\calf_i/Q_i$. Thus $\calf/Q$ is sequentially realizable. 

\smallskip

\noindent\textbf{(b) } If $\calf=\calf_1\times\calf_2$, where $\calf_i$ is 
a fusion system over $S_i$ and $S=S_1\times S_2$, then by (a), 
$\calf_1\cong\calf/S_2$ and $\calf_2\cong\calf/S_1$ are sequentially 
realizable if $\calf$ is. The converse is clear: if $\calf_1$ and $\calf_2$ 
are both sequentially realizable, then so is their product. 
\end{proof}

We finish the section with the following example of an increasing sequence 
of finite saturated fusion systems that alternate between being realizable 
and exotic. Thus the union of the fusion systems is sequentially realizable 
(in fact, realized by a linear torsion group), but it is also the union of 
an increasing sequence of finite exotic subsystems.

\begin{Ex} \label{ex:seq.real-exo}
Fix an odd prime $p$ and a prime $q\equiv1$ (mod $p^2$). Set 
$K=\bigcup_{i=0}^\infty\F_{q^{p^i}} \subseteq \4\F_q$. Set 
$G=\PSL_p(K)$, and $G_i=\PSL_p(q^{p^i})$ for all $i\ge0$. Then there are Sylow 
$p$-subgroups $S_i\in\sylp{G_i}$ and $S\in\sylp{G}$ such that 
$S_i=G_i\cap S$ for each $i\ge0$, and such that if we set $\calf=\calf_S(G)$ and 
$\calf_{2i}=\calf_{S_i}(G_i)$ for all $i\ge0$, then there are 
exotic saturated fusion systems $\calf_{2i+1}$ over $S_{i+1}$ for all $i$ 
such that $\calf_{2i}\le\calf_{2i+1}\le\calf_{2i+2}$. Thus 
$\calf=\bigcup_{i=0}^\infty\calf_i$, where the $\calf_i$ are realizable for 
even $i$ and exotic for odd $i$.
\end{Ex}

\begin{proof} Set $e=v_p(q-1)$; thus $e\ge2$ by assumption. Then 
$v_p(q^{p^i}-1)=e+i$ for each $i\ge0$ (see, e.g., \cite[Lemma 1.13]{BMO2}). 
Let $\5A\le\SL_p(K)$ be the group of all diagonal matrices of $p$-power 
order, and set $Z=Z(\SL_p(K))=\gen{\zeta\cdot\Id}$ where 
$\zeta\in\F_q^\times$ has order $p$. Set $\5A_i=\5A\cap\SL_p(q^{p^i})$ for 
each $i\ge0$, and also $A=\5A/Z\le G$ and $A_i=\5A_i/Z\le G_i$. Thus 
$\5A_i$ is homocyclic of rank $p-1$ and exponent $p^{e+i}$, and hence $A_i$ 
has rank $p-1$, exponent $p^{e+i}$, and order $p^{(e+i)(p-1)-1}$.

Let $x\in G_0=\SL_p(q)$ be a permutation matrix of order $p$ (i.e., a 
matrix with a unique entry $1$ in each row an column and otherwise $0$). 
Set $S=A\gen{x}$ (the subgroup of $G$ generated by $A$ and $x$), and set 
$S_i=A_i\gen{x}$ for each $i$. Then $S_i\in\sylp{G_i}$ (see 
\cite[\S\,2]{Weir}), and hence $S=\bigcup_{i=0}^\infty S_i\in\sylp{G}$ by 
Proposition \ref{p:Pi<Q} (every finite $p$-subgroup of $G$ is contained in 
$G_i$ for some $i$, and hence is conjugate to a subgroup of $S$). By 
definition, $S_i=S\cap G_i$ for each $i$. 

For each $i\ge0$, let $\calh^{(i)}$ be the set of subgroups $E\le S_i$ such 
that $E\cong C_p\times C_p$ and $E\cap A_i=Z(S_i)\cong C_p$. By \cite[Lemma 
2.2(f)]{indp2}, there are exactly $p$ $S_i$-conjugacy classes of subgroups 
in $\calh^{(i)}$, and for $x\in S_i\sminus A_i$ and $a\in A_i$, 
$Z(S_i)\gen{x}$ is $S_i$-conjugate to $Z(S_i)\gen{ax}$ if and only if 
$a\in[S_i,S_i]=[A_i,x]$. (Note that $[S_i,S_i]$ has index $p$ in $A_i$.)

Now assume $i\ge1$. Since $A_{i-1}\le[S_i,S_i]$, the members of 
$\calh^{(i-1)}$ are all contained in one of the $S_i$-conjugacy classes 
$\calh^{(i)}_0\subseteq\calh^{(i)}$. By case (iii) of \cite[Theorem 
2.8]{indp2}, or case (a.i) of \cite[Theorem 2.8]{indp1} if $p=3$, there is 
a unique saturated fusion system $\calf_{2i-1}\le\calf_{2i}$ over $S_i$ whose 
essential subgroups are the members of $\calh^{(i)}_0$, together with $A_i$ 
if $p\ge5$. Furthermore, this fusion system is exotic by \cite[Table 
2.2]{indp2} or the last statement in \cite[Theorem 2.8]{indp1}. 

It remains only to check that $\calf_{2i-1}\ge\calf_{2i-2}$. In both fusion 
systems, the members of $\calh^{(i-1)}$ are essential, and we claim that 
they all have automizer $\SL_2(p)$. This is shown in case (a.i) of 
\cite[Theorem 2.8]{indp1} when $p=3$, and follows by a similar argument 
when $p\ge5$. 

Since $\calf_{2i-1}$ and $\calf_{2i-2}$ are generated by 
$N_{\calf_{2i}}(T_i)\ge N_{\calf_{2i-2}}(T_{i-1})$ and these automizers, 
this shows that $\calf_{2i-1}\ge\calf_{2i-2}$. 
\end{proof}


\section{Linear torsion groups and LT-realizability}
\label{s:LT-gps}

In this section, we recall the definition of linear torsion groups and 
prove that every fusion system that is realized by such a group is also 
sequentially realizable. In fact, this is our only tool so far for showing 
this: whenever we prove that a fusion system over an infinite discrete 
$p$-toral group is sequentially realizable, we do so by showing that it's 
realized by a linear torsion group. 


\begin{Defi} \label{d:LT-real.}
\begin{enuma} 

\item A \emph{linear torsion group in characteristic $q$} is a (discrete) 
group $G$, each of whose elements has finite order, that is isomorphic to a 
subgroup of $\GL_n(K)$ for some $n\ge1$ and some field $K$ of 
characteristic $q$. 

\item A fusion system $\calf$ over a discrete $p$-toral group $S$ is 
\emph{\textup{LT}-realizable} (in characteristic $q$) if 
$\calf\cong\calf_{S^*}(G)$ for some linear torsion group $G$ in 
characteristic $q\ne p$ with $S^*\in\sylp{G}$. 

\end{enuma}
\end{Defi}

Our interest in this class of groups is due mostly to the following 
proposition, shown in \cite{BLO3}. Recall the definition of $\call_S^c(G)$ 
in Definition \ref{d:LSc(G)}.

\begin{Prop} \label{p:LT-real.}
Fix a prime $p$, a field $K$ with $\chr(K)\ne p$, and a linear torsion 
group $G\le\GL_n(K)$ (some $n\ge1$). Then $G$ is locally finite, all 
$p$-subgroups of $G$ are discrete $p$-toral, $\sylp{G}\ne\emptyset$, and 
every maximal $p$-subgroup of $G$ is a Sylow $p$-subgroup. For each 
$S\in\sylp{G}$, the fusion system $\calf_S(G)$ is saturated, $\call_S^c(G)$ 
is a centric linking system associated to $\calf_S(G)$, and 
$|\call_S^c(G)|\pcom\simeq BG\pcom$.
\end{Prop}

\begin{proof} By \cite[Proposition 8.8]{BLO3}, $G$ is locally finite and 
all $p$-subgroups of $G$ are discrete $p$-toral. By Proposition 8.9 in the 
same paper, for each increasing sequence $A_1\le A_2\le\cdots$ of finite 
abelian $p$-subgroups of $G$, there is $r\ge1$ such that 
$C_G(A_i)=C_G(A_r)$ for all $i\ge1$. So all maximal $p$-subgroups of $G$ 
are conjugate to each other by \cite[Theorem 3.4]{KW}, and hence they are 
Sylow $p$-subgroups. (Since the union of an increasing sequence of 
$p$-subgroups of $G$ is again a $p$-subgroup, every $p$-subgroup is 
contained in a maximal $p$-subgroup.)

For each $S\in\sylp{G}$, the fusion system $\calf_S(G)$ is saturated, 
$\call_S^c(G)$ is a centric linking system, and $|\call_S^c(G)|\pcom\simeq 
BG\pcom$ by Theorem 8.7 or 8.10 in \cite{BLO3}. The last statement (the 
homotopy equivalence) was affected by an error we found in the proof of 
\cite[Lemma 5.12]{BLO3}, but this error has now been fixed in \cite[Theorem 
3.7]{O-Lambdas}. 
\end{proof}

The results in the last section imply that LT-realizable fusion systems are 
sequentially realizable. 

\begin{Prop} \label{p:LT=>s.real.}
Every LT-realizable fusion system is sequentially realizable. 
\end{Prop}

\begin{proof} If $\calf=\calf_S(G)$ for a linear torsion group $G$ in 
characteristic different from $p$ and for $S\in\sylp{G}$, then $\calf$ is 
saturated by Proposition \ref{p:LT-real.}. By Proposition 
\ref{p:f.s.union}(a), there is a countable subgroup $G_*\le G$ such that 
$S\le G_*$ and $\calf=\calf_S(G_*)$. Then $S$ is a maximal $p$-subgroup of 
$G_*$ and hence a Sylow $p$-subgroup by Proposition \ref{p:LT-real.} again. 
So $\calf$ is sequentially realizable by Corollary \ref{c:f.s.union}. 
\end{proof}

It seems unlikely that the converse of Proposition \ref{p:LT=>s.real.} 
holds, but we know of no example of a sequentially realizable fusion system 
that is not LT-realizable, and it seems very difficult to construct one.

We now focus attention on the question of in which characteristics a given 
fusion system is LT-realizable. 

When $G$ is a group and $p$ is a prime, the 
\emph{$p$-rank} of $G$ is defined by setting
	\[ \rk_p(G) = \sup\{ \rk(P) \,|\, P\le G ~\textup{a finite abelian 
	$p$-subgroup}\}. \]
We note the following:

\begin{Lem} \label{l:rk_q(G/P)}
Let $G$ be a (discrete) group, and let $N\nsg G$ be a finite normal 
subgroup. Then $\rk_q(G/N)=\rk_q(G)$ for each prime $q$ with $q\nmid|N|$.
\end{Lem}

\begin{proof} If $H/N$ is a finite abelian $q$-subgroup of $G/N$, then $H$ 
is finite, and each $S\in\syl{q}{H}$ is a finite abelian $q$-subgroup of 
$G$ isomorphic to $H/N$. Thus $\rk_q(G)\ge\rk_q(G/N)$, and the opposite 
inequality is clear.
\end{proof}

It is not hard to see that Lemma \ref{l:rk_q(G/P)} also holds if one 
assumes that $N$ be a normal discrete $p$-toral subgroup for some prime 
$p\ne q$.

\begin{Defi} \label{d:srk(G)}
For every group $G$ and every prime $p$, the \emph{sectional 
$p$-rank} of $G$ is defined by setting
	\begin{align*} 
	\srk_p(G) &= \sup\bigl\{ \rk(P/\Phi(P)) \,\big|\, P\le G ~ 
	\textup{a finite $p$-subgroup of $G$} \bigr\} \\
	&= \sup\bigl\{ r \,\big|\, \exists \, N\nsg H\le G ~ 
	\textup{such that $H$ finite,} ~ H/N\cong E_{p^r} \bigr\}. 
	\end{align*}
\end{Defi}

\begin{Lem} \label{l:srk(GLn)}
Fix a locally finite group $G$, a field $K$, and $n\ge1$ such that 
$G\le\GL_n(K)$. Then for each prime $p$, 
\begin{enuma} 

\item if $p\ne\chr(K)$, then $\rk_p(G)\le n$ and 
$\srk_p(G)\le pn/(p-1)$; and 

\item if $\chr(K)=p$, then $G$ has no elements of order $p^k$ if 
$n\le p^{k-1}$. 

\end{enuma}
\end{Lem}

\begin{proof} We can assume without loss of generality that $K$ is 
algebraically closed. 

\smallskip

\noindent\textbf{(a) } Let $D\le\GL_n(K)$ be the subgroup of diagonal 
matrices. If $p\ne\chr(K)$, then by elementary representation theory, all 
irreducible $K$-representations of a finite abelian group of order not 
divisible by $\chr(K)$ are 1-dimensional. Thus every finite abelian 
$p$-subgroup of $\GL_n(K)$ is conjugate to a subgroup of $D$, and hence 
$\rk_p(G)\le\rk_p(D)=n$.

Set $M=N_{\GL_n(K)}(D)$; thus $M/D\cong\Sigma_n$. 
By \cite[\S8.5, Theorem 16]{Serre}, every finite $p$-subgroup of $\GL_n(K)$ 
is conjugate to a subgroup of $M$. So $\srk_p(G)\le n+m$, where $m$ is the 
sectional $p$-rank of $M/D\cong\Sigma_n$. Since $\Sigma_n$ has Sylow 
$p$-subgroups of order $p^e$ where $e=[n/p]+[n/p^2]+[n/p^3]+\cdots$, we have
	\[ \srk_p(G)\le n+m \le n+e \le 
	n\bigl(1+\tfrac1p+\tfrac1{p^2}+\tfrac1{p^3}+\cdots\bigr) 
	= pn/(p-1). \]

\smallskip

\noindent\textbf{(b) } If $\chr(K)=p>0$, and $g\in G\le\GL_n(K)$ has order 
$p^k$, then the minimal polynomial for $g$ divides 
$(X^{p^k}-1)=(X-1)^{p^k}$, hence has the form $(X-1)^m$ for some $m\le n$, 
where $m>p^{k-1}$ since otherwise $|g|\le p^{k-1}$. 
\end{proof}

\begin{Rmk} \label{rmk:srk(GL)}
In fact, one can show that for $n\ge1$ and $K$ an algebraically closed 
field with $\chr(K)\ne p$, we have $\srk_p(\GL_n(K))=n$ if $p$ is odd, and 
$\srk_p(\GL_n(K))=[3n/2]$ if $p=2$. But the above bound is good enough for 
our purposes here. When $p=2$, the bound $[3n/2]$ is realized by taking a 
direct product of $[n/2]$ copies of $C_4\circ D_8$, with an 
additional factor $C_2$ if $n$ is odd.
\end{Rmk}

Thus if $G$ is a linear torsion group in characteristic $q$, then 
$\rk_p(G)<\infty$ and $\srk_p(G)<\infty$ for every prime $p\ne q$, while 
if $q>0$, then there is a bound on the orders of elements of $q$-power 
order in $G$. 


\begin{Ex} 
For each prime $q$, the group $\bigoplus_{i=1}^\infty\Z/q$ is a linear 
torsion group in characteristic $q$ (a subgroup of $\SL_2(\4\F_q)$), but 
not in any characteristic different from $q$. The group 
$\mu_\infty\le\C^\times$ of all complex roots of unity is a linear torsion 
group in characteristic $0$ (a subgroup of $\GL_1(\C)$), but not in any 
other characteristic.
\end{Ex}

The following example shows how one particular fusion system can be 
LT-realizable in certain characteristics but not in others. (The 
$2$-fusion system of $\SO(3)$ is LT-realizable in all odd characteristics 
by Theorem \ref{t:cpt.cn.Lie}.) More examples of fusion systems that 
are not LT-realizable in characteristic $0$ are given in Proposition 
\ref{p:no.char0} and Theorem \ref{t:cn.p-cpt}(b). 

\begin{Ex} \label{ex:no.char0}
Set $p=2$, choose $S\in\syl2{\SO(3)}$, and set $\calf=\calf_S(\SO(3))$. 
Then $\calf$ is not isomorphic to the fusion system of any linear torsion 
group in characteristic $0$.
\end{Ex}

\begin{proof} Assume otherwise: assume $\calf=\calf_S(G)$ for some torsion 
subgroup $G\le\GL_n(K)$ for $n\ge1$ and $K$ a field of characteristic $0$. 
Upon replacing $G$ by $G/O_{p'}(G)$, we can assume that $O_{p'}(G)=1$. By 
Proposition \ref{p:f.s.union}, there is an increasing sequence of finite 
subgroups $\{G_i\}_{i\ge1}$ such that upon setting $S_i=S\cap G_i$ and 
$\calf_i=\calf_{S_i}(G_i)$, we have $S_i\in\sylp{G_i}$, 
$S=\bigcup_{i=1}^\infty S_i$, and $\calf=\bigcup_{i=1}^\infty\calf_i$. 

Since $S$ is a Sylow $2$-subgroup of $\SO(3)$, we have 
$S\cong(\Z/2^\infty)\rtimes C_2$. So there is $m\ge1$ such that $S_i$ is 
dihedral of order at least $16$ for all $i\ge m$. Choose $T\le S_m$ with 
$T\cong E_4$; then $T$ is $\calf$-radical (all subgroups of $S$ isomorphic 
to $E_4$ are $\calf$-radical), so $\Aut_G(T)\cong\Sigma_3$, and there is 
$n\ge m$ such that $\Aut_{\calf_i}(T)=\Aut_{G_i}(T)\cong\Sigma_3$ for all 
$i\ge n$. So $T$ is $\calf_i$-radical for all $i\ge n$.

Thus for $i\ge n$, $S_i\cong D_{2^k}$ for $k\ge4$, and at least one of the 
two conjugacy classes of subgroups of $S_i$ isomorphic to $E_4$ is 
$\calf_i$-radical. So by the Gorenstein-Walter theorem 
\cite[\S16.3]{Gorenstein}, $G_i/O_{2'}(G_i)$ is isomorphic to 
$\PSL_2(q_i)$ or $\PGL_2(q_i)$ for some odd prime power $q_i$, or is an 
extension of one of these by a group of field automorphisms of odd order. 

For each $i\ge n$, the group $G_i/(G_i\cap O_{2'}(G_{i+1}))$ is isomorphic 
to a subgroup of $G_{i+1}/O_{2'}(G_{i+1})$, and is nonsolvable of order a 
multiple of $16$. The subgroups of $\PSL_2(q_{i+1})$ were described by 
Dickson (see \cite[Theorem 6.5.1]{GLS3}), and from his list we see that 
each nonsolvable subgroup of $\PGL_2(q_{i+1})$ of order a multiple of $16$ 
is isomorphic to $\PSL_2(q')$ or $\PGL_2(q')$ for some $q'$ of which 
$q_{i+1}$ is a power. It follows that $q_{i+1}$ is a power of $q_i$.

Thus there is a prime $q$ of which each $q_i$ is a power. The $q_i$ must 
include arbitrarily large powers of $q$ since the Sylow subgroups $S_i$ 
get arbitrarily large. Hence the sectional $q$-rank of $G_i$ becomes 
arbitrarily large for large $i$. But since $G\le\GL_n(K)$ where 
$\chr(K)=0$, we have $\srk_q(G)\le pn/(p-1)$ by Lemma \ref{l:srk(GLn)}(a), 
which is impossible.
\end{proof}

More generally, we will show in Theorem \ref{t:cn.p-cpt}(b) that a similar 
result holds whenever $\calf$ is the fusion system of a compact connected 
Lie group $G$ over a Sylow $p$-subgroup and $p$ divides the order of the 
Weyl group of $G$: such an $\calf$ cannot be realized by a linear torsion 
group in characteristic $0$. However, in contrast to the above result, the 
proof of Theorem \ref{t:cn.p-cpt}(b) depends on the classification of 
finite simple groups.

Using results and arguments similar to those in Section \ref{s:fin.p-gp} 
(and the classification of finite simple groups), we can also prove that 
\begin{itemize} 

\item if $p=3$ and $\calf$ is the fusion system of the $3$-compact group 
$X_{12}$, then $\calf$ is realized by a linear torsion group only 
in characteristic $2$; and 

\item if $p=5$ and $\calf$ is the fusion system of the $5$-compact group 
$X_{31}$, then $\calf$ is realized by a linear torsion group only in 
characteristic $q$ for $q\equiv\pm2$ (mod $5$). 

\end{itemize} Here, $X_n$ denotes a connected $p$-compact group whose Weyl 
group is the $n$-th group in the Shephard-Todd list \cite[Table VII]{ST}. 
The classifying spaces $BX_{12}$ and $BX_{31}$ were first constructed by 
Zabrodsky \cite[\S\,4.3]{Zabrodsky}, and a different construction of these 
spaces as well as of $BX_{29}$ and $BX_{34}$ was given by Aguad\'e 
\cite{Aguade}.

The following is an obvious question, but one that seems quite difficult.

\begin{Quest} 
Is there a fusion system over a discrete $p$-toral group that is realized 
by a linear torsion group in characteristic $0$ but not by one in any positive 
characteristic? 
\end{Quest}


\section{Fusion systems of compact Lie groups}
\label{s:cpt.Lie}

The main result in this section is Theorem \ref{t:cpt.Lie} (Theorem 
\ref{ThB}): for every prime $p$ and every compact Lie group $G$, the 
$p$-fusion system of $G$ is realized by a linear torsion group $\Gamma$ in 
characteristic different from $p$ (hence is also sequentially realizable). 
We begin by reducing this to a question of showing that the $p$-completed 
classifying spaces of $G$ and $\Gamma$ are homotopy equivalent. Here, 
``$p$-completed'' means in the sense of Bousfield and Kan \cite{BK}.

\begin{Lem} \label{l:BG=BGamma}
Fix a prime $p$, a compact Lie group $G$, and a linear torsion group 
$\Gamma$ in characteristic different from $p$ such that $BG\pcom\simeq 
B\Gamma\pcom$. Then for $S\in\sylp{G}$ and $S_\Gamma\in\sylp\Gamma$, we 
have $\calf_S(G)\cong\calf_{S_\Gamma}(\Gamma)$. \\
\end{Lem}

\begin{proof} By Proposition \ref{p:LT-real.} and \cite[Theorem 
9.10]{BLO3}, there are centric linking systems $\call_S^c(G)$ and 
$\call_{S_\Gamma}^c(\Gamma)$, associated to $\calf_S(G)$ and 
$\calf_{S_\Gamma}(\Gamma)$, respectively, such that 
	\[ |\call_S^c(G)|\pcom \simeq BG\pcom \simeq B\Gamma\pcom \simeq 
	|\call_{S_\Gamma}^c(\Gamma)|\pcom. \]
So $\calf_S(G)\cong\calf_{S_\Gamma}(\Gamma)$ by \cite[Theorem 7.4]{BLO3}. 
\end{proof}

We next look at the special case of Theorem \ref{ThB} where we assume that 
the Lie groups are connected. This follows from a theorem of Friedlander 
and Mislin. 

\begin{Thm} \label{t:cpt.cn.Lie}
Let $p$ be a prime, let $G$ be a compact connected Lie group, and fix 
$S\in\sylp{G}$. Then $\calf_S(G)$ is \textup{LT}-realizable. More 
precisely, if $G$ is a maximal compact subgroup of $\gg(\C)$ where $\gg$ is 
a connected, reductive algebraic group scheme over $\Z$, then $\calf_S(G)$ 
is realized by the linear torsion group $\gg(\4\F_q)$ for each prime $q\ne 
p$. 
\end{Thm}

\begin{proof} 
Let $G$ be a compact connected Lie group. By \cite[Propositions 
III.8.2--4]{BtDieck}, there is a complex connected algebraic group $G_{\C}$ 
containing $G$ such that $L(G_\C)\cong\C\otimes_\R L(G)$ (the Lie 
algebras). Hence $G$ is a maximal compact subgroup of $G_\C$, and 
$BG_\C\simeq BG$ since $G_\C/G$ is diffeomorphic to a Euclidean space by 
\cite[Theorem XV.3.1]{Hochschild}. Also, $G_{\C}\cong\gg(\C)$ for some 
connected reductive algebraic group scheme over $\Z$, and 
$B\gg(\C)\pcom\simeq B\gg(\4\F_q)\pcom$ for each prime $q\ne p$ by a 
theorem of Friedlander and Mislin \cite[Theorem 1.4]{FM}. So $\calf_S(G)$ 
is realized by $\gg(\4\F_q)$ for each such $q$ by Lemma \ref{l:BG=BGamma}. 
\end{proof}

We are now ready to prove Theorem \ref{ThB}, in the following slightly more 
precise form.

\begin{Thm} \label{t:cpt.Lie} 
Let $p$ be a prime, let $G$ be a compact Lie group, and fix $S\in\sylp{G}$. 
Then $\calf_S(G)$ is 
\textup{LT}-realizable. More precisely, if $G_e\nsg G$ denotes the identity 
connected component of $G$, then for each prime $q\ne p$, there are 
linear torsion groups 
$\Gamma_e\nsg\Gamma\le\GL_n(\4\F_q)$ (for some $n$) such that 
$\Gamma/\Gamma_e\cong G/G_e$, and for $S_\Gamma\in\sylp{\Gamma}$, we have 
$\calf_S(G)\cong\calf_{S_\Gamma}(\Gamma)$ and 
$\calf_{S\cap G_e}(G_e)\cong\calf_{S_\Gamma\cap\Gamma_e}(\Gamma_e)$. 
\end{Thm}

\begin{proof} Set $\pi=G/G_e=\pi_0(G)$, and let $\delta_G\:G\too\pi$ be the 
natural surjective homomorphism. Set $G_s=[G_e,G_e]$: the ``semisimple 
part'' of $G_e$. Since $Z(G_s)$ is finite, we can replace $G$ by 
$G/O_{p'}(Z(G_s))$, and arrange that $Z(G_s)$ be a finite abelian $p$-group 
(without changing the fusion system of $G$ or of $G_e$).

In general, in what follows, when $X\too B\pi$ is a fibration, we let $\4X$ 
denote the fiberwise $p$-completion of $X$ as defined in 
\cite[\S\,I.8]{BK}. Thus if $F$ is the fiber of $X\too B\pi$, then 
$F\pcom$ is the fiber of $\4X\too B\pi$. Also, when $F$ is $p$-good, 
the natural map $X\too\4X$ is a mod $p$ homology equivalence by the Serre 
spectral sequences for the fibrations and since $H^*(F;\F_p)\cong 
H^*(F\pcom;\F_p)$, and so $\4X\pcom\simeq X\pcom$ by \cite[Lemma 
I.5.5]{BK}.

In each of the cases considered below, we construct a linear torsion group 
$\Gamma$, together with a surjective homomorphism 
$\delta_\Gamma\:\Gamma\too\pi$ and a fiber homotopy equivalence 
	\[ (\4{B\Gamma}\too B\pi ) \simeq (\4{BG} \too B\pi) . \]
In particular, this implies that $B\Gamma\pcom\simeq BG\pcom$, and hence by 
Lemma \ref{l:BG=BGamma} that $\calf_S(G)\cong\calf_{S_\Gamma}(\Gamma)$ (for 
$S_\Gamma\in\sylp\Gamma$). So $\calf_S(G)$ is LT-realizable. Also, 
$(B\Gamma_e)\pcom\simeq(BG_e)\pcom$, so $\calf_{S\cap G_e}(G_e)\cong 
\calf_{S_\Gamma\cap\Gamma_e}(\Gamma_e)$.

\smallskip

\noindent\boldd{Case 1: $Z(G_e)=1$.} In this case, $G_e$ is a product of 
simple groups with trivial center. As in the proof of Theorem 
\ref{t:cpt.cn.Lie}, let $\gg$ be a connected, semisimple 
group scheme over $\Z$ such that $G_e$ is a maximal compact subgroup of 
$\gg(\C)$, and set $\Gamma_e=\gg(\4\F_q)$. Thus $\Gamma_e$ is a product of 
simple groups with $Z(\Gamma_e)=1$. By the Friedlander-Mislin theorem 
\cite[Theorem 1.4]{FM}, there is a homotopy equivalence 
	\[ \psi\:(B\Gamma_e)\pcom \Right4{\simeq} 
	B\gg(\C)\pcom\simeq(BG_e)\pcom. \]

Let $R$ denote the root system of $G_e$ and of $\Gamma_e$. We regard the 
roots as elements in the dual $V^*$ of a real vector space $V$ that can be 
identified with the universal cover of a maximal torus in $G_e$. We fix a 
Weyl chamber $C\subseteq V$ and let $R^+$ denote the corresponding set of 
positive roots. Thus $R^+$ is the set of all $r\in R$ such that $r(x)>0$ 
for $x\in C$, while $C$ is the set of $x\in V$ such that $r(x)>0$ for all 
$r\in R^+$. Also, $R=\{\pm r\,|\,r\in R^+\}$. Let $\Isom(R)$ be the group 
of all isometries of $R$, and let $\Isom^+(R)$ be the subgroup of those 
isometries that permute the positive roots; equivalently, those that send 
the Weyl chamber $C$ to itself.

Let $\Aut(\Gamma_e)$ be the group of all automorphisms of $\Gamma_e$ as an 
algebraic group, and set $\Out(\Gamma_e)=\Aut(\Gamma_e)/\Inn(\Gamma_e)$. By 
\cite[Theorem 1.15.2]{GLS3}, the group $\Out(\Gamma_e)$ is isomorphic to 
$\Isom^+(R)$. By \cite[\S\,4.10, Proposition 17]{Bourbaki9}, 
$\Out(G_e)\cong\Isom(R)/W$, where $W\nsg\Isom(R)$ is the Weyl group. Since 
the Weyl group permutes the Weyl chambers simply and transitively 
\cite[\S\,VI.1.5, Th\'eor\`eme 2]{Bourbaki4-6}, this shows that there is a 
natural isomorphism 
	\[ \theta\:\Out(G_e) \cong \Isom(R)/W \Right4{\cong} \Isom^+(R) \cong 
	\Out(\Gamma_e). \] 

Let $\rho_G\:\Out(G_e)\too\Out((BG_e)\pcom)$ and 
$\rho_\Gamma\:\Out(\Gamma_e)\too\Out((B\Gamma_e)\pcom)$ be the 
homomorphisms induced 
by the functor from topological groups to their $p$-completed classifying 
spaces. Let $\eta_G\:\pi\too\Out(G_e)$ be induced by conjugation, set 
$\eta_\Gamma=\theta\circ\eta_G$, and consider the following diagram 
	\beqq \vcenter{\xymatrix@C=40pt@R=25pt{ 
	\pi \ar@{=}[d] \ar[r]^-{\eta_G} & \Out(G_e) \ar[r]^-{\rho_G} 
	\ar[d]_{\cong}^{\theta} & \Out((BG_e)\pcom) 
	\ar[d]_{\cong}^{c_\psi^{-1}} \\
	\pi \ar[r]^-{\eta_\Gamma} & \Out(\Gamma_e) \ar[r]^-{\rho_\Gamma} & 
	\Out((B\Gamma_e)\pcom) \rlap{.}
	}} \label{e:cpt.Lie0} \eeqq
The left hand square commutes by definition. 

To see that the right hand square in \eqref{e:cpt.Lie0} commutes, fix 
$\alpha\in\Aut(G_e)$. We just saw that its class $[\alpha]\in\Out(G_e)$ is 
induced by an isometry of the root system of $G$, and hence 
$[\alpha]=[\beta(\C)|_G]$ for some $\beta\in\Aut(\gg)$. Also, 
$\theta([\alpha])=\theta([\beta(\C)|_G])=[\beta(\4\F_q)]$ by the above 
definition of $\theta$. The equivalence $\psi\:B\gg(\4\F_q)\pcom\too 
B\gg(\C)\pcom$ is natural as a map of functors from reductive group schemes 
over $\Z$ to the homotopy category, since it is induced by the projection 
of the ring of Witt vectors $\W(\4\F_q)$ onto $\4\F_q$ together with a 
choice of embedding of $\W(\4\F_q)$ into $\C$ (see \cite[Corollary 2]{FP}). 
So $(B\beta(\C)|_{BG})\pcom \circ \psi = \psi \circ 
(B\beta(\4\F_q))\pcom$, and hence 
	\[ c_\psi^{-1}(\rho_G([\alpha])) = 
	c_\psi^{-1}([(B\beta(\C)|_{BG})\pcom]) = [(B\beta(\4\F_q))\pcom] = 
	\rho_\Gamma ([\beta(\4\F_q)]) = \rho_\Gamma 
	\theta([\alpha])\,. \]

Consider the commutative diagram
	\beqq \vcenter{\xymatrix@C=20pt@R=20pt{ 
	1 \ar[r] & G_e \ar[r]^-{\incl} \ar[d]_{c_1}^{\cong} & G 
	\ar[r]^-{\delta_G} \ar[d]_{c_2} & \pi \ar[r] \ar[d]_{\eta_G} & 1 \\
	1 \ar[r] & \Inn(G_e) \ar[r]^-{\incl} & \Aut(G_e) \ar[r] 
	& \Out(G_e) \ar[r] & 1
	}} \label{e:cpt.Lie5} \eeqq
where the rows are short exact sequences and $c_1$ and $c_2$ are induced by 
conjugation. Thus $c_1$ is an isomorphism since $Z(G_e)=1$. So $G$ is 
isomorphic to a pullback of $\pi$ and $\Aut(G_e)$ over $\Out(G_e)$. Define 
$\Gamma$ to be the analogous pullback of $\pi$ and $\Aut(\Gamma_e)$ over 
$\Out(\Gamma_e)$; then a similar diagram (but with $G$ replaced by 
$\Gamma$) shows that we can identify $\Gamma_e$ with a subgroup of $\Gamma$ 
such that $\Gamma/\Gamma_e\cong\pi$.

By \cite[Theorem 7.1]{BLO3} and since $Z(G_e)=1=Z(\Gamma_e)$, 
the space $B\Aut((BG_e)\pcom)$ is an Eilenberg-MacLane space with 
fundamental group $\Out((BG_e)\pcom)$. Hence the fibration sequence 
$(BG_e)\pcom\too\4{BG}\too B\pi$ is classified by the map 
	\[ B\pi \Right7{B(\rho_G\eta_G)} B\Out((BG_e)\pcom) 
	\simeq B\Aut((BG_e)\pcom). \]
(see \cite[Theorem IV.5.6]{BGM}). By the commutativity of 
\eqref{e:cpt.Lie0}, $c_\psi^{-1}\rho_G\eta_G=\rho_\Gamma\eta_\Gamma$, where 
$c_\psi$ is an isomorphism. Hence $((BG_e)\pcom\too\4{BG}\too B\pi)$ is 
fiberwise homotopy equivalent to the fibration sequence classified by 
$B(\rho_\Gamma\eta_\Gamma)$, and that sequence is the fiberwise 
completion of the fibration sequence $B\Gamma_e\too B\Gamma\too B\pi$ 
defined above. So $BG\pcom\cong B\Gamma\pcom$, and the fusion systems of 
$G$ and $\Gamma$ are isomorphic by Lemma \ref{l:BG=BGamma}. 

Finally, $\Gamma$ has a finite dimensional representation over $\4\F_q$ 
since $\Gamma_e$ does, and since $|\Gamma/\Gamma_e|=|\pi|<\infty$. 

\smallskip

\noindent\boldd{Case 2: $G_e$ is a torus.} Set $T=G_e$ to simplify 
notation, and let $T_f\le T$ be the subgroup of elements of finite order. 
Then $T/T_f$ is uniquely divisible, so 
$H^2(\pi;T/T_f)=0$, and the group $G/T_f$ is a semidirect product of 
$T/T_f$ with $\pi=G/T$. 

Let $s\:\pi\too G/T_f$ be a splitting of the projection $G/T_f\too\pi$, and 
set $G_f/T_f=s(\pi)$. Thus $G_f\le G$, where $T_f=T\cap G_f\nsg G_f$ and 
$G_f/T_f\cong\pi$. Also, $TG_f=G$. Set $\Gamma=G_f/O_{p'}(T_f)$ and 
$\Gamma_e=T_f/O_{p'}(T_f)$, where $O_{p'}(T_f)$ denotes the 
subgroup of all elements of order prime to $p$. Thus $\Gamma_e$ is a 
discrete $p$-torus, and $\Gamma/\Gamma_e\cong\pi$.

The natural homomorphisms $T\hookleftarrow T_f\too\Gamma_e$ induce mod $p$ 
homology equivalences $BT\fromm BT_f\too B\Gamma_e$, and they induce 
weak equivalences 
	\[ (BT)\pcom \simeq (BT_f)\pcom \simeq (B\Gamma_e)\pcom. 
	\]
So after fiberwise completion of $BG$, $BG_f$, and $B\Gamma$ over 
$B\pi$, we also get weak equivalences 
	\[ \4{BG} \simeq \4{BG_f} \simeq \4{B\Gamma}. \]
Hence $BG\pcom\simeq B\Gamma\pcom$.

\smallskip

\noindent\boldd{Case 3: $Z(G_s)=1$.} Set $G_1=G/Z(G_e)$ and $G_2=G/G_s$, 
and let $G_{1e}=G_e/Z(G_e)$ and $G_{2e}=G_e/G_s$ be their identity 
components. Then $G_e=Z(G_e)G_s$: every compact connected Lie group is 
a central product of a semisimple group with a torus (see \cite[Corollary 
5.5.31]{MToda}), and the torus factor is clearly contained in the center. 
Also, $Z(G_e)\cap G_s=Z(G_s)=1$ by assumption. So $G$ is isomorphic to the 
pullback of $G_1$ and $G_2$ over $\pi=G/G_e$, and this in turn 
restricts to an isomorphism $G_e\cong G_{1e}\times G_{2e}$.

Since $Z(G_{1e})=1$ and $G_{2e}$ is a torus, by Cases 1 and 2, there are 
pairs of groups $\Gamma_{1e}\nsg\Gamma_1$ and 
$\Gamma_{2e}\nsg\Gamma_2$ such that for $i=1,2$, the fibration sequences 
	\[ (B\Gamma_{ie})\pcom \Right2{} \4{B\Gamma_i} \Right2{} B\pi 
	\qquad\textup{and}\qquad 
	(BG_{ie})\pcom \Right2{} \4{BG_i} \Right2{} B\pi \]
are fiberwise homotopy equivalent and $\Gamma_i/\Gamma_{ie}\cong\pi$. So 
if we let $\Gamma$ be the pullback 
of $\Gamma_1$ and $\Gamma_2$ over $\pi$, and set 
$\Gamma_e=\Gamma_{1e}\times\Gamma_{2e}$, then 
	\[ (B\Gamma_{e})\pcom \Right2{} \4{B\Gamma} \Right2{} B\pi 
	\qquad\textup{and}\qquad 
	(BG_{e})\pcom \Right2{} \4{BG} \Right2{} B\pi \]
are also fiberwise homotopy equivalent and $\Gamma/\Gamma_e\cong\pi$. 
Here, $\4{(-)}$ again denotes fiberwise $p$-completion over $B\pi$.
So $B\Gamma\pcom\simeq BG\pcom$, and the fusion systems are isomorphic by 
Lemma \ref{l:BG=BGamma}.

By construction, $\Gamma_e$ is the product of a semisimple algebraic group 
over $\4\F_q$ (some prime $q\ne p$) with a discrete $p$-torus. So 
$\Gamma_e$ and $\Gamma$ are linear torsion groups over $\4\F_q$, and 
$\calf=\calf_S(G)$ is LT-realizable.

\smallskip

\noindent\textbf{General case: } Now assume $G$ is arbitrary. Set 
$Z=Z(G_s)$: a finite abelian $p$-group by the assumption at the beginning 
of the proof. Set $G^*=G/Z$ and $G^*_e=G_e/Z$. By Case 3 and since 
$Z((G^*)_s)=Z(G_s/Z)=1$ (see \cite[\S\,III.9.8, Proposition 
29]{Bourbaki2-3}), there is a pair of linear torsion groups 
$\Gamma^*_e\nsg\Gamma^*$ such that $\Gamma^*/\Gamma^*_e\cong\pi$, and such 
that the fibration sequences 
	\[ (B\Gamma_{e}^*)\pcom \Right2{} \4{B\Gamma^*} \Right2{} B\pi 
	\qquad\textup{and}\qquad 
	(BG_{e}^*)\pcom \Right2{} \4{BG^*} \Right2{} B\pi \]
are fiberwise homotopy equivalent. By the construction in Case 3, we can 
also assume that $\Gamma_e^*$ is a product of a semisimple algebraic 
group over $\4\F_q$ (some $q\ne p$) with a discrete $p$-torus.

The extensions $1\too Z\too G_e\too G^*_e\too1$ and $1\too 
Z\too G\too G^*\too1$ induce a commutative diagram of spaces
	\beqq \vcenter{\xymatrix@C=30pt@R=25pt{ 
	BZ \ar[r] \ar@{=}[d] & (BG_e)\pcom \ar[r] \ar[d] 
	& (BG^*_e)\pcom \ar[d] \\
	BZ \ar[r] & \4{BG} \ar[r] & \4{BG^*}.
	}} \label{e:cpt.Lie1} \eeqq
Here $BZ$ is $p$-complete since $Z$ is a finite $p$-group. So the top row 
in \eqref{e:cpt.Lie1} is a fibration sequence by \cite[Lemma II.5.1]{BK} 
and since $Z\le Z(G_e)$. Hence the bottom row is also a fibration sequence 
after fiberwise $p$-completion over $B\pi$. Similarly, if $\Gamma$ is any 
discrete group together with a surjection $\chi\:\Gamma\too\Gamma^*$ such 
that $\Ker(\chi)=Z\le Z(\chi^{-1}(\Gamma^*_e))$, and we define 
$\Gamma_e=\chi^{-1}(\Gamma^*_e)$, then we get the following commutative 
diagram whose rows are fibration sequences:
	\beqq \vcenter{\xymatrix@C=30pt@R=25pt{ 
	BZ \ar[r] \ar@{=}[d] & (B\Gamma_e)\pcom \ar[r] \ar[d] & 
	(B\Gamma^*_e)\pcom \ar[d] \\
	BZ \ar[r] & \4{B\Gamma} \ar[r]^-{B\chi} & \4{B\Gamma^*}.
	}} \label{e:cpt.Lie2} \eeqq

It remains to choose the pair $(\Gamma,\chi)$ so that the bottom rows in 
\eqref{e:cpt.Lie1} and \eqref{e:cpt.Lie2} are fiber homotopy equivalent. By 
\cite[Theorem IV.5.6]{BGM}, fibrations over a space $B$ with fiber $BZ$ are 
classified by homotopy classes of maps $B\too B\Aut(BZ)$. Let 
$\Aut_*(BZ)\subseteq\map_*(BZ,BZ)$ be the spaces of pointed self 
equivalences and pointed self maps, respectively. Since $BZ$ is an 
Eilenberg-MacLane space, we have 
	\[ \pi_i(\map_*(BZ,BZ))\cong [BZ,\Omega^i(BZ)]_* = 
	\begin{cases} 
	H^1(Z;Z) \cong \Hom(Z,Z) & \textup{if $i=0$} \\
	0 & \textup{if $i>0$.}
	\end{cases} \]
From this, together with the fibration sequence $\Aut_*(BZ)\too 
\Aut(BZ)\too BZ$, we see that $\pi_1(B\Aut(BZ))\cong\Aut(Z)$, 
$\pi_2(B\Aut(BZ))\cong Z$, and $\pi_i(B\Aut(BZ))=0$ for all $i\ge3$.

Let $[B\Gamma^*,B\Aut(BZ)]_0$ and $[\4{B\Gamma^*},B\Aut(BZ)]_0$ denote the 
sets of homotopy classes of maps for which the induced homomorphism in 
$\pi_1$ factors through $\Gamma^*/\Gamma^*_e\cong\pi$. Since 
$H^*(\4{B\Gamma^*};A)\cong H^*(B\Gamma^*;A)$ for every finite abelian 
$p$-group $A$ with action of $\pi$ (since the map 
$B\Gamma_e^*\to(B\Gamma_e^*)\pcom$ between the $\pi$-covers is a mod $p$ 
homology equivalence), the natural map 
	\beqq [\4{B\Gamma^*},B\Aut(BZ)]_0 \Right4{} [B\Gamma^*,B\Aut(BZ)]_0 
	\label{e:cpt.Lie3} \eeqq 
is a bijection. 

Since $\4{B\Gamma^*}\simeq\4{BG^*}$, there is a fibration 
$X\xto{~\nu~}\4{B\Gamma^*}$ with fiber $BZ$ that is fiberwise homotopy 
equivalent to the fibration $\4{BG}\too\4{BG^*}$. Since $Z\le Z(G_e)$, the 
fibration $\nu$ is classified by a map in $[\4{B\Gamma^*},B\Aut(BZ)]_0$. So 
by \eqref{e:cpt.Lie3}, there is also a fibration $Y\too B\Gamma^*$ with 
fiber $BZ$ whose fiberwise $p$-completion over $B\pi$ is 
$X\too\4{B\Gamma^*}$. Upon putting these together, we get the following 
commutative diagram, each of whose rows is a fibration sequence:
	\beqq \vcenter{\xymatrix@C=30pt@R=25pt{ 
	BZ \ar[r] \ar@{=}[d] & \4{BG} \ar[r] \ar[d]_{\simeq} & \4{BG^*} 
	\ar[d]_{\simeq} \\
	BZ \ar[r] & X \ar[r]^-{\nu} & \4{B\Gamma^*} \\ 
	BZ \ar@{=}[u] \ar[r] & Y \ar[u] \ar[r]^-{\nu_0} & B\Gamma^* \ar[u]. 
	}} \label{e:cpt.Lie4} \eeqq
Set $\Gamma=\pi_1(Y)$; then there is a surjection 
$\chi=\pi_1(\nu_0)\:\Gamma\too\Gamma^*$ 
with kernel $Z$. Set $\Gamma_e=\chi^{-1}(\Gamma^*_e)$; then after 
fiberwise completion we have 
$\4{B\Gamma}\simeq \4Y\simeq X \simeq \4{BG}$. It now follows that 
$B\Gamma\pcom\simeq BG\pcom$, and hence by Lemma \ref{l:BG=BGamma} that 
$\calf_S(G)\cong\calf_{S_\Gamma}(\Gamma)$. 

By construction, $Z\le Z(\Gamma_e)$, and 
	\[ \Gamma_e^* = \Gamma_e/Z = \Gamma_1^*\times\cdots\times 
	\Gamma_k^* \times T^*, \]
where $T^*$ is a discrete $p$-torus and each $\Gamma_i^*$ is a simple 
algebraic group over $\4\F_q$. By Theorems 3.1--3.3 and 4.1 in 
\cite{Steinberg}, for each $i$, the universal central extension of 
$\Gamma_i^*$ is itself an algebraic group over $\4\F_q$. (It's important 
here that we are working over an algebraic extension of a finite field.) So 
$\Gamma_e$ is a central product of simple algebraic groups over $\4\F_q$ 
and a discrete $p$-torus, where a discrete $p$-torus of rank $r$ is 
contained in the algebraic group $(\4\F_q^\times)^r$. Thus $\Gamma_e$ is 
contained in a central product of algebraic groups over $\4\F_q$, and 
this in turn is an 
algebraic group over $\4\F_q$ by \cite[Proposition 5.5.10]{Springer} (the 
quotient of a linear algebraic group by a closed normal subgroup is again a 
linear algebraic group). So $\Gamma_e$ is a linear torsion group over 
$\4\F_q$. Since $\Gamma_e$ has finite index in $\Gamma$, the group 
$\Gamma$ is also linear over $\4\F_q$, finishing the proof that 
$\calf=\calf_S(G)$ is LT-realizable.
\end{proof}


\section{Increasing sequences of finite fusion subsystems}
\label{s:incr.seq.}

In this section and the next, we show some preliminary results that will be 
needed in Section \ref{s:realiz} to prove that certain saturated fusion 
systems are not sequentially realizable. In this section, we mostly look at 
the question of how a saturated fusion system $\calf$ over an infinite 
discrete $p$-toral group $S$ is approximated by sufficiently large finite 
fusion subsystems. The following lemma is a first step towards doing that.

\begin{Lem} \label{l:autf(A)}
Let $\calf$ be a saturated fusion system over an infinite discrete 
$p$-toral group $S$, and let $T\nsg S$ be the identity component. Assume 
$C_S(T)=T$. Then there is $n\ge1$ such that $C_S(\Omega_n(T))=T$, and 
such that for each $A\le T$ that contains $\Omega_n(T)$: 
\begin{enuma} 

\item $A$ is fully centralized in $\calf$, and $\varphi(A)\le T$ 
for each $\varphi\in\homf(A,S)$; and 

\item if $A$ is $\autf(T)$-invariant, then it is weakly 
closed in $\calf$, and restriction induces an isomorphism
	$ \rho_A\: \autf(T) \Right4{\cong} \autf(A)$. 

\end{enuma}
\end{Lem}

\begin{proof} Since $\{C_S(\Omega_i(T))\}_{i\ge1}$ is a descending sequence 
of subgroups with intersection $C_S(T)=T$ (and since $S$ is artinian), 
there is $n_0\ge1$ such that $C_S(\Omega_{n_0}(T))=T$. Let $k$ be such that 
$S/T$ has exponent $p^k$; i.e., such that $s^{p^k}\in T$ for all $s\in S$. 
Set $n_1=n_0+k$. 

Let $A$ be such that $\Omega_{n_1}(T)\le A\le T$. For each morphism 
$\varphi\in\homf(A,S)$, 
	\[ \varphi(\Omega_{n_0}(T))\le\varphi(\{a^{p^k}\,|\,a\in A\})\le T, 
	\]
so $\varphi(\Omega_{n_0}(T))=\Omega_{n_0}(T)$, and $\varphi(A)\le 
C_S(\Omega_{n_0}(T))=T$. Furthermore, $C_S(\varphi(A))=T=C_S(A)$, and 
since $\varphi\in\homf(A,S)$ is arbitrary, $A$ is 
fully centralized in $\calf$. This proves (a). 

Assume in addition that $A$ is $\autf(T)$-invariant. By \cite[Lemma 
2.4(b)]{BLO3}, $\varphi$ is the restriction of some element of 
$\autf(T)$, and so $\varphi(A)=A$ by the assumption. Thus $A$ is weakly 
closed in $\calf$, and restriction defines a surjective homomorphism 
$\rho_A$ from $\autf(T)$ to $\autf(A)$.

Consider the descending sequence $\{\Ker(\rho_{\Omega_{i}(T)})\}_{i\ge 
n_1}$. This sequence is constant for $i$ large since $\autf(T)$ is 
finite, and the intersection of its terms is the trivial subgroup. So 
there is $n\ge n_1$ such that $\rho_{\Omega_n(T)}$ is injective. Hence 
$\rho_A$ is injective for every $\autf(T)$-invariant subgroup $A\le T$ 
containing $\Omega_n(T)$, and this finishes the proof of (b). 
\end{proof}

We now look more closely at unions of increasing sequences of finite 
fusion subsystems, and make more precise what we mean by a subsystem being 
``large enough''.

\begin{Lem} \label{l:autf(A)-2}
Let $\calf$ be a saturated fusion system over an infinite discrete 
$p$-toral group $S$, let $T\nsg S$ be the identity component, and assume 
$C_S(T)=T$. Let $\calf_1\le\calf_2\le\calf_3\le\cdots$ be fusion 
subsystems of $\calf$ over finite subgroups $S_1\le S_2\le S_3\le\cdots$ in 
$S$ such that $\calf=\bigcup_{i=1}^\infty\calf_i$, and set $T_i=T\cap S_i$. 
Then there is $n\ge1$ such that for each $i\ge n$ and each 
$\autf(T)$-invariant subgroup $A$ such that $T_n\le A\le T_i$, 
\begin{enuma} 

\item $A$ is weakly closed in $\calf$ and hence in $\calf_i$, 

\item restriction to $A$ induces an isomorphism 
$\rho_A\:\autf(T)\xto{~\cong~}\autf(A)$, and 

\item $\autf(A)=\Aut_{\calf_i}(A)$. 

\end{enuma}
Furthermore, $n$ can be chosen so that $T_i$ is $\autf(T)$-invariant for 
all $i\ge n$.
\end{Lem}

\begin{proof} Let $\rho_A\:\autf(T)\too\autf(A)$ be the homomorphism 
induced by restriction for each $\autf(T)$-invariant subgroup $A\le T$. By 
Lemma \ref{l:autf(A)}(b), there is $n_0\ge1$ such that for each 
$\autf(T)$-invariant subgroup $\Omega_{n_0}(T)\le A\le T$, $A$ is weakly 
closed in $\calf$ and $\rho_A$ is an isomorphism. Choose $n_1\ge1$ such 
that $T_{n_1}\ge\Omega_{n_0}(T)$; then the same conclusion holds whenever 
$T_{n_1}\le A\le T$. 

Since $\calf$ is the union of the $\calf_i$, we have 
$\autf(T_{n_1})=\bigcup_{i=n_1}^\infty\Aut_{\calf_i}(T_{n_1})$. This is an 
increasing union, and $\autf(T_{n_1})$ is finite since 
$T_{n_1}$ is finite. Hence there is $n\ge n_1$ 
such that $\autf(T_{n_1})=\Aut_{\calf_i}(T_{n_1})$ for all $i\ge n$. 

Now assume $i\ge n$ and $T_n\le A\le T_i$, where $A$ is 
$\autf(T)$-invariant. We already showed that (a) and (b) hold in this 
situation ($A$ is weakly closed in $\calf_i$ since it is weakly closed in 
$\calf$). Since $A$ and $T_{n_1}$ are both weakly closed in $\calf$ and in 
$\calf_i$, they are both fully centralized in $\calf_i$. For each 
$\alpha\in\autf(A)$, we have $\alpha|_{T_{n_1}}\in 
\autf(T_{n_1})=\Aut_{\calf_i}(T_{n_1})$, and by the extension axiom for 
$\calf_i$ (and since $A\le C_S(T_{n_1})$), this extends to 
$\alpha'\in\Aut_{\calf_i}(A)\le\autf(A)$. Since $\rho_A$ and 
$\rho_{T_{n_1}}$ are isomorphisms, restriction induces an isomorphism from 
$\autf(A)$ to $\autf(T_{n_1})$, and hence $\alpha'=\alpha$. Thus 
$\Aut_{\calf_i}(A)=\autf(A)$, and (c) holds. 

It remains to prove the last statement. Choose $m,n'\ge1$ such that 
$T_n\le\Omega_m(T)\le T_{n'}$. By Lemma \ref{l:autf(A)}, we can do this 
so that $C_S(\Omega_m(T))=T$, and so that each subgroup of $T$ containing 
$\Omega_m(T)$ is fully centralized in $\calf$. Set $A=\Omega_m(T)$, and 
note that $A$ is $\autf(T)$-invariant and hence (a), (b), and (c) 
hold for $A$. Fix $i\ge n'$ and $\alpha\in\autf(T)$. Then 
$\alpha|_A\in\autf(A)=\Aut_{\calf_i}(A)$ by (c), and since $A$ is weakly 
closed (hence fully centralized) in $\calf_i$ by (a), this extends to 
some $\beta\in\Hom_{\calf_i}(T_i,S_i)$ by the extension axiom. Also, 
$\beta(T_i)\le C_{S_i}(A)= C_{S_i}(\Omega_m(T))=T_i$. Since $T_i$ is 
fully centralized in $\calf$, $\beta$ extends to 
$\gamma\in\autf(T)$ by the extension axiom again. Then 
$\gamma|_A=\alpha|_A$, so $\alpha=\gamma$ by (b), and $\alpha(T_i)=T_i$. 
This proves that $T_i$ is $\autf(T)$-invariant for all $i\ge n'$, and so 
the last statement holds upon replacing $n$ by $n'$. 
\end{proof}

We next focus on strongly closed subgroups in fusion systems over discrete 
$p$-toral groups.

\begin{Defi} \label{d:minsc}
For each fusion system $\calf$ over a discrete $p$-toral group $S$, let 
$\calf\scl$ be the set of all \emph{nontrivial} subgroups $1\ne R\le S$ 
strongly closed in $\calf$. Set $\minsc\calf=\bigcap_{R\in\calf\scl}R$. 
\end{Defi}

Clearly, $\minsc\calf$ is always strongly closed in $\calf$. But it can be 
trivial.

\begin{Lem} \label{l:minsc}
Let $\calf$ be a saturated fusion system over a discrete $p$-toral group 
$S$, and let $\{\calf_i\}_{i\ge1}$ be an increasing sequence of saturated 
fusion subsystems of $\calf$ over $S_1\le S_2\le\cdots$ such that 
$S=\bigcup_{i=1}^\infty S_i$ and $\calf=\bigcup_{i=1}^\infty\calf_i$. Then 
there is $n\ge1$ such that $S_n\ge\Omega_1(Z(S))$ and $C_S(S_n)=Z(S)$. 
For each such $n$, we have $\minsc{\calf_i}\le\minsc{\calf_{i+1}}$ for 
all $i\ge n$, and $\minsc\calf=\bigcup_{i=n}^\infty\minsc{\calf_i}$. 
\end{Lem}

\begin{proof} Since $\Omega_1(Z(S))$ is finite, we have 
$S_n\ge\Omega_1(Z(S))$ for $n$ large enough. Since 
$Z(S)=\bigcap_{i=1}^\infty C_S(S_i)$, and $S$ is artinian, we have 
$C_S(S_n)=Z(S)$ for $n$ large enough.

Now fix $n\ge1$ such that $S_n\ge\Omega_1(Z(S))$ 
and $C_S(S_n)=Z(S)$. It will be convenient to set $\calf_\infty=\calf$ and 
$S_\infty=S$, and then refer to indices $n\le i\le\infty$. We first claim 
that for each $n\le j<i\le\infty$, 
	\beqq S_j\ge \Omega_1(Z(S_i)), \quad
	\calf_j\scl \supseteq \{ R\cap S_j \,|\, R\in\calf_i\scl \}, 
	\quad\textup{and}\quad \minsc{\calf_j}\le\minsc{\calf_i}. 
	\label{e:minsc-1} \eeqq
The first statement holds since $S_j\ge 
S_n\ge\Omega_1(Z(S))=\Omega_1(C_S(S_i))\ge\Omega_1(Z(S_i))$. For each 
$R\in\calf_i\scl$, the subgroup $R\cap S_j$ is strongly closed in $\calf_j$ 
since $R$ is strongly closed in $\calf_i$, and $R\cap S_j\ge 
R\cap\Omega_1(Z(S_i))\ne1$ where the last inequality holds by Lemma 
\ref{l:P^Z(S)} and since $R\nsg 
S_i$. This proves the second statement, and the third holds since 
	\[ \minsc{\calf_i} = \bigcap\nolimits_{R\in\calf_i\scl}R 
	\ge \bigcap\nolimits_{R\in\calf_i\scl}(R\cap S_j) \ge 
	\bigcap\nolimits_{R\in\calf_j\scl}R = \minsc{\calf_j}. \]

By \eqref{e:minsc-1} when $i=\infty$, we have 
$\bigcup_{i=n}^\infty\minsc{\calf_i}\le\minsc\calf$, and it 
remains only to show that this is an equality. 
For each $x\in\bigcup_{i=n}^\infty\minsc{\calf_i}$ and each $y\in x^\calf$, 
there is $i\ge n$ such that $x\in\minsc{\calf_i}$, $x,y\in S_i$, 
and $y\in x^{\calf_i}$. Then $y\in\minsc{\calf_i}$ since that subgroup is 
strongly closed. Thus $\bigcup_{i=n}^\infty\minsc{\calf_i}$ is 
also strongly closed in $\calf$. So either $\minsc\calf$ is the union of 
the $\minsc{\calf_i}$, or $\minsc{\calf_i}=1$ for all $n\le i<\infty$ while 
$\minsc\calf\ne1$. It remains to show that this last situation cannot 
occur. 

For the rest of the proof, we assume that $\minsc\calf\ne1$, while 
$\minsc{\calf_i}=1$ for all $n\le i<\infty$. Set $U=\Omega_1(Z(S))$ for 
short. Thus $\minsc\calf\cap U\ne1$ since 
$\minsc\calf\nsg S$ (see Lemma \ref{l:P^Z(S)}). For each $i\ge1$ and $V\le 
U$, let $\scrc_i^{(V)}$ be the set of all $R\in\calf_i\scl$ such that 
$R\cap U=V$. Thus $\calf_i\scl$ is the (finite) union of the 
$\scrc_i^{(V)}$ for $V\subseteq U$. Also, 
	\beqq R\in\scrc_i^{(V)} ~ \implies ~ 
	R\cap S_j\in\scrc_j^{(V)} \quad \forall\,n\le j<i 
	\qquad\textup{and}\qquad \scrc_i^{(1)}=\emptyset \quad 
	\forall\,i\ge n\,: \label{e:minsc-2} \eeqq
the implication holds by \eqref{e:minsc-1} and the second statement since 
for each $1\ne P\nsg S_i$, 
	\[ P\cap U = P\cap\Omega_1(Z(S)) = P\cap\Omega_1(C_S(S_i)) 
	\ge P\cap\Omega_1(Z(S_i))\ne1. \]

Set $W=\minsc\calf\cap U\ne1$. If, for some finite $i\ge n$, we have $R\ge 
W$ for each $R\in\calf_i\scl$, then $\minsc{\calf_i}\ge W\ne1$, 
contradicting our assumption. Hence for each $n\le i<\infty$, there is 
$V_i\le U$ such that $V_i\ngeq W$ and $\scrc_i^{(V_i)}\ne\emptyset$. Since 
$U$ has only finitely many subgroups, there is $V\le U$ such 
that $V\ngeq W$ and $\scrc_i^{(V)}\ne\emptyset$ for infinitely many $i$, 
and hence by \eqref{e:minsc-2} for all $m\le i<\infty$ (some $m\ge n$). By 
\eqref{e:minsc-2} again, $V\ne1$. 

For this fixed subgroup $V$, the $\scrc_i^{(V)}$ form an inverse system of 
sets, nonempty for all $i\ge m$, where $\scrc_i^{(V)}\too\scrc_{i-1}^{(V)}$ 
sends $R$ to $R\cap S_{i-1}$. Also, $\scrc_i^{(V)}$ is finite for each $i$ 
since its members are subgroups of the finite $p$-group $S_i$. So the 
inverse limit is nonempty: there is a sequence of subgroups $R_m\le 
R_{m+1}\le\cdots$ with $R_i\in\scrc_i^{(V)}$ for each $i$. Set 
$R=\bigcup_{i=m}^\infty R_i$. Then $R\cap U=V$, and $R$ is strongly closed 
in $\calf_i$ for each $m\le i<\infty$ and hence in $\calf$. In particular, 
$R\in\calf\scl$ and $R\ngeq\minsc\calf$, a contradiction. 
\end{proof}

The first part of the following lemma implies that if $Q$ is a finite 
strongly closed subgroup in a saturated fusion system $\calf$ over a 
discrete $p$-toral group $S$, then there is a morphism of fusion systems 
from $\calf$ onto $\calf/Q$. For fusion systems over finite $p$-groups, 
this is originally due to Puig \cite[Proposition 6.3]{Puig}, and the proof 
below is based on that of \cite[Theorem II.5.14]{Craven}.  

\begin{Lem} \label{l:F/st.cl.}
Let $\calf$ be a saturated fusion system over a discrete $p$-toral group 
$S$, and let $Q\nsg S$ be a finite subgroup strongly closed in $\calf$. 
\begin{enuma} 

\item For each $P,R\le S$ and each $\varphi\in\homf(P,R)$, there is 
$\5\varphi\in\homf(PQ,RQ)$ such that $\varphi(g)\equiv\5\varphi(g)$ (mod 
$Q$) for all $g\in P$. 

\item If $R\le S$ is strongly closed in $\calf$, then $RQ/Q$ is strongly 
closed in $\calf/Q$. Conversely, if $R\ge Q$ and $R/Q$ is strongly closed 
in $\calf/Q$, then $R$ is strongly closed in $\calf$. In particular, if 
$R\le S$ is strongly closed in $\calf$, then so is $RQ$.

\end{enuma}
\end{Lem}

\begin{proof} \textbf{(a) } By Alperin's fusion theorem (Theorem 
\ref{t:AFT}), it suffices to prove this when $\varphi\in\autf(P)$ and $P$ 
is fully normalized in $\calf$. We will do so by induction on $|P\cap Q|$. 
If $P\ge Q$, then there is nothing to prove, so we assume $P\ngeq Q$ and 
hence $|P\cap Q|<|Q|$.

Set 
	\[ K = \Ker\bigl[ \autf(P) \Right2{} \Aut(PQ/Q) \bigr] 
	\qquad\textup{and}\qquad
	N_S^K(P)=\{g\in N_S(P) \,|\, c_g^P\in K\}. \] 
Set $P_1=PN_S^K(P)$. We will show that $|P_1\cap Q|>|P\cap Q|$, and also 
that there is $\varphi_1\in\autf(P_1)$ such that 
$\varphi_1(g)\equiv\varphi(g)$ (mod $Q$) for each $g\in P$. Since $Q$ is 
finite, we can continue this procedure, and after finitely many steps 
construct a morphism $\5\varphi\in\autf(PQ)$ whose restriction to $P$ is 
congruent to $\varphi$ modulo $Q$. 

Now, $PQ>P$ since $P\ngeq Q$, so $PN_Q(P)=N_{PQ}(P)>P$ (see \cite[Lemma 
1.8]{BLO3}), and hence $N_Q(P)\nleq P$. Choose $x\in 
N_Q(P)\sminus P$; then $c_x^P\in K$, and so $x\in(P_1\cap Q)\sminus P$. 
This proves that $P_1\cap Q>P\cap Q$.

Since $K\nsg\autf(P)$, we have $K\cap\Aut_S(P)\in\sylp{K}$. So by the 
Frattini argument (see \cite[Theorem 1.3.7]{Gorenstein}), 
	\[ \autf(P) = K\cdot N_{\autf(P)}(K\cap\Aut_S(P)). \]
Thus $\varphi=\chi\psi$, where $\chi\in K$ and $\psi\in 
N_{\autf(P)}(K\cap\Aut_S(P))$. In particular, $\psi(g)\equiv\varphi(g)$ 
(mod $Q$) for each $g\in P$. By the extension axiom, $\psi$ extends to some 
$\varphi_1\in\autf(P_1)$ (recall $P_1=PN_S^K(P)$), which is what we needed 
to show. 

\smallskip

\noindent\textbf{(b) } If $R\le S$ and $RQ/Q$ is not strongly closed in 
$\calf/Q$, then there are elements $x\in R$ and $y\in S\sminus RQ$, 
and a map $\varphi\in\Hom_{\calf/Q}(\gen{xQ},\gen{yQ})$ that sends $xQ$ 
to $yQ$. Hence there is $\5\varphi\in\homf(Q\gen{x},Q\gen{y})$ such that 
$\5\varphi(x)\in yQ$, and in particular, $\5\varphi(x)\notin R$. So $R$ is 
not strongly closed in $\calf$.

Conversely, assume $R\ge Q$ is not strongly closed in $\calf$. Thus there 
are elements $x\in R$ and $y\in x^\calf\sminus R$, and 
$\varphi\in\homf(\gen{x},\gen{y})$ with $\varphi(x)=y$. By (a), there is 
$\psi\in\homf(Q\gen{x},Q\gen{y})$ with $\psi(x)\in yQ$, and in particular, 
$\psi(x)\notin R$. Then $(\psi/Q)(xQ)=yQ$ in $\calf/Q$, so $R/Q$ is not 
strongly closed in $\calf/Q$. 

The last statement ($R$ strongly closed implies $RQ$ strongly closed) 
follows immediately from the first two.
\end{proof}

In the following lemma, we show among other things that under certain 
conditions on a fusion system over $S$, the strongly closed subgroups 
properly contained in $S$ are all finite of bounded order.

\begin{Lem} \label{l:no.str.cl.}
Let $\calf$ be a saturated fusion system over an infinite discrete 
$p$-toral group $S$, and let $T\nsg S$ be the identity component. Assume 
that $S>T$ and $C_S(T)=T$. 
\begin{enuma} 

\item Assume there are elements $s_1,\dots,s_k\in S\sminus T$ such that 
$s_i^\calf\cap T\ne\emptyset$ for each $i$ and $S=T\gen{s_1,\dots,s_k}$. 
Then there is no proper subgroup $R<S$ containing $T$ that is strongly 
closed in $\calf$. 

\item Assume that
\begin{enumi} 
\item there is no proper subgroup $T\le R<S$ that is 
strongly closed in $\calf$; and 
\item no infinite proper subgroup of 
$T$ is invariant under the action of $\autf(T)$. 
\end{enumi} \smallskip
Let $Q\le S$ be the subgroup generated by all proper subgroups of $S$ 
that are strongly closed in $\calf$. Then $Q\le T$, $Q$ is finite 
and strongly closed in $\calf$, and $\minsc{\calf/Q}=S/Q$. 

\end{enuma}
\end{Lem}

\begin{proof} \textbf{(a) } Let $s_1,\dots,s_k\in S\sminus T$ be as 
assumed. If $R\ge T$ and is strongly closed in $\calf$, then 
for each $1\le i\le k$, we have $s_i\in R$ since it is $\calf$-conjugate to 
an element in $T\le R$, and hence $R\ge T\gen{s_1,\dots,s_k}=S$. 

\smallskip

\noindent\textbf{(b) } Assume (i) and (ii). If $R<S$ is strongly closed in 
$\calf$, then $[R,T]\le R\cap T$ since $R,T\nsg S$, and $R\cap T$ is 
invariant under the action of $\autf(T)$ since $R$ is strongly closed. 
Also, $R\cap T<T$ since $R\ngeq T$ by (i), so $R\cap T$ is finite by (ii). 

If $R<S$ is strongly closed in $\calf$ and $R\nleq T$, then for 
$x\in R\sminus T$, we have $[x,T]\ne1$ since $C_S(T)=T$, and $[x,T]$ is 
infinitely divisible since $T$ is abelian and infinitely divisible and 
normalized by $x$. Thus $R\cap T\ge[x,T]$ is infinite, which we just saw is 
impossible. So $R\le T$, and $R=R\cap T$ is finite.

Let $\scrc$ be the set of all subgroups of $T$ strongly closed in $\calf$. 
For each $n\ge1$, let $\scrc_n$ be the set of members of $\scrc$ contained 
in $\Omega_n(T)$. Thus each set $\scrc_n$ is finite, 
$\scrc=\bigcup_{i=1}^\infty\scrc_i$ since all members of $\scrc$ are 
finite, and there are maps $\omega_n\:\scrc\too\scrc_n$ that send 
$R\in\scrc$ to $R\cap\Omega_n(T)\in\scrc_n$. If $\scrc$ is infinite, then 
we can inductively choose indices $1\le n_0<n_1<n_2<\cdots$ together 
with subgroups $R_i\in\scrc_{n_i}$ such that $\omega_{n_i}^{-1}(R_i)$ is 
infinite, and such that $\omega_{n_i}(R_{i+1})=R_{i}$ and $R_{i+1}>R_i$ for 
each $i$. But then $\bigcup_{i=1}^\infty R_i\le T<S$ is infinite and 
strongly closed in $\calf$, which we just showed is impossible. 

Thus the set $\scrc$ is finite, and hence the subgroup $Q=\gen\scrc$ is 
finite. By Lemma \ref{l:F/st.cl.}(b), $Q$ is strongly closed in $\calf$ and 
no proper nontrivial subgroup of $S/Q$ is strongly closed in $\calf/Q$. 
So $\minsc{\calf/Q}=S/Q$.
\end{proof}


\section{Large abelian subgroups of finite simple groups}
\label{s:fin.p-gp}

\newcommand{\subex}{\textbf{\underline{\textup{ex}}}}
\tdef{expt}

The main results in this section are Propositions \ref{p:ex_p}, 
\ref{p:ex_p-le3}, and \ref{p:Weyl}, which together with the classification 
of finite simple groups (CFSG) are our main tools for proving that certain 
fusion systems are exotic. In Propositions \ref{p:ex_p} and 
\ref{p:ex_p-le3}, we show that if a known finite simple group $G$ has a 
``large abelian $p$-subgroup'' (in a sense made precise in 
Definition \ref{d:largeabel}), then $G$ must be of Lie type in 
characteristic different from $p$. In Proposition \ref{p:Weyl}, we 
show that if $G$ is of Lie type in characteristic different from $p$ and 
has a large abelian $p$-subgroup $A$, then $\Aut_G(A)$ must be one of the 
groups appearing in a very short list. All of these results are summarized 
by Corollary \ref{c:Weyl}.

All results in this section are independent of CFSG.
We begin with the following definition. 

\begin{Defi} \label{d:ex_p(S)}
Fix a prime $p$. For each finite $p$-group $S$, set $\subex(S)=p^n$ 
where $n\ge1$ is the smallest integer for which there is a sequence of 
subgroups $1=S_0<S_1<S_2<\cdots<S_k=S$, all of them normal in $S$, such 
that $|S_i|=p^i$ for each $i$, and such that for each $1\le i\le k$ there 
is an element $g_i\in S_i\sminus S_{i-1}$ of order at most $p^n$. For each 
finite group $G$, set $\subex_p(G)=\subex(S)$ if $S\in\sylp{G}$. We call 
$\subex(S)$ the \emph{subexponent} of the $p$-group $S$, and call 
$\subex_p(G)$ the $p$-subexponent of $G$. 
\end{Defi}

Note that the conditions on the $S_i$ in Definition \ref{d:ex_p(S)} are 
equivalent to requiring that they form a chief series for $S$. Note also 
that $\subex(S)=1$ if and only if $S=1$, and hence $\subex_p(G)=1$ if and 
only if $p\nmid|G|$. 


We first list some of the basic properties of $\subex(-)$. All of them 
follow easily from the definition. Let $\expt(S)$ denote the 
\emph{exponent} of a finite $p$-group $S$.

\begin{Lem} \label{l:ex_p(G)}
Fix a prime $p$, and let $S$ be a finite $p$-group.
\begin{enuma} 

\item In all cases, $\subex(S)\le \expt(S)$, with equality if $S$ is abelian. 

\item If $S=T\times U$, then $\subex(S)=\max\{\subex(T),\subex(U)\}$.

\item If $T\nsg S$, then $\subex(S/T)\le\subex(S)\le 
\expt(T)\cdot\subex(S/T)$.

\item If $S=TU$ where $T\nsg S$, then 
$\subex(S)\le\max\{\expt(T),\subex(U)\}$.

\item If $S\cong D_{2^k}$, $Q_{2^k}$, or $\SD_{2^k}$ for $k\ge4$, then 
$\subex(S)=2^{k-2}$.

\end{enuma}
\end{Lem}

\begin{proof} Point (a) follows directly from Definition \ref{d:ex_p(S)}, 
and point (b) is clear. In point (e), $\subex(S)\ge 2^{k-2}$ since $S$ 
contains a unique normal subgroup of index $4$ and it is cyclic, and the 
opposite inequality is easily checked. 

The first inequality in (c) is clear. To see the second, set 
$p^n=\subex(S/T)$ and $p^m=\expt(T)$, and let 
$1=S_0<S_1<S_2<\cdots<S_\ell=T$ be an arbitrary sequence of subgroups 
normal in $S$ such that $|S_i|=p^i$. Let 
$T=S_\ell<S_{\ell+1}<\cdots<S_{\ell+k}=S$ be normal subgroups such that for 
each $i$, $|S_{\ell+i}/T|=p^i$ and $(S_{\ell+i}/T)\sminus(S_{\ell+i-1}/T)$ 
contains an element of order at most $p^n$. Then for 
each $1\le i\le \ell+k$, the set $S_i\sminus S_{i-1}$ contains an element 
of order at most $p^{m+n}$, and hence $\subex(S)\le 
p^{m+n}=\expt(T)\cdot\subex(S/T)$. 

Now assume $S=TU$ where $T\nsg S$. Let $1=S_0<S_1<\cdots<S_\ell=T$ be 
subgroups normal in $S$ such that $|S_i|=p^i$. Set $\subex(U)=p^n$, and 
let $1=U_0<U_1<\cdots<U_k=U$ be subgroups normal in $U$ such that 
$U_i\sminus U_{i-1}$ contains an element of order at most $p^n$ for each 
$1\le i\le k$. Set $S_{\ell+i}=TU_i$ for each $1\le i\le k$. For each such 
$i$, the subgroup $S_{\ell+i}$ contains $T$ and is normalized by $U$, and 
hence is normal in $S=TU$. Also, $U_i=U_{i-1}\gen{x_i}$ for some $x_i\in 
U_i$ of order at most $p^n$, so 
$S_{\ell+i}=TU_i=TU_{i-1}\gen{x_i}=S_{\ell+i-1}\gen{x_i}$. Finally, 
$|S_i|\le p^i$ for each $i$, so upon removing duplicated terms, we get a 
sequence of the form given in Definition \ref{d:ex_p(S)}, proving that 
$\subex(S)\le\max\{\expt(T),\subex(U)\}$.
\end{proof}

The main property of these functions $\subex(-)$ that we need here is the 
following:

\begin{Prop} \label{p:ex_p}
Fix a prime $p$, let $S$ be a finite nonabelian $p$-group, and assume 
$A\nsg S$ is a normal abelian subgroup. 
\begin{enuma}

\item If $\subex(S)=p^n$, then there is $x\in S\sminus A$ of order at 
most $p^n$ such that $[x,S]\le\Omega_n(A)$. 

\item If $1\le k\le m$ are such that $C_S(\Omega_k(A))=A$ and $\Omega_m(A)$ 
is homocyclic of exponent $p^m$ (i.e., every element of $\Omega_1(A)$ is a 
$p^{m-1}$-st power in $A$), then $\subex(S)\ge p^{m-k+1}$.

\end{enuma}
\end{Prop}

\begin{proof} \textbf{(a) } By assumption, there is a sequence 
$1=S_0<S_1<\cdots<S_k=S$ of subgroups normal in $S$, where $|S_i|=p^i$ 
for each $i=1,\dots,k$, and where there is an element of 
order at most $p^n$ in $S_i\sminus S_{i-1}$ for each $i=1,\dots,k$. Let 
$\ell\le k-1$ be such that $S_\ell\le A$ but $S_{\ell+1}\nleq A$ (recall 
that $S$ is nonabelian, so $S\ne A$). Then 
$\expt(S_\ell)=\subex(S_\ell)\le p^n$ by Lemma \ref{l:ex_p(G)}(a) and since 
$S_\ell$ is abelian, and hence $S_\ell\le\Omega_n(A)$. 

Now, $S_{\ell+1}/S_\ell$ is normal of order $p$ in $S/S_\ell$, hence lies in 
$Z(S/S_\ell)$, and so $[S,S_{\ell+1}]\le S_\ell$. By assumption, there is 
$x\in S_{\ell+1}\sminus S_\ell$ of order at most $p^n$. Then $x\notin A$ 
since $S_{\ell+1}\cap A=S_\ell$, and $[x,S]\le[S_{\ell+1},S]\le S_\ell\le 
\Omega_n(A)$. 

\smallskip

\noindent\textbf{(b) } For each $x\in S\sminus A$, since 
$C_S(\Omega_k(A))=A$, there is $a_0\in\Omega_k(A)$ such that $[a_0,x]\ne1$. 
Choose $a\in A$ such that $a^{p^{m-k}}=a_0$; then $[a,x]$ has order at 
least $p^{m-k+1}$ and lies in $[A,x]$. Hence $[A,x]\nleq\Omega_{m-k}(A)$, 
so $\subex(S)\ge p^{m-k+1}$ by (a). 
\end{proof}

We will use $\subex_p(-)$ to characterize certain simple groups of Lie 
type in characteristic different from $p$. The next proposition is the key 
to doing this.

\begin{Prop} \label{p:ex_p-le3}
Fix a prime $p$. Let $G$ be an alternating group, a simple group of Lie 
type in defining characteristic $p$, the Tits group $\lie2F4(2)'$ (if 
$p=2$), or a sporadic simple group. Then $\subex_p(G)\le p^3$, and 
$\subex_p(G)\le p^2$ if $p$ is odd. 
\end{Prop}

\begin{proof} We consider the three cases separately. Fix $S\in\sylp{G}$. 

\smallskip

\noindent\textbf{Case 1: } Assume $G=A_n$ and $S\in\sylp{G}$. If $p$ is 
odd, then $S$ is a product of iterated wreath products $C_p\wr\cdots\wr 
C_p$, so $\subex_p(A_n)=\subex(S)=p$ by Lemma \ref{l:ex_p(G)}(b,d). If 
$p=2$, then $S\cong E\rtimes T$ where $E$ is elementary abelian of rank 
$[n/2]-1$ and $T$ is a product of iterated wreath products $C_2\wr\cdots\wr 
C_2$, so $\subex_2(S)=2$ by Lemma \ref{l:ex_p(G)}(b,d) again. 

\smallskip

\noindent\textbf{Case 2: } If $S\in\sylp{G}$ where $G$ is a finite group of 
Lie type in defining characteristic $p$, then $S$ has a normal series where 
each term is the semidirect product of the previous term with a root group. 
This follows from \cite[Theorem 3.3.1]{GLS3} when $G$ is a Chevalley group 
or a Steinberg group; from \cite[(3.4,3.8)]{Ree} when $G\cong\lie2F4(q)$ 
(and $p=2$); and holds when $G\cong\Sz(q)$ ($p=2$) or $G\cong\lie2G2(q)$ 
($p=3$) since $S$ is a root group in those cases. By \cite[Table 
2.4]{GLS3}, these root groups are all elementary abelian, except when $G$ 
is an odd dimensional unitary group or a Suzuki or Ree group in which case 
they can have exponent $p^2$. So $\subex_p(G)\le p^2$ in all cases 
(and $\subex_p(G)=p$ if $G$ is a Chevalley group, or a Steinberg group 
other than an odd dimensional unitary group).

By Lemma 1 and Section 3 in \cite{Parrott}, the Tits group $\lie2F4(2)'$ 
contains a Sylow 2-subgroup of the form $T=J\gen{x}$, where $J$ is an 
extension of $E=[J,J]\cong E_{2^5}$ by $E_{2^4}$ and $|x|=4$. So 
$\subex_2(\lie2F4(2)')\le4$. 

\smallskip

\noindent\textbf{Case 3: } In Table \ref{tbl:subex-spor}, for each sporadic 
group $G$, and each prime $p$ such that $S\in\sylp{G}$ is neither 
elementary abelian nor extraspecial of exponent $p$, we give an upper bound 
for $\subex_p(G)$, based on a chosen subgroup $H_p\le G$ containing $S$. 
These subgroups are all listed in the Atlas \cite{atlas}, and more precise 
references in many cases are given in Tables 2.1--2.2 and 
3.2--3.3 in \cite{sportame}. In most cases, the bound follows immediately from 
the description of $H_p$ together with Lemma \ref{l:ex_p(G)}. 

By \cite[\S\,5]{Janko-J3} or \cite[\S\,13]{A-overgp}, if $G=J_3$ and $p=3$, 
then $\Omega_1(S)\cong E_{3^3}$ and $S/\Omega_1(S)\cong E_9$, so 
$\subex_3(J_3)=9$. If $G=F_3$ and $p=3$, then by 
\cite[14.2(1,5)]{A-overgp}, $S$ is an extension of $E_{3^5}$ by 
$C_3\times(C_3\wr C_3)$, so $\subex_3(F_3)\le9$ by Lemma 
\ref{l:ex_p(G)}(c). If $G=\Fi_{23}$, then $H_3\cong 
O_8^+(3){:}3$ where a Sylow 3-subgroup of $O_8^+(3)$ has exponent at most $9$ 
by Lemma \ref{l:srk(GLn)}(b), so $\subex_3(\Fi_{23})\le9$ by Lemma 
\ref{l:ex_p(G)}(d).

If $G=F_5$ and $p=2$, then $H_2$ is an extension of the group $2^{1+8}_+$ 
of exponent $4$ by $A_5\wr C_2$. Since $\subex_2(A_5\wr C_2)=2$, we have 
$\subex_2(F_5)\le8$. 
\end{proof}

\begin{table}[ht] 
\begin{center}
\renewcommand{\arraystretch}{1.2}
\newcommand{\Sm}[1]{\text{\Small{$#1$}}}
\[ \begin{array}{c|cccccccccc}
G & M_{11} & M_{12} & M_{22} & M_{23} & M_{24} & J_1 & J_2 & J_3 & J_4 & 
\Co_3 \\
\hline
\subex_2(G) & 4 & \le4 & 2 & 2 & 2 & 2 & \le4 & \le4 & 2 & \le4 \\
H_2 & \Sm{\SD_{16}} & \Sm{4^2{:}D_{12}} & \Sm{2^4{:}A_6} & \Sm{2^4{:}A_7} & 
\Sm{2^4{:}A_8} & \Sm{2^3} & \Sm{2^{2+4}{:}S_3} & \Sm{2^{2+4}{:}S_3} & 
\Sm{2^{11}{:}M_{24}} & \Sm{2^4{\cdot}A_8} \\ 
\subex_3(G) & 3 & 3 & 3 & 3 & 3 & 3 & 3 & 9 & 3 & 3 \\
H_3 & \Sm{3^2} & \Sm{3^{1+2}_+} & \Sm{3^2} & \Sm{3^2} & \Sm{3^{1+2}_+} & 
\Sm{3} & \Sm{3^{1+2}_+} & \Sm{3^3.3^2} & \Sm{3^{1+2}_+} & 
\Sm{3^5{:}M_{11}} 
\end{array} \] 

\[ \begin{array}{c|ccccccccc}
G & \Co_2 & \Co_1 & \HS & \McL & \Suz & \He & \Ly & \Ru & \ON \\
\hline
\subex_2(G) & 2 & 2& 4 & 2 & \le4 & 2 & \le4 & \le4 & \le8 \\
H_2 & \Sm{2^{10}{:}M_{22}{:}2} & \Sm{2^{11}{:}M_{24}} & \Sm{4^3{:}L_3(2)} & 
\Sm{M_{22}} & \Sm{2^{4+6}{:}3A_6} & \Sm{2^6{:}3S_6} & \Sm{2A_{11}} & 
\Sm{2^{3+8}{:}L_3(2)} & \Sm{4^3{\cdot}L_3(2)} \\ 
\subex_3(G) & 3 & 3 & 3 & 3 & 3 & 3 & 3 & 3 & 3 \\
H_3 & \Sm{\McL} & \Sm{3^6{:}2M_{12}} & \Sm{S_8} & \Sm{3^4{:}A_6} & 
\Sm{3^5{:}M_{11}} & \Sm{3^{1+2}_+} & \Sm{3^5{:}M_{11}} & \Sm{3^{1+2}_+} & 
\Sm{3^4} \\ 
\subex_5(G) &5 & 5 & 5 & 5 & 5 & 5 & 5 & 5 & 5 \\
H_5 & \Sm{5^{1+2}_+} & \Sm{5^3{:}A_5} & \Sm{5^{1+2}_+} & \Sm{5^{1+2}_+} & 
\Sm{5^2} & \Sm{5^2} & \Sm{G_2(5)} & \Sm{5^{1+2}_+} & \Sm{5}  
\end{array} \] 

\[ \begin{array}{c|ccccccc}
G & \Fi_{22} & \Fi_{23} & \Fi_{24}' & F_5 & F_3 & F_2 & F_1 \\
\hline
\subex_2(G) & 2&\le4&\le4&\le8&\le4&\le8&\le8 \\
H_2 & \Sm{2^{10}{:}M_{22}} & \Sm{2^{11}{\cdot}M_{23}} & 
\Sm{2^{11}{\cdot}M_{24}} & \Sm{2^{1+8}_+.(A_5\wr2)} & 
\Sm{2^5{\cdot}L_5(2)} & \Sm{2^{1+22}_+{\cdot}\Co_2} & 
\Sm{2^{1+24}_+{\cdot}\Co_1} \\ 
\subex_3(G) & 3&\le9&\le9&3&\le9&\le9&\le9 \\
H_3 & \Sm{O_7(3)} & \Sm{O_8^+(3){:}3} & \Sm{3^7{\cdot}O_7(3)} & 
\Sm{3^4{:}2(A_4\times A_4)} & \Sm{3^5.3^4.\GL_2(3)} & 
\Sm{\Fi_{23}} & \Sm{3^8{\cdot}O_8^-(3)} \\ 
\subex_5(G) & 5&5&5&5&5&5&5 \\
H_5 & \Sm{5^2} & \Sm{5^2} & \Sm{5^2} & \Sm{5^{1+4}_+{:}(2^{1+4}_-.5)} 
& \Sm{5^{1+2}_+} & \Sm{F_5} & \Sm{5^{1+6}_+{:}4J_2} \\
\subex_7(G) & 7&7&7&7&7&7&7 \\
H_7 & \Sm{7} & \Sm{7} & \Sm{7^{1+2}_+} & \Sm{7} & \Sm{7^2} & 
\Sm{7^2} & \Sm{7^{1+4}_+{:}2S_7} 
\end{array} \] 
\end{center}
\caption{Upper bounds for $\subex_p(G)$ for sporadic groups $G$. In each 
case, $H_p$ is a subgroup of $G$ of index prime to $p$, except when in 
brackets in which case it is some group whose Sylow $p$-subgroups are 
isomorphic to those of $G$. The groups are described using \cite{atlas} 
notation; in particular, $H{:}K$ and $H{\cdot}K$ denote split and nonsplit 
extensions, respectively.} 
\label{tbl:subex-spor}
\end{table}


Propositions \ref{p:ex_p}(b) and \ref{p:ex_p-le3} motivate the following 
definition.

\begin{Defi} \label{d:largeabel}
A \emph{large abelian subgroup} of a finite $p$-group $S$ is a normal 
abelian subgroup $A\nsg S$ such that $C_S(\Omega_2(A))=A$ and $\Omega_5(A)$ 
is homocyclic of exponent $p^5$. A large abelian $p$-subgroup of an 
arbitrary finite group $G$ is a large abelian subgroup of a Sylow 
$p$-subgroup of $G$. 
\end{Defi}

Our next main result, Proposition \ref{p:Weyl}, describes the possible 
automizers of a large abelian subgroup of a simple group $G$. Some 
technical lemmas are first needed.

\begin{Lem} \label{l:expt(G)}
Let $B\le A$ be finite abelian groups, and set 
	\[ G = \{ \alpha\in\Aut(A) \,|\, \alpha|_B=\Id, ~ [\alpha,A]\le B 
	\}. \]
Then $G$ is abelian, and $\expt(G)=\gcd(\expt(B),\expt(A/B))$.
\end{Lem}

\begin{proof} We write the groups $A$ and $B$ additively. Set 
	\[ D = \{ \varphi\in\End(A) \,|\, \Im(\varphi)\le B\le \Ker(\varphi) \} 
	\cong \Hom(A/B,B), \]
regarded as an additive group. For all $\rho,\sigma\in D$, we have 
$\rho\circ\sigma=0$, so 
$(\rho+\Id_A)\circ(\sigma+\Id_A)=(\rho+\sigma+\Id_A)$. In particular, this 
shows that $\rho+\Id_A\in G$ for all $\rho\in D$, so there is an injective 
homomorphism $\chi\:D\too G$ defined by $\chi(\rho)=\rho+\Id_A$. Also, 
$\alpha-\Id_A\in D$ for each $\alpha\in G$ by definition of $G$, so $\chi$ 
is an isomorphism, and hence 
	\beq \expt(G) = \expt(D) = \expt(\Hom(A/B,B)) = 
	\gcd(\expt(A/B),\expt(B)). \qedhere \eeq
\end{proof}

\begin{Lem} \label{l:Ai<|S}
Fix a prime $p$ and an integer $m\ge1$. Let $S$ be a finite, nonabelian 
$p$-group, and let $A\nsg S$ be a normal abelian subgroup such that 
\begin{enumi} 

\item $A=C_S(\Omega_m(A))$, and 

\item $\Omega_{2m}(A)$ is homocyclic of exponent $p^{2m}$.

\end{enumi}
Then $A$ is the only normal abelian subgroup of $S$ that satisfies \rm{(i)}.
\end{Lem}

\begin{proof} Let $A^*$ be an arbitrary normal abelian subgroup of 
$S$ such that $A^*=C_S(\Omega_m(A^*))$. Set $B=A\cap A^*$. Then 
$[A,\Omega_m(A^*)]\le A\cap\Omega_m(A^*)=\Omega_m(B)$ since $A$ and 
$\Omega_m(A^*)$ are both normal. So $\Aut_{A}(\Omega_m(A^*))$ has exponent 
at most $p^m$ by Lemma \ref{l:expt(G)}, applied with 
$\Omega_m(B)\le\Omega_m(A^*)$ in the role of $B\le A$. Also, 
$\Aut_{A}(\Omega_m(A^*))\cong A/C_{A}(\Omega_m(A^*))=A/(A\cap A^*)=A/B$ by 
assumption, so $A/B$ has exponent at most $p^m$.

Since $\Omega_{2m}(A)$ is homocyclic of exponent $p^{2m}$ by (ii), every 
element of $\Omega_m(A)$ is a $p^m$-power in $A$, and hence lies in 
$B$. So $\Omega_m(A)=\Omega_m(B)\le\Omega_m(A^*)$, and hence $A= 
C_S(\Omega_m(A))\ge C_S(\Omega_m(A^*))=A^*$ by (i). So 
$\Omega_m(A)=\Omega_m(A^*)$, and hence $A=A^*$.
\end{proof}

In particular, Lemma \ref{l:Ai<|S} shows that if $A\nsg S$ is a large 
abelian subgroup, and $B\nsg S$ is such that $C_S(\Omega_2(B))=B$, then 
$B=A$.

\begin{Lem} \label{l:CS(Omega1)}
Let $T$ be a discrete $p$-torus, and let $G\le\Aut(T)$ be a finite group of 
automorphisms. Then $G$ acts faithfully on $\Omega_1(T)$ if $p$ is odd, 
and on $\Omega_2(T)$ if $p=2$.
\end{Lem}

\begin{proof} Set $q=p^k$, where $k=1$ if $p$ is odd and $k=2$ if $p=2$. 
Let $\theta\:\Aut(T)\too\Aut(\Omega_k(T))$ be the homomorphism induced by 
restriction; we must show that $\theta|_G$ is injective. Upon identifying 
$\Aut(T)=\GL_n(\Z_p)$ and $\Aut(\Omega_k(T))=\GL_n(\Z/q)$, we see that it 
suffices to show that the multiplicative group $\Ker(\theta)=I+qM_n(\Z_p)$ 
contains no nonidentity elements of finite order. Note that this is not 
true when $q=2$, since $-I\in I+2M_n(\Z_p)$. 

Assume otherwise: let $0\ne X\in M_n(\Z_p)$ and $n>1$ be such that 
$(I+qX)^n=I$. Let $i\ge k$ and $Y\in M_n(\Z_p)$ be such that $qX=p^iY$ and 
$p\nmid Y$. Write $n=p^jm$ where $p\nmid m$. Then
	\begin{align*} 
	I = (I+qX)^n &= (I+p^iY)^{p^jm} = I+p^{i+j}mY + 
	p^{2i+j}m\bigl(\tfrac{n-1}2\bigr)Y^2 + \dots \\
	&\equiv I+p^{i+j}mY \pmod{p^{i+j+1}M_n(\Z_p)}, 
	\end{align*} 
which is impossible. (This argument does not work when $p=2$ and 
$i=1$, since $p^{2i+j}=p^{i+j+1}$ and the factor $\frac{n-1}2$ need not be in 
$\Z_p$.) 
\end{proof}

\begin{Not} \label{d:ST}
For all $4\le i\le37$, let $\ST{i}$ denote the $i$-th group in the 
Shephard-Todd list \cite[Table VII]{ST} of irreducible unitary reflection 
groups. For each prime $p$, each $k\mid m\mid(p-1)$, and each $n\ge2$, 
set 
	\begin{multline*} 
	G(m,k,n) = \Gen{ \diag(u_1,\dots,u_n) \,\big|\, 
	u_1^m=\cdots=u_n^m=1 ,~ (u_1u_2\dots u_n)^{m/k}=1 } 
	\, \textup{Perm}_n \\
	\le \GL_n(\F_p);  
	\end{multline*}
i.e., the group of all monomial matrices in $\GL_n(\F_p)$ whose nonzero 
entries are $m$-th roots of unity with product an $(m/k)$-th root of unity.
\end{Not}

Thus for each $m$, $k$, and $n$ as above, $G(m,k,n)$ is normal of index $k$ 
in $G(m,1,n)\cong C_m\wr\Sigma_n$.

We are now ready to apply Lemma \ref{l:Ai<|S} to show that if a finite 
simple group $G$ of Lie type contains a large abelian $p$-subgroup 
$A$, then its automizer is one of the groups on a very short list. 

\newcommand{\cptG}{\textup{\textbf{G}}}
\newcommand{\cptH}{\textup{\textbf{H}}}
\newcommand{\cptT}{\textup{\textbf{T}}}

\begin{Prop} \label{p:Weyl}
Fix a prime $p$, and let $G$ be a finite simple group of Lie type in 
characteristic different from $p$. Assume 
$G$ has a large abelian $p$-subgroup $A\le G$. 
Set $k=2$ if $p=2$, or $k=1$ if $p$ is odd. Then 
$(\Aut_G(A),\Omega_k(A))\cong(W,M)$ where either 
\begin{enuma} 

\item $W\cong\Aut_{\cptG}(\cptT)$ for some simple compact connected Lie 
group $\cptG$ with maximal torus $\cptT$, and $M$ is the group of elements 
of order dividing $p^k$ in $\cptT$; or

\item $W\cong G(m,1,n)\cong C_m\wr\Sigma_n$ where $3\le m\mid(p-1)$, $n\ge 
p$, and $M\cong(\F_p)^n$ has the natural action of $W\le\GL_n(p)$; or 

\item $W\cong G(2m,2,n)$ where $2\le m\mid(p-1)/2$, $n\ge p$, and 
$M\cong(\F_p)^n$ has the natural action of $W\le\GL_n(p)$; or 

\item $p=3$, $W\cong\ST{12}\cong\GL_2(3)$, and $M\cong(\F_3)^2$ has the 
natural action of $W$; or 

\item $p=5$, $W\cong\ST{31}$ (an extension of the form 
$(C_4\circ2^{1+4}).\Sigma_6$), and $M$ is the unique faithful 
$\F_5W$-module with $\rk(M)=4$. 

\end{enuma} 
\end{Prop}

\begin{proof} Let $S\in\sylp{G}$ be such that $A\nsg S$ is a large abelian 
subgroup. Thus $A=C_S(\Omega_2(A))$, and $\Omega_5(A)$ is homocyclic of 
exponent $p^5$.

By definition, and since $G$ is a finite group of Lie type in 
characteristic $r$ for some prime $r\ne p$, there is a simple algebraic 
group $\4G$ over $\4\F_r$, and an algebraic 
endomorphism $\sigma\in\End(\4G)$, such that $C_{\4G}(\sigma)$ is finite 
(the group of elements fixed by $\sigma$) and $G\cong 
O^{r'}(C_{\4G}(\sigma))$. See, e.g., \cite[\S\,2.2]{GLS3} for more detail. 
From now on, we identify $G$ with $O^{r'}(C_{\4G}(\sigma))$. By 
\cite[Theorem 4.10.2(a,b)]{GLS3}, there is a maximal torus $\4T\le\4G$ 
(denoted $\4T_w$ there) and a Sylow $p$-subgroup $S\in\sylp{G}$ such that 
$\sigma(\4T)=\4T$ and $S\le N_{\4G}(\4T)$, and such that if we set 
$S_T=S\cap\4T$, then $S/S_T$ is isomorphic to a subgroup of the Weyl group 
of $\4G$. (The relation $S_T=S\cap\4T$ isn't stated explicitly in that 
theorem, but it appears in its proof.) 

Thus $\Aut_S(S_T)$ consists of the restrictions to $S_T$ of a finite group 
of automorphisms of a discrete $p$-torus (the $p$-power torsion in $\4T$). 
Hence it acts faithfully on $\Omega_2(S_T)$ by Lemma \ref{l:CS(Omega1)}, 
and so $C_S(\Omega_2(S_T))=C_S(\4T)=S_T$. The hypotheses of Lemma 
\ref{l:Ai<|S} thus apply with $A$ and $S_T$ in the roles of $A_1$ and $A_2$ 
(and $m=2$). So $A=S_T$ by that lemma. 

Assume $G^*$ is a finite group whose $p$-fusion system is isomorphic to 
that of $G$. Choose a fusion preserving isomophism 
$\rho\:S\xto{~\cong~}S^*$ for some $S^*\in\sylp{G^*}$, and set 
$A^*=\rho(A)\nsg S^*$. Then $A^*=C_{S^*}(\Omega_2(A^*))$ and 
$\Omega_4(A^*)$ is homocyclic of exponent $p^4$, and 
$\Aut_{G^*}(A^*)\cong\Aut_G(A)$. So the proposition holds for $G$ if it 
holds for $G^*$.

We first handle two special cases: certain Chevalley groups and Steinberg 
groups. Afterwards, we deal with the remaining 
cases: first when $p=2$ and then when $p$ is odd.

\smallskip

\noindent\textbf{Case 1: } Assume $G\cong\gg(q)$, where $q\equiv1$ (mod 
$p^k$) (and $q$ is a power of $r$). By \cite[Lemma 6.1]{BMO2}, we can 
choose $\4G=\gg(\4\F_r)$ and $\sigma$ such that \cite[Hypotheses 5.1, case 
(III.1)]{BMO2} holds. So by \cite[Lemma 5.3]{BMO2}, $\Aut_G(A)= 
\Aut_{\4G}(A)$ is isomorphic to the Weyl group $\gg$, and so 
$(\Aut_G(A),\Omega_k(A))\cong (\Aut_{\4G}(\4T),\Omega_k(O_p(\4T)))$. 


Let $\cptG$ be a maximal compact subgroup of $\gg(\C)$, let $\cptT$ be a 
maximal torus in $\cptG$, and let $O_p(\cptT)\le\cptT$ be the subgroup of 
elements of $p$-power order. Then $(\Aut_{\cptG}(\cptT),O_p(\cptT)) \cong 
(\Aut_{\4G}(\4T),O_p(\4T))$, by Theorem \ref{t:cpt.cn.Lie}, and hence 
$(\Aut_{\cptG}(\cptT),\Omega_k(O_p(\cptT)))\cong (\Aut_G(A),\Omega_k(A))$. 

\smallskip

\noindent\textbf{Case 2: } Now assume $G\cong{}^2\,\gg(q)$, where $\gg\cong 
A_n$, $D_n$, or $E_6$, and $q\equiv1$ (mod $p^k$) (and $q$ is a power of 
$r$). Again by \cite[Lemma 6.1]{BMO2}, we can choose $\4G=\gg(\4\F_r)$ and 
$\sigma$ such that \cite[Hypotheses 5.1, case (III.1)]{BMO2} holds. So by 
\cite[Lemma 5.3]{BMO2}, $\Aut_G(A)=\Aut_{W_0}(A)$, where 
$W_0=C_{W(\4G)}(\tau)$, and where $\tau\in\Aut(\4G)$ is a graph 
automorphism of order $2$ that permutes the root groups for positive 
roots. Also, $\Aut_G(A)$ acts faithfully on $\Omega_k(A)$ by Lemma 
\ref{l:CS(Omega1)}, and $\Omega_k(A)=\Omega_k(C_{\4T}(\tau))$ since 
$q\equiv1$ (mod $p^k$).

By \cite[\S\,13.3]{Carter}, $W_0$ (denoted $W^1$ in \cite{Carter}) is 
isomorphic to the Weyl group of $\hh$, where $\hh\cong B_{n/2}$ or 
$C_{(n+1)/2}$, $B_{n-1}$, or $F_4$ when $\4G$ has type $A_n$ ($n$ even or 
odd), $D_n$, or $E_6$, respectively. Furthermore, the arguments in 
\cite[\S\,13.3]{Carter} show that the pair $(W_0,C_{\4T}(\tau))$ is 
isomorphic to the Weyl group and maximal torus of $\hh(\4\F_r)$. So if we 
let $\cptH$ be a maximal compact subgroup of $\hh(\C)$ and let 
$\cptT_0\le\cptH$ be a maximal torus, then $(\Aut_G(A),\Omega_k(A))$ is 
isomorphic to $(\Aut_{\cptH}(\cptT_0),\Omega_k(O_p(\cptT_0)))$ by Theorem 
\ref{t:cpt.cn.Lie}.

\smallskip

\noindent\boldd{General case when $p=2$: } By \cite[Table 0.1]{BMO2} and 
the remark just before the table, the $2$-fusion system of $G$ is 
isomorphic to that of a group $G^*$, where either $G^*$ is one of the 
Chevalley or Steinberg groups considered in Cases 1 and 2, or $G^*\cong 
G_2(3)$ or $\lie2G2(q)$. But neither $G_2(3)$ nor $\lie2G2(q)$ satisfies 
the hypotheses of the proposition, since neither contains elements of order 
$2^4=16$.

\smallskip

\noindent\boldd{General case when $p$ is odd: } By \cite[Proposition 
1.10]{BMO2}, the $p$-fusion system of $G$ is isomorphic to that of a group 
$G^*$, where either $G^*$ is one of the groups handled in Cases 1 and 2, or 
it is one of the following groups:

\noindent\boldd{$G^*\cong\PSL_n(q)$ or $\POmega_{2n}^\pm(q)$, where $n\ge 
p$ and $q\not\equiv0,1$ (mod $p$).} By \cite[Lemma 
6.5]{BMO2}, we can choose a $\sigma$-setup for $G^*$ such that 
\cite[Hypotheses 5.1, case (III.3)]{BMO2} holds. In particular, 
$\Aut_{G^*}(A)=\Aut_{W_0}(A)$, where $W_0$ is a certain subgroup of the 
Weyl group of $\POmega_{2n}$ described, together with its action on a 
maximal torus, in \cite[Table 6.1]{BMO2}. By that table, there are 
$\mu\ge1$ and $\kappa\ge p$ such that $\rk(A)=\kappa$, and $\Aut_{W_0}(A)$ 
has index at most $2$ in a group $\Aut_{W_0^*}(A)\cong 
C_{2\mu}\wr\Sigma_\kappa\cong G(2\mu,1,\kappa)$. More precisely, from the 
proof of that lemma, one sees that $\Aut_{W_0}(A)\cong G(2\mu,1,\kappa)$ or 
$G(2\mu,2,\kappa)$. 

\noindent\boldd{$G^*\cong\lie2F4(q)$, where $p=3$.} 
By \cite[Section 1.2]{Malle}, $G^*$ contains a maximal subgroup 
$H\cong(C_{q+1}\times C_{q+1})\rtimes\GL_2(3)$, and it 
has index prime to $3$ in $G^*$. So we can take 
$A=O_3(H)$. Then $H=N_{G^*}(A)$ since it is maximal, and thus 
$\Aut_{G^*}(A)\cong\GL_2(3)$ with its natural action on 
$\Omega_1(A)\cong(\F_3)^2$. 

\noindent\boldd{$G^*\cong\lie3D4(q)$, where $p=3$.} By \cite[10-1(4)]{GL}, 
$S$ is a semidirect product $A\rtimes C_3$, where $A\cong C_{3^a}\times 
C_{3^{a+1}}$ for $a=v_3(q^2-1)$. Thus $\Aut_{G^*}(A)$ is isomorphic to a 
subgroup of $C_2\times\Sigma_3$, which we identify with the group of upper 
triangular matrices in $\GL_2(3)$. Also, by the main theorem in 
\cite{Kleidman}, $G^*$ contains a maximal subgroup 
$H\cong(C_{q^2+q+1}\circ\SL_3(q)).\Sigma_3$ (if $q\equiv1$ (mod 3)) or 
$H\cong(C_{q^2-q+1}\circ\SL_3(q)).\Sigma_3$ (if $q\equiv2$ (mod 3)), in 
either case of index prime to $3$, and using these we see that 
$|N_{G^*}(S)/S|=4$ and hence $\Aut_{G^*}(A)\cong C_2\times\Sigma_3\cong 
W(G_2)$. Since this is the normalizer of a Sylow $3$-subgroup in 
$\GL_2(3)$, its action on $\Omega_1(A)\cong(\F_3)^2$ is unique up to 
isomorphism. (The $3$-fusion system of $G^*$ is described in detail in 
case (a.ii) of \cite[Theorem 2.8]{indp1}.) 

\smallskip

\noindent\boldd{$G^*\cong E_8(q)$, where $p=5$ and $q\equiv\pm2$ (mod 
$5$).} In this case, $\Aut_{G^*}(A)\cong\ST{31}$ by \cite[Lemma 6.7]{BMO2} 
and its proof. More precisely, a subgroup $W_1\nsg \Aut_{G^*}(A)$ was 
defined in that proof and shown to be isomorphic to $C_4\circ2^{1+4}$, and 
$\Aut_{G^*}(A)/W_1$ was shown to be isomorphic to $\Sigma_6$. 

Let $\Out_0(W_1)$ be the subgroup of elements that are the identity on 
$Z(W_1)$; then $\Out_0(W_1)\cong\Sp_4(2)\cong\Sigma_6$. Since $W_1$ is 
irreducible in $\GL_4(\4\F_5)$, we have 
$C_{\GL_4(5)}(W_1)\cong\F_5^\times$, and hence $\Aut_{G^*}(A)\cong 
N_{\GL_4(5)}(W_1)$. 

Since $\GL_4(5)$ has only one conjugacy class of subgroups isomorphic to 
$W_1$, this shows that it also contains only one class of extensions of the 
form $(C_4\circ2^{1+4}).\Sigma_6$. On the other hand, Shephard \cite[p. 
275]{Shephard} constructed explicit matrices in $\GL_4(\C)$ representing 
reflections that generate $\ST{31}$, and those matrices have entries in 
$\Z[\frac12,i]$ and hence reduce to matrices over $\F_5$. Thus $\ST{31}$ 
embeds in $\GL_4(5)$, so $\Aut_{G^*}(A)\cong\ST{31}$. See also Aguad\'e's 
paper \cite[\S6]{Aguade} for more discussion. \qedhere

\end{proof}

The following corollary combines the three main propositions in this 
section. 

\begin{Cor} \label{c:Weyl}
Fix a prime $p$, and set $k=2$ if $p=2$, or $k=1$ if $p$ is odd. 
If $G$ is a known finite simple group and has a large 
abelian $p$-subgroup $A\le G$, then $G$ is a group of Lie type in 
characteristic different from $p$, and $(\Aut_G(A),\Omega_k(A))\cong(W,M)$ 
where $(W,M)$ is one of the modules listed in cases (a)--(e) of Proposition 
\ref{p:Weyl}. 
\end{Cor}

\begin{proof} Let $S\in\sylp{G}$ be such that $A\nsg S$ is a large abelian 
subgroup. Then $C_S(\Omega_2(A))=A$ and $\Omega_5(A)$ 
is homocyclic of exponent $p^5$, so by Proposition \ref{p:ex_p}(b), 
$\subex_p(G)=\subex(S)\ge p^4$. Hence by Proposition \ref{p:ex_p-le3} and 
since $G$ is a known simple group, $G$ must be of Lie type in 
characteristic different from $p$. The conclusion now follows from Proposition 
\ref{p:Weyl}.
\end{proof}

The following lemma will be useful when applying Proposition 
\ref{p:Weyl} in the proof of Theorem \ref{t:index.p}. 

\begin{Lem} \label{l:V.simple}
Fix an odd prime $p$, and let $(W,M)$ be one of the pairs that appears in 
cases {\rm(a)--(e)} of Proposition \ref{p:Weyl}. Then $M$ is a simple 
$\F_pW$-module, \emph{except} in case {\rm(a)} when $W=\Aut_G(T)$ and 
$M=\Hom(C_p,T)$, and either $G\cong\PSU(n)$ and $p\mid n$, or $G\cong 
G_2$ or $E_6$ and $p=3$.
\end{Lem}

\begin{proof} We refer to the five cases (a)--(e) listed in the statement 
of Proposition \ref{p:Weyl}. In cases (b) and (c), we have $W\cong 
G(m,k,n)$ for some $k\mid m\mid(p-1)$ and some $n\ge p$, and $M$ is the 
natural $n$-dimensional $\F_pW$-module. Let $W_0\nsg W$ be the subgroup of 
diagonal matrices (in the notation of \ref{d:ST}); then $M|_{W_0}$ splits 
in a unique way as a direct sum of $1$-dimensional submodules, and these 
summands are permuted transitively by $W/W_0\cong\Sigma_n$. So $M$ is 
simple in this case.

In case (d), we have $W\cong\GL_2(3)$, and $M\cong(\F_3)^2$ is clearly 
simple. In case (e), there is a subgroup $H\cong2^{1+4}_+$ of index $2$ 
in $O_2(W)$ whose action on $M$ is generated by the diagonal matrices
$-\Id$, $\diag(1,1,-1,-1)$, and $\diag(1,-1,1,-1)$, 
together with the permutation matrices for the permutations $(1\,2)(3\,4)$ 
and $(1\,3)(2\,4)$, and $M|_H$ is irreducible. So $M$ is also irreducible 
as an $\F_5W$-module. 

Now assume we are in case (a), where $W\cong\Aut_G(T)$ and $M$ is the 
$p$-torsion in $T$ for some simple compact connected Lie group $G$ with 
maximal torus $T$. If $G$ has type $A_{n-1}$, so $W\cong\Sigma_n$ and 
$M\cong(\F_p)^{n-1}$, then $M$ is simple except when $p\mid n$ (see 
\cite[Example 5.1]{James}). If $G$ has type $B_n$, $C_n$, or $D_n$, then 
$W\cong G(2,1,n)$ or $G(2,2,n)$, and $M$ is simple by a similar argument to 
that used in cases (b) and (c). If $G$ has type $F_4$, then $W$ contains a 
subgroup isomorphic to $G(2,1,4)$ (with the natural action on 
$M\cong(\F_p)^4$), and hence $M$ is simple. If $G$ is of type $E_n$ for 
$n=6,7,8$, then by the character tables in \cite{atlas} (when $p\nmid|W|$) 
and \cite{braueratlas} (when $p\mid|W|$), we have that $M$ is a simple 
$\F_pW$-module in all cases except when $G$ has type $E_6$ and $p=3$. 
\end{proof}


\section{Realizability of fusion systems over discrete 
\texorpdfstring{$p$-toral}{p-toral} groups}
\label{s:realiz}

We are now ready to prove our main nonrealizability results. Theorem 
\ref{t:seq-exotic} (a slightly more general version of Theorem \ref{ThC}) 
says that saturated fusion systems over infinite discrete $p$-toral 
groups satisfying certain conditions are not sequentially realizable. 
In Proposition \ref{p:no.char0}, we give a criterion for showing that 
certain fusion systems are not realized by linear torsion groups in 
characteristic $0$. All of the results in this section (except Lemmas 
\ref{l:F/Q-seq.real} and \ref{l:CT(G)}) depend on the classification of 
finite simple groups.

We first note that a quotient of a sequentially realizable fusion system 
(see Definition \ref{d:F/Q}) is again sequentially realizable.

\begin{Lem} \label{l:F/Q-seq.real}
Let $\calf$ be a sequentially realizable fusion system over a discrete 
$p$-toral group $S$, and let $Q\nsg S$ be a finite subgroup that is weakly 
closed in $\calf$. Then $\calf/Q$ is also sequentially realizable. If 
$T\nsg S$ is the identity component of $S$ and $Q\le T$, then 
$C_{S/Q}(T/Q)=C_S(T)/Q$ and $\Aut_{\calf/Q}(T/Q)\cong\autf(T)$.
\end{Lem}

\begin{proof} Since $\calf$ is sequentially realizable, there is an 
increasing sequence $\calf_1\le\calf_2\le\dots$ of realizable fusion 
subsystems of $\calf$ over finite subgroups $S_1\le S_2\le\dots$ of $S$, 
such that $S=\bigcup_{i=1}^\infty S_i$ and 
$\calf=\bigcup_{i=1}^\infty\calf_i$. Since $Q$ is finite, we have $Q\le 
S_k$ for some $k\ge1$, and by removing the first $(k-1)$ terms in the 
sequence, we can arrange that $Q\le S_i$, and hence $Q$ is weakly closed in 
$\calf_i$, for all $i\ge1$. 

For each $i$, since $\calf_i$ is realizable, there is a finite group $G_i$ 
with $S_i\in\sylp{G_i}$ and $\calf_i=\calf_{S_i}(G_i)$. Then 
$\calf_i/Q=\calf_{S_i/Q}(N_{G_i}(Q)/Q)$ by Lemma \ref{l:F(G)/Q}, where 
$S_i/Q\in\sylp{N_{G_i}(Q)/Q}$, so $\calf_i/Q$ is realizable. Also, 
$\calf/Q=\bigcup_{i=1}^\infty\calf_i/Q$, 
so $\calf/Q$ is sequentially realizable.

If $Q\le T$ where $T$ is the identity component of $S$, then by definition 
and since $Q$ is weakly closed in $\calf$, 
$\Aut_{\calf/Q}(T/Q)$ is the image of the natural 
homomorphism $\chi\:\autf(T)\too\Aut(T/Q)$. Let $m$ be such that $Q$ 
has exponent $p^m$. If $\alpha\in\Ker(\chi)$, then $\alpha(t)\in tQ$ for 
each $t\in T$, so $\alpha(t^{p^m})=t^{p^m}$, and $\alpha=\Id_T$ since $T$ 
is infinitely divisible. Thus $\Ker(\chi)=1$, and so 
$\Aut_{\calf/Q}(T/Q)\cong\autf(T)$ and $C_{S/Q}(T/Q)=C_S(T)/Q$. 
\end{proof}

The following elementary lemma will also be needed. 

\begin{Lem} \label{l:CT(G)}
Let $T$ be a discrete $p$-torus, and let $G\le\Aut(T)$ be a finite group of 
automorphisms of $T$. Set $\ell=v_p(|G|)$; thus $p^\ell\mid|G|$ but 
$p^{\ell+1}\nmid|G|$. Then either $C_T(G)$ is infinite, or it has exponent 
at most $p^\ell$.
\end{Lem}

\begin{proof} Set $P_1=C_T(G)$ for short, and set 
	\[ P_2 = \bigl\{ x\in T \,\big|\, 
	\textstyle\prod_{g\in G}g(x)=1 \bigr\}. \]
For each $x\in T$, we have $\prod_{g\in G}g(x)\in P_1$. Also, 
$x(g(x))^{-1}\in P_2$ for each $g\in G$ since $\prod_{h\in 
G}h(x)\cdot\prod_{h\in G}(hg(x))^{-1} = 1$. So 
	\[ x^{|G|} = \bigl( \textstyle\prod_{g\in G}gx \bigr) \cdot 
	\bigl( \textstyle\prod_{g\in G}x(gx)^{-1} \bigr) \in P_1P_2, \]
and $T=P_1P_2$ since it is infinitely divisible. 

If $P_1$ is finite, then $P_2$ has finite index in $T$, and hence $P_2=T$, 
again since it is infinitely divisible. Hence for each $x\in 
P_1=C_T(G)$, we have $x^{|G|}=\prod_{g\in G}gx=1$ by definition of $P_2$, 
and so $x^{p^\ell}=1$ since $T$ has no nonidentity elements of order prime 
to $p$.
\end{proof}

We next recall the definition of the \emph{generalized Fitting subgroup} 
$F^*(G)$ of a finite group $G$. A subgroup $H\le G$ is subnormal (denoted 
$H\snsg G$) if there is a sequence of subgroups $H=H_0\nsg 
H_1\nsg\cdots\nsg H_n=G$ each normal in the following one. 
A component of $G$ is a subnormal 
subgroup $C\snsg G$ that is quasisimple (i.e., $C$ is perfect and $C/Z(C)$ 
is simple). All components of $G$ commute with each other and with the 
subgroups $O_p(G)$ for all primes $p$. Thus $F^*(G)$, the subgroup 
generated by all components of $G$ and all subgroups $O_p(G)$ for primes 
$p$, is a central product of these groups. One of the 
important properties of $F^*(G)$ is that it is centric in $G$; i.e., 
$C_G(F^*(G))\le F^*(G)$. We refer to \cite[\S31]{A-FGT} for more details 
and proofs of these statements, and also to \cite[A.11--A.13]{AKO} for a 
shorter summary.

\begin{Lem} \label{l:seq-exotic}
Let $\calf$ be a saturated fusion system over an infinite discrete 
$p$-toral group $S$, and let $T\nsg S$ be the identity component. Assume 
\begin{enumi} 

\item $S>T$ and $C_S(T)=T$, and 

\item $\minsc\calf=S$.

\end{enumi}
Assume also that $\calf$ is sequentially realizable, and let 
$\calf_1\le\calf_2\le\calf_3\le\cdots$ be realizable fusion subsystems of 
$\calf$ over finite subgroups $S_1\le S_2\le\cdots$ of $S$ such that 
$S=\bigcup_{i=1}^\infty S_i$ and $\calf=\bigcup_{j=1}^\infty\calf_i$. 
If $G_1,G_2,G_3,\dots$ are finite groups such that for each $i\ge1$,
	\[ S_i\in\sylp{G_i}, \qquad O_{p'}(G_i)=1, 
	\qquad\textup{and}\qquad  \calf_i=\calf_{S_i}(G_i), \]
then for $n$ large enough, the generalized Fitting subgroup $F^*(G_n)$ is a 
product of nonabelian simple groups of Lie type in characteristic different 
from $p$ that are pairwise conjugate in $G_n$.
\end{Lem}

\begin{proof} Set $T_i=T\cap S_i$ and $Q_i=\minsc{\calf_i}\le S_i$ for each 
$i$. Then $\bigcup_{i=1}^\infty Q_i=\minsc\calf=S$ by Lemma \ref{l:minsc} 
and (ii). So we can choose $m\ge1$ such that for all $n\ge m$, we 
have 
	\beqq TQ_n=S, \qquad Q_n\ge\Omega_5(T), \qquad\textup{and}\qquad
	C_S(Q_n\cap T)=C_S(T)=T. \label{e:seq-ex} \eeqq

For the rest of the proof, we fix $n$ such that $n\ge m$ and hence 
\eqref{e:seq-ex} holds. 
If $O_p(G_n)\ne1$, then $\Omega_1(Z(O_p(G_n)))$ is nontrivial and 
strongly closed in $\calf_n$, hence contains $Q_n=\minsc{\calf_n}$, 
contradicting \eqref{e:seq-ex}. Thus $O_p(G_n)=1$, and since 
$O_{p'}(G_n)=1$ by assumption, the generalized Fitting 
subgroup $F^*(G_n)$ is a product of nonabelian simple groups. Set 
$H=F^*(G_n)$ and let $H_1,\dots,H_\ell\le H$ be its simple factors; thus 
	\[ H = F^*(G_n) = H_1 \times \cdots \times H_\ell. \]
Let $\pr_j\:H\too H_j$ be projection to the $j$-th factor.

Set $U=H\cap S_n=H\cap S\in\sylp{H}$, and set $U_j=U\cap H_j\in\sylp{H_j}$ 
for each $j$. Thus $U=U_1\times\cdots\times U_\ell$. 
Also, $U_j\ne1$ for each $j$ since otherwise $H_j\le O_{p'}(G_n)=1$. 
Finally, $U\ge Q_n=\minsc{\calf_n}$ since $U$ is strongly closed in 
$\calf_n$.

Set $A=U\cap T=H\cap T$, and set $A_j=U_j\cap T=H_j\cap T$ for each $1\le 
j\le\ell$. Then $C_U(A)\le U\cap C_S(Q_n\cap T)=A$ by \eqref{e:seq-ex}, so 
$A$ is a maximal abelian subgroup of $U$. Since 
$A\le\pr_1(A)\times\cdots\times \pr_\ell(A)\le U$, this implies that 
$A_j=\pr_j(A)$ for each $j$, and hence that $A=A_1\times\cdots\times 
A_\ell$. If $A_j=1$ for some $j$, then $U_j\le C_U(A)=A$, so $U_j=1$ which 
we just saw is impossible. Thus each factor $A_j$ is nontrivial.

For each $j$, since $1\ne A_j\le T$, we have $\Omega_1(T)\cap A_j\ne1$. 
If the simple factors $H_j$ are not permuted transitively by 
$\Aut_{G_n}(H)$, then there is a nonempty proper subset 
$J\subsetneqq\{1,\dots,\ell\}$ such that $\gen{H_j\,|\,j\in J}\nsg G_n$. 
Then $\gen{U_j\,|\,j\in J}\in\calf_n\scl$, hence contains $Q_n$ but does 
not contain $\Omega_1(T)$. This contradicts \eqref{e:seq-ex}, and we 
conclude that the $H_j$ are permuted transitively by $\Aut_{G_n}(H)$ and 
hence are pairwise conjugate in $G_n$. 

Now, 
\begin{itemize}

\item $A\nsg U$ since $T\nsg S$ and $A=U\cap T$; 

\item $\Omega_5(A)$ is homocyclic of exponent $p^5$ since $\Omega_5(T)\le 
Q_n\le U$ by \eqref{e:seq-ex} and $A=T\cap U$; and 

\item $C_U(\Omega_2(A))=A$ by Lemma \ref{l:CS(Omega1)} and since 
$\Aut_U(A)$ is the restriction of an action on $T$. 

\end{itemize}
Hence for each $j=1,\dots,\ell$, we have $A_j\nsg U_j$, $\Omega_5(A_j)$ is 
homocyclic of exponent $p^5$, and $C_{U_j}(\Omega_2(A_j))=A_j$. 
So $A_j$ is a large abelian $p$-subgroup of $H_j$, and by Corollary 
\ref{c:Weyl} and the classification of finite simple groups,
each simple factor $H_j$ is of Lie type in characteristic different 
from $p$. 
\end{proof}

We are now ready to show that under certain hypotheses on a fusion system 
$\calf$, if it is sequentially realizable, then 
the automizer of a maximal torus contains one of the groups $W$ that 
appears in cases (a)--(e) of Proposition \ref{p:Weyl} as a normal subgroup 
of index prime to $p$. 

\begin{Thm} \label{t:seq-exotic}
Let $\calf$ be a saturated fusion system over an infinite discrete 
$p$-toral group $S$, and let $T\nsg S$ be the identity component. Assume 
\begin{enumi} 

\item $S>T$ and $C_S(T)=T$; 

\item no proper subgroup $R<S$ containing $T$ is strongly closed in 
$\calf$; and 

\item no infinite proper subgroup of $T$ is invariant under the action of 
$\autf(T)$.

\end{enumi}
Assume also that $\calf$ is sequentially realizable. 
Set $k=2$ if $p=2$, or $k=1$ if $p$ is odd. Then there is a normal 
subgroup $H\nsg\autf(T)$ of index prime to $p$ such that for some 
$\ell\ge1$, $(H,\Omega_k(T))\cong(W^\ell,(M')^{\oplus\ell})$, where $(W,M)$ 
is one of the pairs listed in Proposition \ref{p:Weyl}(a--e) and $M'$ is an 
$\Z/p^kW$-module with the same composition factors as $M$. If $\ell>1$, then 
the $\ell$ factors are permuted transitively by $\autf(T)$. If {\rm(i)} and 
{\rm(ii)} hold and no infinite proper subgroup of $T$ is invariant under 
the action of $O^{p'}(\autf(T))$, then this conclusion holds with 
$\ell=1$.
\end{Thm}

\begin{proof} Let $Q\le S$ be the subgroup generated by all proper 
subgroups of $S$ strongly closed in $\calf$. By Lemma \ref{l:no.str.cl.}(b) 
and points (i)--(iii), $Q$ is finite, is contained in $T$, is strongly 
closed in $\calf$, and $\minsc{\calf/Q}=S/Q$. So by Lemma 
\ref{l:F/Q-seq.real} and since $\calf$ is sequentially realizable, 
$\calf/Q$ is also sequentially realizable, $C_{S/Q}(T/Q)=C_S(T)/Q=T/Q$, and 
$\Aut_{\calf/Q}(T/Q)\cong\autf(T)$. Also, $\calf/Q$ is saturated by Lemma 
\ref{l:F/Q}. We claim that 
	\beqq \textup{\begin{small} 
	$\Omega_k(T)$ and $\Omega_k(T/Q)$ have isomorphic 
	composition factors as $\Z/p^k\autf(T)$-modules.
	\end{small}}
	\label{e:seq-exotic1} \eeqq
Then upon replacing $\calf$ by $\calf/Q$, we can assume without 
loss of generality that $\minsc\calf=S$.

To prove \eqref{e:seq-exotic1}, first note that for each $i\ge0$, there is 
an exact sequence 
	\[ 0 \to \Omega_{i+1}(Q)/\Omega_i(Q) \too 
	\Omega_k(T/\Omega_i(Q)) \too \Omega_k(T/\Omega_{i+1}(Q))  
	\xto{~\rho_i~} \Omega_{i+1}(Q)/\Omega_i(Q) \to 0, \]
where $\rho_i(x\Omega_{i+1}(Q))=p^kx\Omega_i(Q)$. (Here, $\Omega_0(T)=1$.) So 
$\Omega_k(T/\Omega_i(Q))$ and $\Omega_k(T/\Omega_{i+1}(Q))$ have isomorphic 
composition factors for each $i\ge0$ by the Jordan-H\"older theorem, and 
since $Q=\Omega_i(Q)$ for $i$ large enough, this implies that $\Omega_k(T)$ 
and $\Omega_k(T/Q)$ have isomorphic composition factors.

Now assume that $\minsc\calf=S$. Fix realizable fusion subsystems 
$\calf_1\le\calf_2\le\dots$ of $\calf$ over finite subgroups $S_1\le 
S_2\le\dots$ of $S$ such that $\calf=\bigcup_{i=1}^\infty\calf_i$ and 
$S=\bigcup_{i=1}^\infty S_i$. Let $G_1,G_2,\dots$ be finite groups such 
that for each $i\ge1$, $O_{p'}(G_i)=1$, $S_i\in\sylp{G_i}$, and 
$\calf_i=\calf_{S_i}(G_i)$. Set $F^*_i=F^*(G_i)$, $U_i=F^*_i\cap 
S_i\in\sylp{F^*_i}$, $T_i=T\cap S_i$, and $A_i=T\cap U_i$. Also, set 
$Q_i=\minsc{\calf_i}\le S_i$ for each $i$. Note that $U_i\ge Q_i$ for each 
$i$ since $U_i$ is strongly closed in $\calf_i$.

We now fix $n_0\ge1$ such that 
\begin{enumerate}[(1) ]

\item for all $i\ge n_0$, 
$F^*_i$ is a product of nonabelian simple groups of Lie type in 
characteristic different from $p$ that are pairwise conjugate in $G_i$; 

\item $Q_i\le Q_{i+1}$ for all $i\ge n_0$, and $\bigcup_{i=n_0}^\infty 
Q_i=\minsc\calf=S$; 

\item $Q_{n_0}T=S$ and $Q_{n_0}\ge\Omega_5(T)$;

\item for all $i\ge n_0$, the subgroup $T_i$ is $\autf(T)$-invariant; and 

\item $\Aut_{G_i}(A)=\autf(A)\cong\autf(T)$ for each $i\ge n_0$ and each 
$\autf(T)$-invariant subgroup $A\le T_i$ containing $T_{n_0}$, where 
the isomorphism $\autf(T)\xto{~\cong~}\autf(A)$ is induced by restricting 
from $T$ to $A$.

\end{enumerate}
Point (1) holds for $n_0$ large enough by Lemma \ref{l:seq-exotic}, point (2) 
by Lemma \ref{l:minsc}, point (3) as a consequence of (2), and points (4) 
and (5) by Lemma \ref{l:autf(A)-2}. 

For each $i\ge n_0$, the subgroup $T_i=T\cap G_i$ is $\autf(T)$-invariant 
by (4), and hence $\Aut_{G_i}(T_i)=\autf(T_i)\cong\autf(T)$ by (5). Since 
$F^*_i\nsg G_i$, the subgroup $A_i=F^*_i\cap T_i$ is 
$\Aut_{G_i}(T_i)$-invariant, and hence is $\autf(T)$-invariant. By (2), we can 
choose $n\ge n_0$ so that $Q_{n}\ge T_{n_0}$; then $A_i\ge T_{n_0}$ for all 
$i\ge n$, and so $\Aut_{G_i}(A_i)=\autf(A_i)\cong\autf(T)$ by (5).

Now assume $i\ge n\ge n_0$. Thus $U_iT=S$ by (3) and since $U_i\ge Q_i\ge 
Q_0$. So for each $g\in N_{S_i}(A_i)$, there is $h\in U_i\cap gT$, and 
hence $c_g=c_h\in\Aut(A_i)$. Thus $\Aut_{S_i}(A_i)\le\Aut_{F^*_i}(A_i)$. 
Since $\Aut_{S_i}(A_i)\in\sylp{\Aut_{G_i}(A_i)}$, this proves that 
$\Aut_{F^*_i}(A_i)$ is normal of index prime to $p$ in $\Aut_{G_i}(A_i)$. 
Set $H_i=\Aut_{F^*_i}(A_i)$, and let $H\nsg\autf(T)$ be the 
corresponding normal subgroup of index prime to $p$ under the isomorphism 
$\Aut_{G_i}(A_i)\cong\autf(T)$. Thus 
	\[ H = \bigl\{ \alpha\in\autf(T) \,\big|\, 
	\alpha|_{A_i}\in H_i \bigr\}
	= \bigl\{ \alpha\in\autf(T) \,\big|\, 
	\alpha|_{A_i}\in\Aut_{F^*_i}(A_i) \bigr\}. \]

Still assuming $i\ge n$, since $U_i\ge Q_i\ge Q_{n_0}\ge\Omega_5(T)$, we 
have $A_i=U_i\cap T\ge\Omega_5(T)$. So $C_{U_i}(A_i)\le C_S(\Omega_2(T))=T$ 
by Lemma \ref{l:CS(Omega1)}, and hence $C_{U_i}(A_i)=U_i\cap T=A_i$. Thus 
$U_i$ is nonabelian ($U_i>A_i$ since $U_iT=S$), and $A_i$ is a maximal 
abelian subgroup of $U_i$. Also, $\Omega_5(A_i)=\Omega_5(T)$ is homocyclic 
of exponent $p^5$, and $F^*_i$ is a product of simple groups of Lie type in 
characteristic different from $p$ permuted transitively by $G_i$ by (1). So 
$A_i$ is a large abelian $p$-subgroup of $F^*_i$, and each of the simple 
factors of $F^*_i$ has some direct factor of $A_i$ as a large abelian 
$p$-subgroup. By Proposition \ref{p:Weyl}, for one of the pairs $(W,M)$ 
listed in \ref{p:Weyl}(a)--(e), we have 
$(H,\Omega_k(T))\cong(H_i,\Omega_k(A_i))\cong(W^\ell,M^{\oplus\ell})$, 
where $\ell$ is the number of simple factors in $F^*_i$. We have already 
seen that the $\ell$ factors are permuted transitively by 
$\Aut_{G_i}(A_i)\cong\autf(T)$.

It remains to prove the last statement. Assume there is no infinite 
proper subgroup of $T$ that is $O^{p'}(\autf(T))$-invariant; we must show 
that $\ell=1$. Assume otherwise: assume $\ell\ge2$. Then 
$\Aut_{F^*_i}(A_i)\ge O^{p'}(\Aut_{G_i}(A_i))$, and we can write 
$\autf(T)=W_1\times\cdots\times W_\ell$ and $A_i=M_1\times\cdots\times 
M_\ell$, where $W_j\cong W_1$ and $M_j\cong M_1$ for all $2\le j\le\ell$, 
and where $W_j$ acts trivially on $M_t$ for $t\ne j$. We can assume that 
$i$ was chosen so that $A_i\ge\Omega_m(T)$ where $m=v_p(|\autf(T)|)$; in 
particular, so that each $M_j$ has exponent at least $p^m$. Thus $M_2\le 
C_T(W_1)<T$, so $C_T(W_1)$ has exponent at least equal to 
$|\autf(T)|>|W_1|$, and hence is infinite by Lemma \ref{l:CT(G)}. Thus 
$C_T(W_1)$ is an infinite proper subgroup of $T$ invariant under the 
action of $O^{p'}(\autf(T))$, contradicting our assumption. 
\end{proof}

We now show that a fusion system that satisfies the same hypotheses 
(i)--(iii) as in Theorem \ref{t:seq-exotic} cannot be realized by a linear 
torsion group in characteristic zero. 

Recall that $\rk_q(G)$ denotes the $q$-rank of a group $G$: the least upper 
bound for ranks of finite abelian $q$-subgroups of $G$.

\begin{Prop} \label{p:no.char0}
Let $\calf$ be a saturated fusion system over an infinite discrete 
$p$-toral group $S$, and let $T\nsg S$ be the identity component. Assume 
\begin{enumi} 

\item $S>T$ and $C_S(T)=T$; 

\item no proper subgroup $R<S$ containing $T$ is strongly closed in 
$\calf$; and 

\item no infinite proper subgroup of $T$ is invariant under the action of 
$\autf(T)$. 

\end{enumi}
Then if $\calf=\calf_S(G)$ for some locally finite group $G$ with 
$S\in\sylp{G}$, there is a prime $r\ne p$ such that $\srk_r(G)=\infty$. In 
particular, $\calf$ is not isomorphic to the fusion system of any linear 
torsion group in characteristic $0$. 
\end{Prop}

\begin{proof} Assume $\calf=\calf_S(G)$, where $G$ is a locally finite 
group with $S\in\sylp{G}$. We will show that $\srk_r(G)=\infty$ for some 
prime $r\ne p$. It then follows from Lemma \ref{l:srk(GLn)}(a) that $G$ 
cannot be a subgroup of $\GL_N(K)$ for any $N\ge1$ and any field $K$ of 
characteristic $0$. 

Let $Q\nsg S$ be the subgroup generated by all proper subgroups of $S$ 
strongly closed in $\calf$. By Lemma \ref{l:no.str.cl.}(b) and (i)--(iii), 
the subgroup $Q$ is finite, $Q\le T$, $Q$ is strongly closed in 
$\calf$, and $\minsc{\calf/Q}=S/Q$. By Lemma \ref{l:F(G)/Q}, we have 
$\calf/Q=\calf_{S/Q}(N_G(Q)/Q)$, where 
$S/Q\in\sylp{N_G(Q)/Q}$ since $Q$ is strongly closed. Also, 
$\calf/Q$ is saturated by Lemma \ref{l:F/Q}, and 
$C_{S/Q}(T/Q)=C_S(T)/Q=T/Q$ and 
$\Aut_{\calf/Q}(T/Q)\cong\autf(T)$ by Lemma \ref{l:F/Q-seq.real}. So 
conditions (i)--(iii) also hold for $\calf/Q$. Since 
$\srk_q(G)\ge\srk_q(N_G(Q)/Q)$ for every prime $q\ne p$, we can 
replace $\calf$ by $\calf/Q$, and reduce to the case where 
$\minsc\calf=S$.

By Proposition \ref{p:f.s.union}, there is an increasing sequence $G_1\le 
G_2\le\cdots$ of finite subgroups of $G$ such that if we set 
$S_i=S\cap G_i$ and $\calf_i=\calf_{S_i}(G_i)$ for each $i$, then 
$S_i\in\sylp{G_i}$ for each $i$, and $S=\bigcup_{i=1}^\infty S_i$ and 
$\calf=\bigcup_{i=1}^\infty\calf_i$. Also, 
$\calf=\calf_S\bigl(\bigcup_{i=1}^\infty G_i\bigr)$ by the same 
proposition, so we can assume that $G$ is the union of the $G_i$.

For each $i$, set $G_i^*=G_i/O_{p'}(G_i)$. For all $j\ge i$, we have 
$O_{p'}(G_j)\cap G_i\nsg G_i$, and hence $O_{p'}(G_j)\cap G_i\le 
O_{p'}(G_i)$. So $G_i^*$ is isomorphic to a section (subquotient) of 
$G_j^*$ whenever $i<j$.

For each $i$, set $H_i^*=F^*(G_i^*)$: the generalized Fitting subgroup 
of $G_i^*$. By Lemma \ref{l:seq-exotic}, there 
is $n\ge1$ such that for all $i\ge n$, $H_i^*$ is a product of simple 
groups of Lie type in characteristic different from $p$, where the simple 
factors are permuted transitively by $\Aut_{G_i}(H_i^*)$. Assume, for 
each $i\ge n$, that $H_i^*$ is isomorphic to a product of $\ell_i$ copies 
of the simple group $\9{d_i}\,\gg_i(q_i)$, where $q_i$ is a power of a 
prime $r_i$, where $\gg_i$ is a simple group scheme over $\Z$, and where 
$\9{d_i}\,\gg_i(-)$ means the group $\gg_i(-)$ twisted by a graph automorphism of 
order $d_i=1,2,3$. Thus $\ell_i\cdot\rk(\9{d_i}\,\gg_i)=\rk(T)$ for each 
$i$, where $\rk(\9{d_i}\,\gg_i)$ is its Lie rank. Since there are only 
finitely many possibilities for $\9{d_i}\,\gg_i$ of any given Lie rank, 
there must be some pair $(\9{d_i}\,\gg_i,\ell_i)$ that occurs for 
infinitely many indices $i\ge n$. So upon removing the other terms in the 
sequence, we can assume there are $\ell$, $\gg$, and $d=1,2,3$ such that 
$H_i^*$ is a product of $\ell$ copies of $\9d\,\gg(q_i)$ for each $i$. 
Let $\4W$ be the Weyl group of $\gg$.

We claim that 
	\beqq \textup{there is a prime $r$ such that $r_i=r$ for 
	infinitely many $i\ge n$.} \label{e:r=ri} \eeqq
Once we have shown this, it then follows that the set $\{q_i\,|\,i\ge 
n\}$ contains arbitrarily large powers of $r$. But the group 
$\9d\,\gg(q_i)$ always contains a subgroup isomorphic to $(\F_{q_i},+)$ 
(a subgroup of a root subgroup), and hence the set $\{\rk_r(G_i^*)\}$ of 
$r$-ranks of the $G_i^*$ is unbounded. So $\srk_r(G)=\infty$, which is 
what we needed to show.

To prove \eqref{e:r=ri}, let $\5H_i^*\nsg G_i^*$ be the subgroup of 
elements that normalize each simple factor in $H_i^*$. If \eqref{e:r=ri} 
does not hold, then in particular, there are only finitely many terms for 
which $r_i\mid|\4W|$ or $r_i\le\rk(T)$. Upon removing those terms, we have 
$r_i\nmid|\4W|$ and $r_i\nmid|G_j^*/\5H_j^*|$ for all $i,j\ge n$. (Note 
that $|G_j^*/\5H_j^*|$ divides $\rk(T)!$ since $G_j^*/\5H_j^*$ acts 
faithfully on the set of simple factors of $H_j^*$, and the number of 
factors is at most $\rk(T)$.)

For each pair $i<j$, we already saw that $G^*_i$, and hence each simple 
factor of $H^*_i$, is isomorphic to a section (subquotient) of $G_j^*$. 
Hence each simple factor is isomorphic to a section of $H_j^*$, 
$\5H_j^*/H_j^*$, or $G_j^*/\5H_j^*$: the second is impossible since 
$\5H_j^*/H_j^*$ is solvable (since the outer automorphism group of each 
simple factor is solvable), and the last since $r_i\nmid|G_j^*/\5H_j^*|$. 
Hence for each $i<j$, each simple factor of $H_i^*$ is isomorphic to a 
section of a simple factor of $H_j^*$. 

By \cite[Theorem 4.10.2]{GLS3} and since $r_i\nmid|\4W|$, the Sylow 
$r_i$-subgroups of $H_j^*$ are abelian for all $j>i$ such that $r_j\ne 
r_i$. Hence the Sylow 
$r_i$-subgroup of $H_i^*$ is abelian for all $i$, and $H_i^*$ is a product 
of copies of $\PSL_2(q_i)$. So $\PSL_2(q_i)$ is isomorphic to a subquotient 
of $\PSL_2(q_j)$ for all $i<j$, which for $q_i\ge7$ is possible only 
when $r_i=r_j$ (see, e.g., \cite[Theorem 6.5.1]{GLS3}). 
\end{proof}

Proposition \ref{p:no.char0} will be applied in the next section to fusion 
systems of connected $p$-compact groups, and in particular, of connected 
compact Lie groups.


\section{Fusion systems of \texorpdfstring{$p$-compact}{p-compact} groups} 
\label{s:p-cpt}

A $p$-compact group consists of a triple $(X,BX,i)$, where $X$ 
and $BX$ are spaces and $i\:X\too\Omega(BX)$ is a homotopy equivalence, and 
such that $BX$ is $p$-complete in the sense of Bousfield and Kan \cite{BK} 
and $H^*(X;\F_p)$ is finite. Usually, we just say that $X$ is a $p$-compact 
group and $BX$ is its classifying space. For example, if $G$ is a compact 
Lie group such that $\pi_0(G)$ is a $p$-group, then 
$G\pcom\simeq\Omega(BG\pcom)$ is a $p$-compact group with classifying space 
$BG\pcom$.

A $p$-subgroup of a $p$-compact group $X$ is a pair $(P,f)$, where $P$ is a 
discrete $p$-toral group and $f\:BP\too BX$ is a pointed map that does not 
factor (up to homotopy) through any quotient group $P/Q$ for $1\ne Q\nsg 
P$. (By \cite[Theorems 7.2--7.3]{DW}, this is equivalent to the definition 
of a $p$-subgroup of $X$ given in \cite[\S\,3]{DW} and used in 
\cite[\S\,10]{BLO3}.) By a theorem of Dwyer and Wilkerson (see 
\cite[Propositions 2.10 and 2.14]{DW-center} or \cite[Proposition 
10.1(a)]{BLO3}), every $p$-compact group $X$ contains a Sylow $p$-subgroup: 
a $p$-subgroup $(S,f)$ of $X$ such that every other $p$-subgroup of $X$ 
factors through $f$ up to homotopy.

Let $X$ be a $p$-compact group with $(S,f)\in\sylp{X}$. In \cite[Definition 
10.2]{BLO3}, we give an explicit construction of a fusion system 
$\calf_{S,f}(X)$ over $S$, defined via maps to $BX$. If $T$ is the identity 
component of $S$, we refer to $T$ (or $(T,f|_{BT})$) as the maximal torus 
of $X$ and to $\Aut_{\calf_{S,f}(X)}(T)$ as its Weyl group.

Recall that $\Z_p$ denotes the ring of $p$-adic integers and 
$\Q_p=\Z_p[1/p]$ its field of fractions. Note that 
$\Q_p/\Z_p\cong\Z[\frac1p]/\Z\cong\Z/p^\infty$.

\begin{Lem} \label{l:Hom(Qp,T)}
Let $T$ be a discrete $p$-torus of rank $r$. Then 
\begin{enuma} 

\item $\Hom(\Q_p/\Z_p,T)\cong(\Z_p)^r$ and 
$\Hom(\Q_p,T)\cong(\Q_p)^r$, and the inclusion of $\Hom(\Q_p/\Z_p,T)$ into 
$\Hom(\Q_p,T)$ induces an isomorphism $\Q\otimes_{\Z}\Hom(\Q_p/\Z_p,T)\cong 
\Hom(\Q_p,T)$; and 

\item there is a natural isomorphism of $\Z_p\Aut(T)$-modules 
$\Hom(\Q_p/\Z_p,T)\xto{~\cong~}\pi_2(BT\pcom)$. 

\end{enuma}
\end{Lem}

\begin{proof} \textbf{(a) } It suffices to prove this when $r=1$; i.e., 
when $T=\Q_p/\Z_p$. Consider the homomorphisms 
	\[ \Q_p \Right4{\psi} \Hom(\Q_p,\Q_p/\Z_p) \Right4{\ev_1} 
	\Q_p/\Z_p \]
defined by setting $\psi(x)(y)=xy+\Z_p$ and $\ev_1(\rho)=\rho(1)$. One 
easily seens that $\psi$ is injective, and $\ev_1$ is surjective with 
kernel $\Hom(\Q_p/\Z_p,\Q_p/\Z_p)$. If $\rho\in\Hom(\Q_p,\Q_p/\Z_p)$, 
then for each $n\ge1$, we can choose elements $x_n\in\Q_p$ such that 
$\rho(p^{-n})=x_n+\Z_p$; the sequence $\{p^nx_n\}_{n\ge0}$ converges in 
the $p$-adic topology to some $x\in\Q_p$, and $\rho=\psi(x)$. So $\psi$ 
is surjective.

Thus $\psi$ is an isomorphism, and 
$\Hom(\Q_p/\Z_p,\Q_p/\Z_p)=\Ker(\ev_1)\cong \Ker(\ev_1\circ\psi)$. Since 
$\ev_1\circ\psi$ is the natural surjection of $\Q_p$ onto $\Q_p/\Z_p$, this 
proves that 
	\[ \Hom(\Q_p/\Z_p,\Q_p/\Z_p)\cong\Z_p \qquad\textup{and}\qquad 
	\Hom(\Q_p,\Q_p/\Z_p)\cong\Q\otimes_{\Z}\Hom(\Q_p/\Z_p,\Q_p/\Z_p). 
	\]

\smallskip

\noindent\textbf{(b) } This is a special case of \cite[VI.5.1]{BK}, 
applied with $X=BT$. 
\end{proof}

In the next lemma, we use these $\Q_p$-vector spaces $\Hom(\Q_p,T)$ to 
provide a useful criterion when verifying condition (iii) in Theorem 
\ref{t:seq-exotic}.

\begin{Lem} \label{l:Theta}
Let $T$ be a discrete $p$-torus, and let $G$ be a 
finite group of automorphisms of $T$. Then there is an infinite proper 
$G$-invariant 
subgroup $T_0<T$ if and only if the $\Q_pG$-module $\Hom(\Q_p,T)$ is 
reducible. 
\end{Lem}

\begin{proof} Set $V=\Hom(\Q_p,T)$ for short, regarded as a $\Q_pG$-module.
Assume $P<T$ is an infinite proper $G$-invariant 
subgroup of $T$, and let $T_0\le P$ be its identity component. Then $T_0$ 
is an infinite discrete $p$-torus and is also $G$-invariant. Set 
$V_0=\Hom(\Q_p,T_0)\le V$. Then 
$0<\rk(T_0)<\rk(T)$ implies that $0<\dim_{\Q_p}(V_0)<\dim_{\Q_p}(V)$, so 
$V_0$ is a proper nontrivial $\Q_pG$-submodule of $V$, and $V$ is reducible. 

Conversely, assume $V$ is reducible, and let $V_0<V$ be a proper nontrivial 
submodule. Let $\ev_1\:V=\Hom(\Q_p,T)\too T$ be the homomorphism $\ev_1(v)=v(1)$, and 
set $T_0=\ev_1(V_0)$. The image under $\ev_1$ of each 1-dimensional 
subspace of $V$ is isomorphic to $\Q_p/\Z_p\cong\Z/p^\infty$, and from this 
it follows that $T_0$ is a discrete $p$-torus where 
$\rk(T_0)=\dim_{\Q_p}(V_0)<\dim_{\Q_p}(V)=\rk(T)$. Thus $T_0$ is an 
infinite, $G$-invariant, proper subgroup of $T$.
\end{proof}

We showed in \cite[Theorem 10.7]{BLO3} that if $X$ is a $p$-compact group 
and $(S,f)\in\sylp{X}$, then $\calf_{S,f}(X)$ is saturated, it has an 
associated centric linking system $\call_{S,f}^c(X)$, and 
$|\call_{S,f}^c(X)|\pcom\simeq BX$. We also showed that the fusion system 
of $X$ is determined by the homotopy type of $BX$, as is made explicit in 
the following lemma. 

\begin{Lem} \label{l:BX=BY}
Let $X$ and $Y$ be $p$-compact groups. If $BX\simeq BY$, then for 
$(S,f)\in\sylp{X}$ and $(U,g)\in\sylp{Y}$ we have $\calf_{S,f}(X)\cong 
\calf_{U,g}(Y)$. If $p$ is odd and $X$ and $Y$ are connected, then this is 
the case if the Weyl groups of $X$ and $Y$ are isomorphic and have 
isomorphic actions on their maximal tori. 
\end{Lem}

\begin{proof} By \cite[Theorem 10.7]{BLO3}, there are centric linking systems 
$\call_{S,f}^c(X)$ and $\call_{U,g}^c(Y)$, associated to $\calf_{S,f}(X)$ 
and $\calf_{U,g}(Y)$, such that
	\[ |\call_{S,f}^c(X)|\pcom \simeq BX\pcom \simeq BY\pcom \simeq 
	|\call_{U,g}^c(Y)|\pcom. \]
Hence $\calf_{S,f}(X)\cong\calf_{U,g}(Y)$ by \cite[Theorem 7.4]{BLO3}. 
This proves the first statement. The second follows from 
\cite[Theorem 1.1]{AGMV}, together with Lemma \ref{l:Hom(Qp,T)}(b) which 
says that the actions of the Weyl groups on $\Hom(\Q_p/\Z_p,T_X)$ and 
$\Hom(\Q_p/\Z_p,T_Y)$, where $T_X$ and $T_Y$ are maximal discrete $p$-tori 
in $X$ and $Y$, are isomorphic to the actions on $L_X$ and $L_Y$ used in 
\cite{AGMV}. 
\end{proof}

The next lemma is similar, and is our means of showing that the fusion 
systems of certain $p$-compact groups are realized by linear torsion 
groups. 

\begin{Lem} \label{l:BX=BGamma}
Fix a prime $p$, a $p$-compact group $X$, and a linear torsion group 
$\Gamma$ in characteristic different from $p$ such that $BX\simeq 
B\Gamma\pcom$. Then for $(S,f)\in\sylp{X}$ and $S_\Gamma\in\sylp\Gamma$, we 
have $\calf_{S,f}(X)\cong\calf_{S_\Gamma}(\Gamma)$. 
\end{Lem}

\begin{proof} By \cite[Theorem 10.7]{BLO3} and Proposition 
\ref{p:LT-real.}, there are centric linking systems $\call_{S,f}^c(X)$ and 
$\call_{S_\Gamma}^c(\Gamma)$, associated to $\calf_{S,f}(X)$ and 
$\calf_{S_\Gamma}(\Gamma)$, respectively, such that 
	\[ |\call_{S,f}^c(X)|\pcom \simeq BX\pcom \simeq B\Gamma\pcom \simeq 
	|\call_{S_\Gamma}^c(\Gamma)|\pcom. \]
So $\calf_{S,f}(X)\cong\calf_{S^*}(\Gamma)$ by \cite[Theorem 7.4]{BLO3}. 
\end{proof}

The next lemma proves some of the properties of fusion systems of 
$p$-compact groups needed to apply Theorem \ref{t:seq-exotic} and similar 
results. Recall that a connected $p$-compact group $X$ with maximal torus 
$T$ and Weyl group $W$ is \emph{simple} if $Z(X)=1$ and 
$\Q\otimes_{\Z}\pi_2(BT\pcom)$ is irreducible as a $\Q_pW$-module (see 
\cite[Definition 1.2]{DW-prod}).

\begin{Lem} \label{l:X.conn}
Let $X$ be a connected $p$-compact group, fix $(S,f)\in\sylp{X}$, and set 
$\calf=\calf_{S,f}(X)$. Let $T\nsg S$ be the identity component. Then 
\begin{enuma} 

\item each $s\in S$ is $\calf$-conjugate to an element in $T$; 

\item $C_S(T)=T$, and no proper subgroup of $S$ containing $T$ is strongly 
closed in $\calf$; and 

\item if $X$ is simple, then no infinite proper subgroup of $T$ is 
invariant under the action of $\autf(T)$.

\end{enuma}
Thus whenever $X$ is simple and $p\mid|\autf(T)|$, conditions (i), (ii), 
and (iii) in Theorem \ref{t:seq-exotic} all hold. 
\end{Lem}

\begin{proof} The last statement follows immediately from (b) and (c), and 
since $p\mid|\autf(T)|$ implies $S>T$. 

\smallskip

\noindent\textbf{(a) } By assumption, $f\:BS\too BX$ is a Sylow 
$p$-subgroup of $X$. Set $f_0=f|_{BT}\:BT\too BX$. 

Fix $s\in S\sminus T$, and let $m\ge1$ be such that $|s|=p^m$. Let 
$\rho\in\Hom(\Z/p^m,S)$ be the homomorphism that sends the class of $1$ to 
$s$, and set $\chi=f\circ B\rho$. By \cite[Proposition 5.6]{DW} and since 
$X$ is connected, $\chi$ extends to maps defined on $B\Z/p^n$ for all 
$n>m$, and hence to a map $\5\chi\:BA\too BX$ where $A\cong(S^1)\pcom$ has 
discrete approximation $A_\infty\cong\Z/p^\infty$. 

By \cite[Proposition 8.11]{DW}, there is a pointed map $\5\tau\:BA\too 
BT\pcom$ such that $\5\chi\simeq f_0\circ\5\tau$. By \cite[Proposition 
3.2]{DW-center}, there is $\tau\in\Hom(A_\infty,T)$ such that $\5\tau\simeq 
B\tau$. So $f\circ B\rho=\5\chi|_{B\Z/p^m}\simeq f_0\circ 
B(\tau|_{\Z/p^m})$, and by definition of the fusion system $\calf$, this 
means that the homomorphisms $\rho\in\Hom(\Z/p^m,S)$ and 
$\tau|_{\Z/p^m}\in\Hom(\Z/p^m,T)$ are $\calf$-conjugate. So $s=\rho(1)$ is 
$\calf$-conjugate to $\tau(1)\in T$. 

\smallskip

\noindent\textbf{(b) } Fix an element $x\in C_S(T)$. 
By (a), $\gen{x}$ is $\calf$-conjugate to a subgroup of $T$, so by 
\cite[Lemma 2.4(a)]{BLO3}, it is conjugate to a subgroup of $T$ fully 
centralized in $\calf$. So fix $\varphi\in\homf(\gen{x},T)$ such that 
$\varphi(\gen{x})\le T$ and is fully centralized in $\calf$. Then $\varphi$ extends 
to some $\4\varphi\in\homf(C_S(x),S)$ by the extension axiom. But $T\le 
C_S(x)$ and $\4\varphi(T)=T$, so $x\in T$ since $\4\varphi(x)\in T$. Thus 
$C_S(T)=T$.

If $R$ is strongly closed in $\calf$ and $T\le R\le S$, then $R=S$ since 
each element of $S$ is $\calf$-conjugate to an element of $R$.

\smallskip

\noindent\textbf{(c) } Set $L_X=\pi_2(BT\pcom)$ (following the notation of 
Dwyer and Wilkerson \cite{DW-prod}). By the definition in 
\cite[\S\,1]{DW-prod} of a simple connected $p$-compact group, $\Q\otimes 
L_X$ is a simple $\Q_p\autf(T)$-module whenever $X$ is simple. So 
$\Hom(\Q_p,T)$ is a simple $\Q_p\autf(T)$-module by Lemma 
\ref{l:Hom(Qp,T)}(a,b), and there is no infinite proper $\autf(T)$-invariant 
subgroup of $T$ by Lemma \ref{l:Theta}.
\end{proof}

The following notation for certain infinite algebraic extensions of $\F_q$ 
will be useful when stating our main theorem. 

\begin{Not} \label{n:Kqp}
When $q$ is a prime power and $\scrp$ is a set of primes, then 
$\F_{q^{\gen\scrp}}$ denotes the union of all finite fields $\F_{q^a}$ for 
$a\ge1$ divisible only by primes in $\scrp$. When $0\ne m\in\N$, we also 
write $\F_{q^{\gen{m}}}=\bigcup_{i=1}^\infty\F_{q^{m^i}}$ and 
$\F_{q^{\gen{m'}}}=\bigcup\{\F_{q^k}\,|\,\gcd(k,m)=1\}$; i.e., the special 
cases of $\F_{q^{\gen\scrp}}$ where $\scrp$ is the set of primes dividing $m$ 
or its complement.
\end{Not}

Thus, for example, $\F_{q^{\gen1}}=\F_q$ and $\F_{q^{\gen{1'}}}=\4\F_q$.

We are now ready to determine exactly which fusion systems of 
simple, connected $p$-compact groups are sequentially realizable and which 
are not. 

\begin{Thm} \label{t:p-compact}
Let $X$ be a simple, connected $p$-compact group, choose $S\in\sylp{X}$, 
and set $\calf=\calf_S(X)$. Let $T\nsg S$ be the identity component of $S$, 
and let $W=\autf(T)$ be the Weyl group. 
\begin{enuma} 

\item If $p\nmid|W|$, then $\calf$ is \textup{LT}-realized by the 
semidirect product $T\rtimes W$.

\item If $W$ is the Weyl group of a simple group scheme $\gg$ 
over $\Z$, then $\calf$ is \textup{LT}-realized by the algebraic group 
$\gg(\4\F_q)$ for each prime $q\ne p$. 

\item If $p\mid|W|$ and $W$ is not the Weyl group of a simple algebraic 
group scheme, then $\calf$ is either \textup{LT}-realizable or not 
sequentially realizable, as described in Table \ref{tbl:p-cpt}.

\end{enuma}
\begin{table}[ht]
\[ \renewcommand{\arraystretch}{1.4} \renewcommand{\4}[1]{\overline{#1}}
\begin{array}{cccc|c} 
p & W & \textup{conditions or comments} & \rk(T) 
& \textup{realized by} \\\hline
p & G(m,1,n) & 3\le m\mid(p-1),~ n\ge p & n & 
\GL_{mn}(\F_{q_1^{\gen{p}}}) \\
p & G(2m,2,n) & 2\le m\mid(p-1)/2,~ n\ge p & n & 
\PSO_{2mn}^{(-1)^n}(\F_{q_2^{\gen{p}}}) \\
p & G(m,k,n) & 1\ne k\mid m\mid(p-1),~ k\ge3,~ n\ge p & n & 
\textup{not seq. real.} \\
3 & \ST{12} & W\cong\GL_2(3) & 2 & \lie2F4(\F_{2^{\gen{3}}}) \\ 
2 & \ST{24} & W\cong C_2\times\SL_3(2) & 3 & \textup{not seq. real.} \\
5 & \ST{29} & W\cong (C_4\circ Q_8\circ Q_8).\Sigma_5 & 4 
& \textup{not seq. real.} \\ 
5 & \ST{31} & W\cong(C_4\circ Q_8\circ Q_8).\Sigma_6 & 4 
& E_8(\F_{q_3^{\gen5}}) \\ 
7 & \ST{34} & W\cong 6.\PSU_4(3).2 & 6 & \textup{not seq. real.} \\
\end{array}
\]
\caption{Here, $q_1$, $q_2$, and $q_3$ are prime powers: $q_1$ and 
$q_2$ have order $m$ and $2m$, respectively, in $(\Z/p)^\times$, while 
$q_3\equiv\pm2$ (mod $5$). In the last five cases, $\ST{n}$ is the 
reflection group of type $n$ in the Shephard-Todd list \cite[p. 301]{ST}. 
Also, $G\circ H$ denotes a central product of the groups $G$ and $H$, while 
$G.H$ denotes an extension of $G$ by $H$.}
\label{tbl:p-cpt}
\end{table} 
\end{Thm}

\begin{proof} For the proofs that each simple, connected $p$-compact group 
has the form listed in (a) or (b) or is one of the entries in Table 
\ref{tbl:p-cpt}, we refer to \cite[Theorem 1.1]{AGMV} or \cite[Theorem 
1.1]{AG}, as well as the classification by Clark and Ewing \cite{CE} of 
finite reflection groups over $\Z_p$. Also, when $p$ is odd, a simple, 
connected $p$-compact group is uniquely determined by the Weyl group and 
its action on the maximal torus by \cite[Theorem 1.1]{AGMV} again.

\smallskip

\noindent\textbf{(a) } If $p\nmid|W|$, then $S=T\nsg\calf$ by 
Theorem \ref{t:AFT} (Alperin's fusion theorem), and so $\calf$ is the 
fusion system of $T\rtimes W$. 

\smallskip

\noindent\textbf{(b) } This was shown in Theorem \ref{t:cpt.cn.Lie}. 

\smallskip

\noindent\textbf{(c) } We consider individually the entries in Table 
\ref{tbl:p-cpt}.

\smallskip

\noindent\textbf{Case 1: } Assume $W\cong G(m,1,n)\cong C_m\wr\Sigma_n$ for 
some $3\le m\mid(p-1)$ and $n\ge p$. In particular, $p$ is odd. Fix $q_1$ of 
order $m$ in $(\Z/p)^\times$, set 
$K=\F_{q_1^{\gen{m'}}}\subseteq\4\F_{q_1}$, the 
union of all finite extensions of $q_1$ of degree prime to $m$, and set 
$\Gamma=GL_{mn}(K)$. Also, set $K_0=\F_{q_1^{\gen{p}}}\le K$ and 
$\Gamma_0=\GL_{mn}(K_0)\le\Gamma$. For each $r\ge1$ prime to $m$, if 
$p^i\mid r$ is the largest power of $p$ dividing $r$, then the $p$-fusion 
systems of $\GL_{mn}(q_1^{p^i})$ and of $\GL_{mn}(q_1^r)$ are isomorphic by 
\cite[Theorem A(a)]{BMO1}. Since 
$\Gamma_0=\bigcup_{i=0}^\infty\GL_{mn}(q_1^{p^i})$ and $\Gamma$ is the 
union of the groups $\GL_{mn}(q_1^r)$ for all $r$ prime to $m$,
the $p$-fusion systems of $\Gamma_0$ and $\Gamma$ are also isomorphic. 

By a theorem of Quillen \cite[\S10]{Quillen}, $H^*(B\Gamma;\F_p)$ is a 
polynomial algebra, and hence the loop space $\Omega(B\Gamma\pcom)$ has 
finite cohomology (an exterior algebra). So $B\Gamma\pcom$ is the 
classifying space of a $p$-compact group. Thus by Lemma 
\ref{l:BX=BGamma}, the $p$-fusion system of $\Gamma$ is the fusion system 
of a $p$-compact group. From \cite[Table 6.1]{BMO2} (applied to finite 
subfields of $K$), it follows that $\Gamma$ has a maximal torus $T$ of 
rank $n$ with Weyl group $\autf(T)\cong C_m\wr\Sigma_n$. So by the 
uniqueness statement in \cite[Theorem 1.1]{AGMV}, $\calf$ is the fusion 
system of $X$.

\smallskip

\noindent\textbf{Case 2: } Assume $W\cong G(2m,2,n)$ for some $2\le 
m\mid(p-1)/2$ and $n\ge p$. Let $G=SO_{2nm}(\C)$ be the simple complex 
algebraic group of type $SO_{2nm}$. Let $\tau\in\Aut(BG)$ be the self map 
induced by a graph automorphism of order $2$. By \cite[Theorem 
12.2]{Friedlander}, if $q=r^m$ for some prime $r\ne p$ and 
$\psi^q\in\Aut(BG\pcom)$ is the unstable Adams map of degree $q$, then 
$(B\SO_{2nm}^+(q))\pcom$ and $(B\SO_{2nm}^-(q))\pcom$ are the spaces of 
homotopy fixed points of the actions of $\psi^q$ and $\tau\psi^q$, 
respectively, on $BG\pcom$.

Choose a prime power $q_2$ with $\ord_p(q_2)=2m$. Let $\zeta$ be the primitive 
$2m$-th root of unity in the $p$-adic integers $\Z_p$ such that $\zeta\equiv 
q_2$ (mod $p$). Set $q_0 =\zeta^{-1}q_2\in \Z_p$; thus $q_2= \zeta q_0$ and 
$q_0\equiv1$ (mod $p$). Set $\Gamma=\gen{\tau^n\psi^\zeta}$. Since $\tau$ 
has order $2$ and commutes with $\psi^\zeta$ up to homotopy by 
\cite[Corollary 3.5]{JMO}, 
$\Gamma$ is cyclic of order $2m$ as a subgroup of $\Out(BG\pcom)$ in all cases. 
By \cite[Theorem B]{BM}, the homotopy action of $\Gamma\cong C_{2m}$ on $BG$ 
lifts to an actual action (i.e., a homomorphism $\Gamma\too\Aut(BG)$). So 
for each $i\ge0$, 
	\beqq (B\SO_{2nm}^{(-1)^n}(q_2^{p^i}))\pcom \simeq 
	(B\SO_{2nm}(\C)\pcom)^{h\tau^n\psi^{q_2^{p^i}}} \simeq 
	\bigl((B\SO_{2nm}(\C)\pcom)^{h\Gamma}\bigr)(q_0^{p^i}) 
	\label{e:p-cpt-1} \eeqq
where the second equivalence follows from \cite[Theorem E]{BM} and since 
$q_2^{p^i}=\zeta q_0^{p^i}$ (recall $\zeta^{2m}=1$ and $2m\mid(p-1)$). Here, 
$B\SO_{2nm}^{(-1)^n}$ means $B\SO_{2nm}^+$ if $n$ is even and 
$B\SO_{2nm}^-$ if $n$ is odd. When $X$ is a $p$-compact group, we follow 
the notation in \cite{BM} and let $BX(q_2)$ denote the homotopy fixed set of 
an unstable Adams operation of degree $q_2$ acting on $BX$. 

Set $BY=(B\SO_{2nm}(\C)\pcom)^{h\Gamma}$ for short. Then \eqref{e:p-cpt-1} 
implies that 
	\[ (B\SO_{2nm}^{(-1)^n}(\F_{q_2^{\gen{p}}}))\pcom \simeq 
	\hocolim_{i\ge0} BY(q_0^{p^i}). \]


By \cite[Theorem B]{BM}, $BY$ is the classifying space of a connected 
$p$-compact group. By the same theorem and since $H^*(B\SO(2nm);\F_p)$ is 
polynomial, $H^*(BY;\F_p)$ is also polynomial. So by \cite[Theorem F]{BM}, 
for each $i\ge1$, $H^*(BY(q_0^{p^i});\F_p)$ is isomorphic as a graded ring 
to the tensor product of $H^*(BY;\F_p)$ with an exterior algebra, where the 
polynomial generators are higher Bocksteins of the exterior generators 
(different higher Bocksteins depending on $i$). Hence the natural maps from 
$BY(q_0^{p^i})$ to $BY(q_0^{p^{i+1}})$ induce isomorphisms on the 
polynomial parts of the cohomology rings and trivial maps on the exterior 
parts. The natural map from $\hocolim(BY(q_0^{p^i}))$ to $BY$ thus induces 
an isomorphism 
	\begin{align*} 
	H^*(BY;\F_p) \cong \holim_{i\ge0} H^*(BY(q_0^{p^i});\F_p) &\cong 
	H^*\bigl( \hocolim_{i\ge0}(BY(q_0^{p^i})) ; \F_p) \\ 
	&\cong H^*(B\SO_{2nm}^{(-1)^n}(\F_{q_2^{\gen{p}}});\F_p), 
	\end{align*}
and so $B\SO_{2nm}^{(-1)^n}(\F_{q_2^{\gen{p}}})\pcom \simeq BY$. 

Thus $B\SO_{2nm}^{(-1)^n}(\F_{q_2^{\gen{p}}})\pcom$ is the classifying space of 
a connected $p$-compact group. Let $W(Y)$ denote its Weyl group; 
equivalently, the Weyl group (torus automizer) of 
$B\SO_{2nm}^{(-1)^n}(q_2^{p^i})$ for large $i$. From \cite[Table 6.1]{BMO2}, we 
see that $W(Y)\cong C_{G(2,2,mn)}(w_0)$ where $w_0$ acts on a maximal torus in 
$\SO_{2mn}(\4\F_{q_2})$ by sending $(\lambda_1,\lambda_2,\dots,\lambda_{mn})$ 
to 
	\[ (\lambda_m^{-1},\lambda_1,\dots,\lambda_{m-1}, 
	\lambda_{2m}^{-1},\lambda_{m+1},\dots,\lambda_{2m-1},\dots, 
	\lambda_{nm}^{-1},\lambda_{(n-1)m+1},\dots,\lambda_{nm-1}). \]
Thus $W(Y)$ is generated by all products of evenly many permutations 
	\[ [\lambda_{im+1},\dots,\lambda_{im+m-1},\lambda_{(i+1)m}]~ \mapsto ~ 
	[\lambda_{(i+1)m}^{-1},\lambda_{im+1},\dots,\lambda_{im+m-1}] \]
(of order $2m$) together with all permutations of the $n$ blocks of length 
$m$, and is isomorphic to $G(2m,2,n)$. So $BY\simeq BX(2m,2,n)$ by Lemma 
\ref{l:BX=BY}, and hence the fusion systems of 
$\SO_{2nm}^{(-1)^n}(\F_{q^{\gen{p}}})$ and $X(2m,2,n)$ are isomorphic by 
Lemma \ref{l:BX=BGamma}.

\smallskip

\noindent\textbf{Case 3: } 
Assume $W\cong\ST{12}$ or $\ST{31}$. The fusion systems of these two 
$p$-compact groups are described in \cite[Table 5.1]{indp3}.

Consider the pairs $(G,p)=(\lie2F4(2^k),3)$ and $(E_8(q_3^k),5)$, where 
$q_3\equiv\pm2$ (mod $5$) is a prime power and $k$ is odd. In each of these 
cases, the Sylow $p$-subgroup $S\in\sylp{G}$ contains a homocyclic abelian 
subgroup $A\nsg S$ of index $p$, and information about the fusion systems 
$\calf_S(G)$, including the essential subgroups other than $A$, is given in 
\cite[Table 4.2]{indp3}. If $(G,p)=(\lie2F4(2^k),3)$, then $A$ 
has rank $2$ and exponent $3^e$ where $e=v_3(2^k+1)>0$. If 
$(G,p)=(E_8(q_3^k),5)$, then $A$ has rank $4$ and exponent $5^e$ where 
$e=v_5(q_3^{2k}+1)>0$. By comparing Tables 4.2 and 5.1 in \cite{indp3}, we see 
that $\calf$ is the union of the fusion systems of the finite groups 
$\lie2F4(2^k)$ or $E_8(q_3^k)$ taken over all odd $k$, and hence is realized 
by the linear torsion group $\lie2F4(\F_{2^{\gen{2'}}})$ or 
$E_8(\F_{q_3^{\gen{2'}}})$, respectively. A similar argument shows that $\calf$ 
is also realized by $\lie2F4(\F_{2^{\gen3}})$ or $E_8(\F_{q_3^{\gen5}})$. 

\smallskip

\noindent\textbf{Case 4: } Assume either $W\cong G(m,k,n)$ where $1\ne 
k\mid m\mid(p-1)$, $k\ge3$, and $n\ge p$, or $W\cong \ST{24}$, $\ST{29}$, 
or $\ST{34}$. By Lemma \ref{l:X.conn}, $\calf$ satisfies conditions 
(i)--(iii) in Theorem \ref{t:seq-exotic}. 

By inspection, $W=\autf(T)$ does not contain a normal subgroup of 
index prime to $p$ isomorphic 
to a product of copies of any of the groups listed in points (a)--(e) of 
Theorem \ref{t:seq-exotic}. So $\calf$ is not sequentially realizable by 
that theorem. 
\end{proof}

\begin{Rmk} \label{rmk:Ree}
\renewcommand{\4}[1]{\overline{#1}}
In Theorem \ref{t:p-compact}(c), the fusion systems of certain 
$3$-compact groups are realized by Ree groups $\lie2F4(K)$ for an 
infinite field $K$ of characteristic $2$. 
The groups $\lie2F4(K)$ are often defined only when $K$ is a finite 
extension of $\F_2$ of odd degree, but in fact, they were defined by Ree 
in \cite{Ree} for each perfect field $K$ of characteristic $2$ that has an 
automorphism $\theta\in\Aut(K)$ such that $\theta^2\psi=\Id_K$ where 
$\psi\in\Aut(K)$ is the Frobenius $(t\mapsto t^2)$. In particular, this is 
the case for each algebraic extension of $\F_2$ that is a union of 
extensions of odd degree. Ree's construction in this general situation is 
also described in \cite[Sections 12.3 and 13.4]{Carter}. 

More precisely, Ree defines, for each perfect field $K$ of characteristic 
$2$, a graph automorphism $\sigma\in\Aut(F_4(K))$, with the property that 
$\sigma^2=\5\psi$, where $\5\psi\in\Aut(F_4(K))$ is the field automorphism 
induced by $\psi$. He does this by defining $\sigma$ explicitly on the root 
groups, exchanging the root groups for long and short roots. If 
$\theta\in\Aut(K)$ is such that $\theta^2\psi=\Id_K$ and 
$\5\theta\in\Aut(F_4(K))$ is the induced field automorphism, then 
$(\sigma\5\theta)=\Id$, and he defines 
$\lie2F4(K)=C_{F_4(K)}(\sigma\5\theta)$. 
\end{Rmk}

So far, all of our examples of realizing fusion systems of $p$-compact 
groups involve linear torsion groups in characteristic $q$ for $q\ne p$. In 
Theorem \ref{t:cn.p-cpt}, as well as generalizing Theorem 
\ref{t:p-compact} to arbitrary connected $p$-compact groups, we will 
apply Proposition \ref{p:no.char0} to show that in most cases, their fusion 
systems cannot be realized by linear torsion groups in characteristic $0$. 

We first need to know that the fusion system of a product of $p$-compact 
groups is the product of their fusion systems.


\begin{Prop} \label{p:X1xX2}
Let $X_1$ and $X_2$ be $p$-compact groups, and set $X=X_1\times X_2$. 
Choose Sylow $p$-subgroups $(S_i,f_i)\in\sylp{X_i}$, and set $S=S_1\times 
S_2$ and $f=f_1\times f_2\:BS\too BX$. Then $(S,f)\in\sylp{X}$, and 
	\[ \calf_{S,f}(X) \cong \calf_{S_1,f_1}(X_1)\times 
	\calf_{S_2,f_2}(X_2). \]
\end{Prop}

\begin{proof} Set $\calf=\calf_{S,f}(X)$ and $\calf_i=\calf_{S_i,f_i}(X_i)$ 
for short. Let $\pr_i\:S\too S_i$ and $\rho_i\:BX\too BX_i$ be the 
projections. By a theorem of Dwyer and Wilkerson (see \cite[Proposition 
10.1(a)]{BLO3} for details), a pair $(S,f)$ is in $\sylp{X}$ if and 
only if the homotopy fiber of $f\:BS\too BX$ has finite mod $p$ homology 
and Euler characteristic prime to $p$. Since $\textup{hofib}(f)\simeq 
\textup{hofib}(f_1)\times \textup{hofib}(f_2)$, it follows immediately that 
$(S,f)\in\sylp{X}$. 

Set $\call=\call_{S,f}^c(X)$ and $\call_i=\call_{S_i,f_i}^c(X_i)$ in the 
notation of \cite[Theorem 10.7]{BLO3}. Then $BX\simeq|\call|\pcom$ and 
$BX_i\simeq|\call_i|\pcom$ by the same theorem, and so 
$|\call|\pcom\simeq|\call_1|\pcom\times|\call_2|\pcom$. For each discrete 
$p$-toral group $P$, the projections $\pr_i$ induce a bijection 
$[BP,|\call|\pcom]\cong[BP,|\call_1|\pcom]\times [BP,|\call_2|\pcom]$, and 
hence by \cite[Theorem 6.3(a)]{BLO3} a bijection 
	\beqq \Rep(P,\call) \cong \Rep(P,\call_1) \times \Rep(P,\call_2). 
	\label{e:X1xX2} \eeqq
Here, $\Rep(P,\call)=\Hom(P,S)/{\sim}$, where $\psi\sim\psi'$ if and only 
if there is $\varphi\in\homf(\psi(P),\psi'(P))$ such that 
$\psi'=\varphi\psi$. 

Thus for each $P\le S$ and each $\varphi\in\Hom(P,S)$, 
we have $\varphi\in\homf(P,S)$ if and only if $[\varphi]=[P\hookrightarrow 
S]$ in $\Rep(P,\call)$. By \eqref{e:X1xX2}, this holds if and only if 
$[\pr_i\circ\varphi]=[\pr_i|_P]$ in $\Rep(P,\call_i)$ for $i=1,2$; i.e., if 
and only if $\Ker(\pr_i\circ\varphi)=\Ker(\pr_i|_P)=P\cap S_{3-i}$ and 
$\pr_i\circ\varphi=\varphi_i\circ(\pr_i|_P)$ for some 
$\varphi_i\in\Hom_{\calf_i}(\pr_i(P),S_i)$. This is exactly the 
condition for $\varphi$ to be in $\Hom_{\calf_1\times\calf_2}(P,S)$, and 
hence $\calf=\calf_1\times\calf_2$. 
\end{proof}


We are now ready to look more generally at fusion systems of connected 
$p$-compact groups.

\begin{Thm} \label{t:cn.p-cpt}
Let $X$ be a connected $p$-compact group, fix $(S,f)\in\sylp{X}$, and set 
$\calf=\calf_{S,f}(X)$. 
\begin{enuma} 

\item Set $q=2$ if $p=3$, and let $q$ be any prime whose class generates 
$(\Z/p)^\times$ if $p\ne3$. If $\calf$ is sequentially realizable, then 
$\calf$ is realized by a torsion group that is linear over $\4\F_q$. 

\item The fusion system $\calf$ is realized by a linear torsion group in 
characteristic $0$ if and only if the Weyl group of $X$ has order prime to 
$p$. 

\end{enuma}
\end{Thm}

\begin{proof} By \cite[Theorem 1.2]{AGMV} or \cite[Theorem 1.1]{AG}, there 
are a compact connected Lie group $G$, and simple connected $p$-compact 
groups $X_1,\dots,X_k$ (some $k\ge0$) such that $X$ is the direct product 
of $G\pcom=\Omega(BG\pcom)$ with the $X_i$. Hence for $S_0\in\sylp{G}$ 
and $(S_i,f_i)\in\sylp{X_i}$, the fusion system $\calf$ is the product of 
$\calf_{S_0}(G)$ and the $\calf_{S_i,f_i}(X_i)$ by Proposition \ref{p:X1xX2}. 

\smallskip

\noindent\textbf{(a) } The fusion system of $G$ is realized by a linear 
torsion group in characteristic $q$ by Theorem \ref{t:cpt.cn.Lie}. By 
Theorem \ref{t:p-compact}(a,c) and Table \ref{tbl:p-cpt}, the fusion system 
of each $X_i$ is either realizable by a linear torsion group in 
characteristic $q$ or is not sequentially realizable. So either their 
product $\calf$ is realizable by a linear torsion group in characteristic 
$q$; or the fusion system of one of the factors $X_i$ is not sequentially 
realizable, in which case $\calf$ fails to be sequentially realizable by 
Proposition \ref{p:F1xF2.real}(b). 

\smallskip

\noindent\textbf{(b) } Let $T\nsg S$ be the identity 
component of $S$, and set $W=\autf(T)$ (the Weyl group of $X$). If 
$p\nmid|W|$, then $S=T$, and by Alperin's fusion theorem, $\calf$ is 
realized by an extension $\Gamma$ of $T$ by $W$. Since $W$ is finite, 
$\Gamma$ is a linear torsion group in characteristic $0$. 

Now assume that $p\mid|W|$; equivalently, that $S>T$. Let $\Gamma$ be a 
locally finite group such that $S\in\sylp{\Gamma}$ and 
$\calf=\calf_S(\Gamma)$. We will show that there is a prime $q\ne p$ such 
that $\srk_q(\Gamma)=\infty$, and hence that $\Gamma$ cannot be 
linear in characteristic $0$ by Lemma \ref{l:srk(GLn)}(a). 

Assume first that $X$ is simple. By Lemma \ref{l:X.conn} and since 
$p\mid|W|$, conditions (i), (ii), and (iii) in Proposition \ref{p:no.char0} 
all hold. So by that proposition, there is a prime $q\ne p$ such that 
$\srk_q(\Gamma)=\infty$.

Now let $X$ be arbitary (not necessarily simple). Then either $p$ divides 
the order of the Weyl group in one of the simple factors $X_i$ (for $1\le 
i\le k$), or $p$ divides the order of the Weyl group of $\Gamma$. In either 
case, there is a connected simple $p$-compact group $Y$ whose Weyl group 
has order a multiple of $p$, and whose fusion system is a quotient of the 
fusion system of $X$. Then for some $U\nsg S$ strongly closed in $\calf$, 
$\calf/U$ is isomorphic to the fusion system of $Y$. Then 
$\calf/U\cong\calf_{S/U}(N_\Gamma(U)/U)$ by Lemma \ref{l:F(G)/Q}, so by what was 
shown in the last paragraph, there is a prime $q\ne p$ such that 
$\srk_q(N_\Gamma(U)/U)=\infty$. Hence $\srk_q(\Gamma)=\infty$. 
\end{proof}

Theorem \ref{t:cn.p-cpt}(b) shows in particular that the fusion system of a 
compact connected Lie group $G$ at a prime $p$ that divides the order of 
the Weyl group of $G$ cannot be realized by any torsion subgroup of $G$.


\section{Other examples}
\label{s:other.ex.}

By analogy with groups, a saturated fusion system is \emph{simple} if it 
contains no proper nontrivial normal fusion subsystems (see 
\cite[Definition I.6.1]{AKO}). In Section 5 of \cite{indp3}, the authors 
described all simple fusion systems over nonabelian infinite discrete 
$p$-toral groups containing a discrete $p$-torus with index $p$. With the 
help of Theorem \ref{t:seq-exotic}, we can now determine in all cases 
exactly which of those fusion systems are sequentially realizable and which 
are not. We will see that in fact, most of them are not sequentially 
realizable. The only exceptions are those that are fusion systems of 
$p$-compact groups and were shown in Section \ref{s:p-cpt} to be 
sequentially realizable.

Before going into details of the examples, we recall a few of the definitions 
used in \cite{indp3}. Let $Z_2(G)$ denote the second term in the 
upper central series of a group $G$; thus $Z_2(G)/Z(G)=Z(G/Z(G))$.

\begin{Not} \label{n:Delta}
Let $p$ be an odd prime. Assume $\calf$ is a saturated fusion system over a 
nonabelian discrete $p$-toral group $S$ whose identity component 
$T$ has index $p$ in $S$. Assume also that $T\nnsg\calf$; then $S$ 
splits over $T$ by \cite[Corollary 2.6]{indp3}.
\begin{itemize} 

\item Let $\calh$ be the set of all $Z(S)\gen{x}$ for $x\in S\sminus T$. 

\item Let $\calb$ be the set of all $Z_2(S)\gen{x}$ for $x\in S\sminus T$.

\item Let $\EE\calf$ be the set of all $\calf$-essential subgroups 
(see Definition \ref{d:subgroups}). By \cite[Lemma 2.2]{indp3}, 
$\EE\calf\subseteq\{T\}\cup\calh\cup\calb$. 

\end{itemize}
\end{Not}

If $\rk(T)=p-1$, then by Proposition A.4 and Lemma A.5(a,c) in \cite{indp3}, 
$T\cong\Q_p(\zeta)/\Z_p[\zeta]$ as $\Z[S/T]$-modules, where $\zeta$ is a 
primitive $p$-th root of unity in $\4{\Q_p}$ and a generator of $S/T$ acts 
via multiplication by $\zeta$. Hence $Z(S)=C_S(T)$ has order $p$ in this 
case, and so $\autf^\vee(T)=N_{\autf(T)}(\Aut_S(T))$ in the terminology of 
\cite[Theorem B]{indp3}. See also \cite[Notation 2.9]{indp3}. 

\begin{Lem} \label{l:(iii).indp}
Let $\calf$ be a saturated fusion system over an infinite nonabelian discrete 
$p$-toral group $S$ whose identity component $T\nsg S$ has index $p$ in 
$S$. Assume also that $O_p(\calf)=1$. Then there is no infinite proper 
$\autf(T)$-invariant subgroup of $T$. 
\end{Lem}

\begin{proof} Set $G=\autf(T)$ for short. Set $V=\Hom(\Q_p,T)$, 
regarded as a $\Q_pG$-module, and let $\ev_1\:V\too T$ be the surjective 
homomorphism evaluation at $1\in\Q_p$. By Lemma \ref{l:Theta}, it 
suffices to show that the $V$ is irreducible. 

Set $U=\Aut_S(T)\in\sylp{G}$. By \cite[Lemma 5.3]{indp3} and since 
$O_p(\calf)=1$, $T/C_T(U)$ is a discrete $p$-torus of rank $p-1$. Hence 
$\dim_{\Q_p}(V/C_V(U))=p-1$, and $V/C_V(U)$ is 
irreducible as a $\Q_pU$-module. So if $V=V_1\oplus V_2$ 
where the $V_i$ are $\Q_pG$-submodules, then one of the factors, say 
$V_2$, has trivial action of $U$, and hence trivial action of 
$O^{p'}(G)$. But then $\ev_1(V_2)$ is an infinite $G$-invariant subgroup 
of $T$ on which $O^{p'}(G)$ acts trivially, which is impossible by 
\cite[Lemma 2.7]{indp3} and since $O_p(\calf)=1$. 
\end{proof}

We are now ready to examine the realizability of such fusion systems.

\begin{Thm} \label{t:index.p}
Let $\calf$ be a simple saturated fusion system over an infinite nonabelian 
discrete $p$-toral group $S$ whose identity component $T\nsg S$ has index 
$p$ in $S$. Then either 
\begin{enuma} 

\item $\calf$ is isomorphic to the fusion system of a simple, connected 
$p$-compact group, and is \textup{LT}-realizable or not sequentially 
realizable as described in Theorem \ref{t:p-compact}; or 

\item $\calf$ is not sequentially realizable.

\end{enuma}
\end{Thm}

\begin{proof} Let $\calf$ be a simple saturated fusion system over an 
infinite nonabelian discrete $p$-toral group $S$ with identity component 
$T$ of index $p$ in $S$. Note that $C_S(T)=T$ since $S$ is nonabelian. 

If $p=2$, then by \cite[Theorem 5.6]{indp3}, $\calf$ is isomorphic to the 
fusion system of $\SO(3)$ ($\autf(T)\cong C_2$ and $\rk(T)=1$) or $\PSU(3)$ 
($\autf(T)\cong\Sigma_3$ and $\rk(T)=2$). Both of these are 
\textup{LT}-realizable by Theorem \ref{t:cpt.cn.Lie}. 

If $p$ is odd and $T\notin\EE\calf$, then by \cite[Theorem 5.12]{indp3}, 
$\autf(T)\cong C_p\rtimes C_{p-1}$, $\rk(T)=p-1$, and $\EE\calf=\calh$, 
and there is a unique such fusion system for each odd prime $p$. When 
$p=3$, this is the $3$-fusion system of the compact Lie group (or $3$-compact 
group) $\PSU(3)$. So assume $p\ge5$. 
Condition (i) in Theorem \ref{t:seq-exotic} clearly holds, (iii) holds by 
Lemma \ref{l:(iii).indp}, and (ii) holds by Lemma \ref{l:no.str.cl.}(a) 
and since $T\nnsg\calf$ (since a normal 
subgroup is contained in all essential subgroups and hence in each member 
of $\calh$). No subgroup of 
index prime to $p$ in $C_p\rtimes C_{p-1}$ is among those listed in 
Theorem \ref{t:seq-exotic}(b), and so $\calf$ is not sequentially 
realizable by that theorem.

The remaining cases, where $p$ is odd and $T\in\EE\calf$, are all described 
in \cite[Theorem B]{indp3}. Note first that $\rk(T)\ge p-1$ in all 
cases (the minimal dimension of a faithful $\Q_pC_p$-module).

We recall some more notation used in \cite{indp3}, when $\calf$ is a 
saturated fusion system over a nonabelian discrete $p$-toral group $S$ 
whose identity component $T$ has index $p$ in $S$. 
\begin{itemize} 

\item Set $\Delta=(\Z/p)^\times\times(\Z/p)^\times$. For each $i\in\Z$, set 
$\Delta_i=\{(r,r^i)\,|\,r\in(\Z/p)^\times\}\le\Delta$.


\item Define $\mu\:\Aut(S)\too\Delta$ by setting 
$\mu(\alpha)=(r,s)$ if $\alpha(x)\in x^rT$ for $x\in S\sminus T$ 
and $\alpha(g)=g^s$ for $g\in Z(S)\cap[S,S]$. (In all cases, 
$|Z(S)\cap[S,S]|=p$ by \cite[Lemma 2.4]{indp3}.) 

\item Define $\mu_\calf\:N_{\autf(T)}(\Aut_S(T))\too\Delta$ by setting 
$\mu_\calf(\beta)=\mu(\alpha)$ for some $\alpha\in\Aut(S)$ such that 
$\alpha|_T=\beta$. (Such an $\alpha$ exists by the extension axiom.)

\item Set $\autf^\vee(T)=\bigl\{\beta\in N_{\autf(T)}(\Aut_S(T)) \,\big|\, 
[\beta,Z(S)]\le Z(S)\cap[S,S]\bigr\}$. 

\end{itemize}

By Theorem B in \cite{indp3}, a simple fusion system $\calf$ that realizes a 
given pair $(\autf(T),T)$ is determined up to isomorphism 
by $\EE\calf$, where $\EE\calf=\{T\}\cup\calh$ or $\{T\}\cup\calb$ in all 
cases. Set $W=\autf(T)$ and $W^\vee=\autf^\vee(T)$ for short. Then 
the following implications hold by \cite[Theorem B]{indp3} again: 
	\beqq \begin{split} 
	\rk(T)=p-1,~ \EE\calf=\{T\}\cup\calb 
	&~\implies~ \mu_\calf(W^\vee)\ge\Delta_0 
	~\textup{and}~ W=O^{p'}(W)\mu_\calf^{-1}(\Delta_0) \\
	\rk(T)=p-1,~ \EE\calf=\{T\}\cup\calh 
	&~\implies~ \mu_\calf(W^\vee)\ge\Delta_{-1} 
	~\textup{and}~ W=O^{p'}(W)\mu_\calf^{-1}(\Delta_{-1}) .
	\end{split} \label{e:index.p-1} \eeqq
	\beqq \rk(T)\ge p \quad\implies\quad \EE\calf=\{T\}\cup\calb, ~
	\mu_\calf(W^\vee)=\Delta_0, ~\textup{and}~ 
	W=O^{p'}(W)\cdot W^\vee. \label{e:index.p-2} \eeqq
Thus there are up to isomorphism at most two simple fusion systems 
realizing $(\autf(T),T)$ if $\rk(T)=p-1$, and at most one such system if 
$\rk(T)\ge p$.

Condition (i) in Theorem \ref{t:seq-exotic} ($S>T$ and $C_S(T)=T$) holds in 
all cases since $|S/T|=p$ and $S$ is nonabelian. Condition (ii) ($T$ is not 
strongly closed in $\calf$) holds since $\EE\calf\not\subseteq\{T\}$, and 
(iii) holds by Lemma \ref{l:(iii).indp}. Thus by Theorem 
\ref{t:seq-exotic}, $\calf$ can be sequentially realizable only if there is 
a normal subgroup $H\nsg\autf(T)$ of index prime to $p$ isomorphic to one of 
the groups listed in that theorem, and whose action on $\Omega_1(T)$ has 
the same composition factors as the $\F_pH$-module $M$ listed in 
Proposition \ref{p:Weyl}. 

Assume first that $\rk(T)\ge p$. Thus we are in one of cases (a)--(c) in 
Proposition \ref{p:Weyl}, so there is a connected simple $p$-compact group 
$X$ with maximal discrete $p$-torus $T^*$ and Weyl group $H$. Also, by 
Lemma \ref{l:V.simple}, the $\F_pH$-module $M$ is simple, and hence 
$\Omega_1(T^*)\cong M\cong\Omega_1(T)$ as $\F_pH$-modules. Let $\cale$ be the 
fusion system of $X$, and set $H^\vee=H\cap W^\vee=\Aut_\cale^\vee(T)$. By 
\cite[Theorem B(b)]{indp3}, we have $T^*\cong T$ as $\Z_pH$-modules. 


By \cite[Lemma 2.5(a)]{indp2}, $\Ker(\mu_\calf|_{W^\vee})=\Aut_S(T)\le 
H^\vee$. Since $\mu_\calf(W^\vee)=\Delta_0=\mu_\cale(H^\vee)$ by 
\eqref{e:index.p-2}, this proves that $H^\vee=W^\vee$, and hence that $H=W$ 
by \eqref{e:index.p-2} again (and since $O^{p'}(H)=O^{p'}(W)$). We already 
saw that the fusion system is determined by $(W,T)$, and hence 
$\calf\cong\cale$. Thus $\calf$ is always the fusion system of a connected 
simple $p$-compact group in this case.

We are left with the cases where $\rk(T)=p-1$ 
and $H\nsg\autf(T)$ is one of the groups listed in Theorem 
\ref{t:seq-exotic}. Note that for such $\calf$, $Z(S)\le[S,S]$ and hence 
$\autf^\vee(T)=N_{\autf(T)}(\Aut_S(T))$. We list these cases in Table 
\ref{tbl:inf-real}.
\begin{table}[ht]
\[ \renewcommand{\arraystretch}{1.4} \renewcommand{\4}[1]{\overline{#1}}
\begin{array}{ccccc} 
p & \autf(T) & \rk(T) & 
\EE\calf & \textup{seq. realizable?} \\\hline\hline
p\ge5 & \Sigma_p & p{-}1 & \{T\}\cup\calh & \PSL_p(\4\F_q) \\ \hline
3 & \ST{12}\cong GL_2(3) & 2 & \{T\}\cup\calb & \lie2F4(\F_{2^{\gen3}}) 
\\ \hline 
3 & \ST{12}\cong GL_2(3) & 2 & \{T\}\cup\calh 
& \textup{not seq. real.} \\ \hline 
5 & \ST{31}\cong(C_4\circ2^{1+4}).\Sigma_6 & 4 & \{T\}\cup\calb & 
E_8(\F_{2^{\gen5}}) \\ \hline
5 & \ST{31}\cong(C_4\circ2^{1+4}).\Sigma_6 & 4 & \{T\}\cup\calh & 
\textup{not seq. real.} \\\hline 
\end{array}
\]
\caption{In each case, either the entry in the last column either is a 
linear torsion group that realizes the fusion system $\calf$ determined by 
the information in the other columns, or $\calf$ is not 
sequentially realizable. In Case 1, $q$ is an arbitrary prime different 
from $p$. See Notation \ref{n:Kqp} for the definition of 
the fields $\F_{2^{\gen{3}}}$ and $\F_{2^{\gen{5}}}$.} 
\label{tbl:inf-real}
\end{table} 

\smallskip

\noindent\boldd{$\autf(T)$ contains $\Sigma_p$ for $p\ge5$. } 
Again set $W=\autf(T)$, and assume that $H\nsg W$ has index prime to $p$, 
that $H\cong\Sigma_p$, and that the action of $H$ on $\Omega_1(T)$ has the 
same composition factors as that of the Weyl group of $\PSU(p)$ on the 
$p$-torsion in its maximal torus. Let $M\cong(\F_p)^p$ be the module with 
permutation action of $H$, and let $M_1<M_{p-1}<M$ be the 
$\F_pH$-submodules of dimension $1$ and $p-1$, respectively. (Thus 
$M_1=C_M(H)$ and $M_{p-1}=[H,M]$.) By \cite[Table 6.1]{indp3}, 
$\Omega_1(T)$ is isomorphic to $M_{p-1}$ or $M/M_1$ as $\F_pH$-modules. 
By direct computation,
	\[ \Omega_1(T)\cong M_{p-1} \implies \mu_\calf(H)=\Delta_0 
	\qquad\textup{and}\qquad
	\Omega_1(T)\cong M/M_1 \implies \mu_\calf(H)=\Delta_{-1} . \] 
By \cite[Theorem B]{indp3} and since $M_{p-1}$ contains a $1$-dimensional 
$\F_pW$-submodule, $\EE\calf=\{T\}\cup\calh$ if $\Omega_1(T)\cong M_{p-1}$. 

If $\Omega_1(T)\cong M_{p-1}$ and $\EE\calf=\{T\}\cup\calh$ or 
$\Omega_1(T)\cong M/M_1$ and $\EE\calf=\{T\}\cup\calb$, then the first 
condition on each line in \eqref{e:index.p-1} implies that $W=H\times 
C_{p-1}$. But then $O^{p'}(W)\mu_\calf^{-1}(\Delta_t)$ ($t=0,-1$) has index 
$2$ in $W$ in each case, so the second condition fails to hold. So these 
cases cannot occur, and we are left with the case where $\Omega_1(T)\cong 
M/M_1$ and $\EE\calf=\{T\}\cup\calh$, where $\calf$ is the fusion system of 
the compact Lie group $\PSU(p)$, and is sequentially realizable by Theorem 
\ref{t:cpt.cn.Lie} and Proposition \ref{p:LT=>s.real.}.

\smallskip

\noindent\textbf{$\autf(T)$ contains $\ST{12}$ or $\ST{31}$. } 
Set $W=\autf(T)$, and let $H\nsg W$ be such that either 
$H\cong\ST{12}$, $p=3$, and $\rk(T)=2$; or $H\cong\ST{31}$, $p=5$, and 
$\rk(T)=4$. Recall (Lemma \ref{l:CS(Omega1)}) that $W$ acts faithfully on 
$\Omega_1(T)$. So $H=W$ if $H\cong\ST{12}$. If $H\cong\ST{31}$, then since 
$\ST{31}$ is the normalizer in $\GL_4(5)$ of $4\circ2^{1+4}$, as shown in 
the proof of Proposition \ref{p:Weyl}, we also have $W=H$ in this case.

In both cases, $\mu_\calf(\autf^\vee(T))=\Delta$. If 
$\EE\calf=\{T\}\cup\calb$, then $\calf$ is the fusion system of a 
$p$-compact group by Theorem B and Table 5.1 in \cite{indp3}. By Theorem 
\ref{t:p-compact}, $\calf$ is realized by the group 
$\lie2F4(\F_{2^{\gen3}})$ or $E_8(\F_{2^{\gen5}})$ as listed in Table 
\ref{tbl:inf-real}.

Now assume $\autf(T)\cong\ST{12}$ or $\ST{31}$ and 
$\EE\calf=\{T\}\cup\calh$; we must show that $\calf$ is not sequentially 
realizable. Assume otherwise, and let $\calf_1\le\calf_2\le\cdots$ be an 
increasing sequence of realizable fusion subsystems of $\calf$ over finite 
subgroups $S_1\le S_2\le\cdots$ of $S$ such that $S=\bigcup_{i=1}^\infty 
S_i$ and $\calf=\bigcup_{i=1}^\infty\calf_i$. Let $G_1,G_2,\dots$ be finite 
groups such that $S_i\in\sylp{G_i}$ and $\calf_i=\calf_{S_i}(G_i)$ for each 
$i\ge1$. Set $T_i=T\cap S_i$.

By Lemma \ref{l:autf(A)-2}, there is $n$ such that $S_n\nleq T$ and 
$S_n\ge\Omega_2(T)$, and also $T_n\nnsg\calf_n$ and 
$\Aut_{\calf_n}(T_n)\cong\autf(T)$. Thus $T_n=\Omega_k(T)$ for some $k\ge2$ 
since $\autf(T)$ acts irreducibly on $\Omega_1(T)$. By assumption, 
$\EE\calf\sminus\{T\}=\calh$: the set of subgroups of $S$ not contained in 
$T$ and isomorphic to $C_p\times C_p$. Hence no extraspecial subgroups of 
order $p^3$ can be essential in $\calf_n$, so 
$\EE{\calf_n}\subseteq\{T_n\}\cup\calh$. Also, 
$\EE{\calf_n}\supsetneqq\{T_n\}$ since $T_n\nnsg\calf_n$.

We now check that $\calf_n$ is reduced using the criteria in \cite[Lemma 
2.7]{indp2}.
\begin{itemize} 

\item By point (a) in that lemma, $O_p(\calf)=1$ if there are no nontrivial 
$\Aut_{\calf_n}(T_n)$-invariant subgroups of $Z(S_n)$, and this holds since 
$\Aut_{\calf_n}(T_n)$ acts irreducibly on $\Omega_1(T_n)>Z(S_n)$. 

\item By point (b) in the lemma, $O^p(\calf_n)=\calf_n$ since 
$[\Aut_{\calf_n}(T_n),T_n]=T_n$ --- again since the action is irreducible.

\item By point (c.iii) in the lemma and since $\EE{\calf_n}\sminus\{T_n\} 
\subseteq\calh$, $O^{p'}(\calf_n)=\calf_n$ if 
	\beqq \Aut_{\calf_n}(S_n)=\Gen{ 
	\Aut_{\calf_n}(S_n)\cap\mu^{-1}(\Delta_{-1}), 
	\Aut_{\calf_n}^{(T_n)}(S_n) }, \label{e:Op'2} \eeqq
where $\mu$ and $\Delta_{-1}$ are defined above and 
	\[ \Aut_{\calf_n}^{(T_n)}(S_n) = \bigl\{ \alpha\in \Aut_{\calf_n}(S_n) 
	\,\big|\, \alpha|_A\in O^{p'}(\Aut_{\calf_n}(T_n)) \bigr\}. \]
One checks in each case that 
$\mu(O^{p'}(\Aut_{\calf_n}(T_n)))=\{(1,s)\,|\,s\in(\Z/p)^\times\}<\Delta$. 
Since this subgroup together with 
$\Delta_{-1}=\{(r,r^{-1})\,|\,r\in(\Z/p)^\times\}$ generates $\Delta$, 
condition \eqref{e:Op'2} does hold, and hence $O^{p'}(\calf_n)=\calf_n$. 

\end{itemize}

Thus $\calf_n$ is reduced, and we are in the situation of case (d.iii) in 
\cite[Theorem 2.8]{indp2}. But by Table 2.2 in that paper, such a fusion 
system is not realizable, contradicting the original assumption that 
$\calf$ is sequentially realizable. 
\end{proof}

Note in particular the cases in Theorem \ref{t:index.p} where 
$\autf(T)\cong\ST{12}$ ($p=3$) or $\ST{31}$ ($p=5$). Each of these groups 
is realized by two different simple fusion systems: one which is realized 
by a $p$-compact group and by a linear torsion group, and another which is 
not sequentially realizable.

There are many simple fusion systems over discrete $p$-toral groups $S$ 
whose identity component $T$ has index $p$ in $S$ that are described by 
\cite[Theorem B]{indp3} and are not sequentially realizable. One can get an 
idea of the many possibilities for $\autf(T)$ by looking at Table 6.1 in 
\cite{indp3}, and at Examples 6.2 and 6.3 in the same paper.


\end{document}